\documentclass[12pt]{article}

\usepackage{amsmath}
\usepackage{amssymb}
\usepackage{amsthm}
\usepackage{amsfonts}
\usepackage{amscd}
\usepackage{graphicx}       % \includegraphics
\usepackage[T1]{fontenc}    % removes warning: some font shapes were not available
\usepackage{hyperref}
\usepackage{array}          % \setlength\extrarowheight{5pt}
\usepackage{enumitem}       % http://tex.stackexchange.com/questions/86054/how-to-remove-the-whitespace-before-itemize-enumerate
\usepackage{titling}        % \droptitle
\usepackage[numbers,sort&compress]{natbib}
\usepackage{caption}
\usepackage{tikz}
\usepackage{tikz-cd}
\usepackage[medium]{titlesec}
\usepackage{etoolbox}       % section number patch
\usepackage{geometry}
\theoremstyle{plain}
\newtheorem{theorem}[table]{Theorem}
\newtheorem{lemma}[table]{Lemma}
\newtheorem{proposition}[table]{Proposition}
\newtheorem{corollary}[table]{Corollary}
\newtheorem{theoremext}{Theorem}

\newtheoremstyle{dfinition}{2mm}{0mm}{}{}{\bfseries}{.}{ }{}
\theoremstyle{dfinition}
\newtheorem{definition}[table]{Definition}
\newtheorem{remark}[table]{Remark}
\newtheorem{notation}[table]{Notation}
\newtheorem{example}[table]{Example}
\hypersetup{colorlinks,urlcolor=black!50!green,citecolor=black!45!green,linkcolor=blue}
\makeatletter\patchcmd{\ttlh@hang}{\parindent\z@}{\parindent\z@\leavevmode}{}{}\patchcmd{\ttlh@hang}{\noindent}{}{}{}\makeatother % section number patch
\titlespacing*{\section}{0pt}{1mm}{1mm}
\titlespacing*{\subsection}{0pt}{1mm}{1mm}
\titlespacing*{\paragraph}{0pt}{1mm}{1mm}
\newcommand{\myspace}{\setlength{\abovedisplayskip}{2mm}\setlength{\belowdisplayskip}{2mm}}
\newenvironment{Mlist}{\begin{itemize}[topsep=0pt,itemsep=0pt]}{\end{itemize}}
\newenvironment{Menum}{\begin{enumerate}[topsep=0pt,itemsep=0pt]}{\end{enumerate}}
\newcommand{\PRP}[1]{Proposition~\ref{prp:#1}}
\newcommand{\LEM}[1]{Lemma~\ref{lem:#1}}
\newcommand{\THM}[1]{Theorem~\ref{thm:#1}}
\newcommand{\COR}[1]{Corollary~\ref{cor:#1}}
\newcommand{\DEF}[1]{Definition~\ref{def:#1}}
\newcommand{\RMK}[1]{Remark~\ref{rmk:#1}}
\newcommand{\EQN}[1]{(\ref{eqn:#1})}
\newcommand{\SEC}[1]{\textsection\ref{sec:#1}}
\newcommand{\FIG}[1]{Figure~\ref{fig:#1}}
\newcommand{\TAB}[1]{Table~\ref{tab:#1}}
\newcommand{\EXM}[1]{Example~\ref{exm:#1}}
\newcommand{\NTN}[1]{Notation~\ref{ntn:#1}}
\newcommand{\END}{\hfill $\vartriangleleft$}
\newcommand{\Mfig}[3]{\includegraphics[height=#1cm, width=#2cm]{img/#3}}
\newcommand{\CAP}[1]{{\bf #1.}}%{\textit{\textbf{Figure~\thefigure#1}}}
\newcommand{\Iff}{if and only if }
\newcommand{\st}{{such that }}
\newcommand{\Wlog}{without loss of generality }
\newcommand{\resp}{respectively}
\newcommand{\wrt}{with respect to }
\newcommand{\bas}[1]{\langle #1\rangle}
\newcommand{\set}[2]{\{ #1 ~|~ #2 \}}
\renewcommand{\c}{\colon}
\newcommand{\dto}{\dashrightarrow}
\newcommand{\hto}{\hookrightarrow}
\renewcommand{\P}{{\mathbb{P}}}
\renewcommand{\S}{{\mathbb{S}}}
\newcommand{\C}{{\mathbb{C}}}
\newcommand{\R}{{\mathbb{R}}}
\newcommand{\Q}{{\mathbb{Q}}}
\newcommand{\Z}{{\mathbb{Z}}}
\newcommand{\T}{{\mathbb{T}}}
\newcommand{\bD}{{\operatorname{\bf D}}}
\newcommand{\bL}{{\operatorname{\bf L}}}
\newcommand{\bM}{{\operatorname{\bf M}}}
\newcommand{\bS}{{\operatorname{\bf S}}}
\newcommand{\bT}{{\operatorname{\bf T}}}

\newcommand{\cA}{{\mathcal{A}}}
\newcommand{\cB}{{\mathcal{B}}}
\newcommand{\cC}{{\mathcal{C}}}
\newcommand{\cS}{{\mathcal{S}}}
\newcommand{\Ys}{{Y_\star}}
\newcommand{\Yc}{{Y_\circ}}
\newcommand{\fL}{{\mathfrak{L}}}
\newcommand{\df}[1]{\textit{#1}}
\newcommand{\aut}{\operatorname{Aut}}
\newcommand{\autc}{\aut_\circ}
\newcommand{\sym}{\operatorname{Sym}}
\newcommand{\lie}{\operatorname{Lie}}
\newcommand{\Msl}{\mathfrak{sl}}
\newcommand{\Msu}{\mathfrak{su}}
\newcommand{\Mso}{\mathfrak{so}}
\newcommand{\Mh}{\mathfrak{h}}
\newcommand{\Mg}{\mathfrak{g}}
\newcommand{\Mi}{{\mathfrak{i}}}
\newcommand{\p}{{\varepsilon}}
\newcommand{\hh}{h}
\newcommand{\kk}{k}
\renewcommand{\l}{{\ell}}
\newcommand{\SO}{\operatorname{SO}}
\newcommand{\PSO}{\operatorname{PSO}}
\newcommand{\PSX}{\operatorname{PSX}}
\newcommand{\PSE}{\operatorname{PSE}}
\newcommand{\PSL}{\operatorname{PSL}}
\newcommand{\PSA}{\operatorname{PSA}}
\newcommand{\X}{\widetilde{X}}
\title{M\"obius automorphisms of surfaces with many circles}
\author{Niels Lubbes}
\date{\today}

\begin{document}
\myspace

\maketitle

\begin{abstract}
We classify real two-dimensional orbits of conformal subgroups such that the orbits contain two circular arcs through a point. Such surfaces must be toric and admit a M\"obius automorphism group of dimension at least two. Our theorem generalizes the classical classification of Dupin cyclides.
\\[2mm]
\\[2mm]
\end{abstract}

\begingroup
\def\addvspace#1{\vspace{-2mm}}
\tableofcontents
\endgroup

\section{Introduction}
\label{sec:intro-defs}

Our main result is \THM{M}, which states a classification 
of algebraic surfaces that admit many conformal automorphisms. 
Let us consider some known examples in order to motivate and explain this result.
Suppose that a surface $Z\subseteq\R^n$ is a $G$-orbit for some conformal subgroup~$G$
and that $Z$ is not contained in a hyperplane or hypersphere.
It follows from Liouville's theorem that the conformal transformations of
$\R^n$ for $n\geq 3$ are exactly the M\"obius transformations.
If $n=3$, then $Z\subset \R^3$ is M\"obius equivalent to either 
a circular cone, a circular cylinder or a ring torus, and thus a \df{Dupin cyclide}. 
If $\dim G>2$, then either $Z=\R^2$ or $Z\subset\R^4$ 
is a stereographic projection of a Veronese surface in the unit-sphere $S^4$.
The considered examples of $G$-orbits contain at least two circles through each point
and motivate us to address the following problem 
about surfaces that are in a sense ``generalized Dupin cyclides'':

\textbf{Problem.} 
{\it Classify, up to M\"obius equivalence, real surfaces that 
are the orbit of a M\"obius subgroup
and that
contain at least two circles through a point.}

We see in \FIG{dP6} a linear projection of an orbit of a M\"obius subgroup in~$\R^5$
that contains three circles through each point. 
This surface is characterized by the third row of \THM{M}. 
\begin{figure}[!ht]
\centering
\begin{tabular}{@{}c@{}c@{}c@{}c@{}}
\Mfig{3}{3.4}{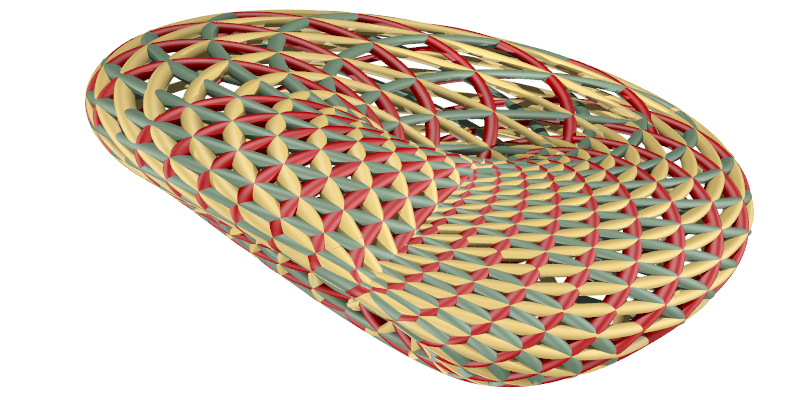} &
\Mfig{3}{3.4}{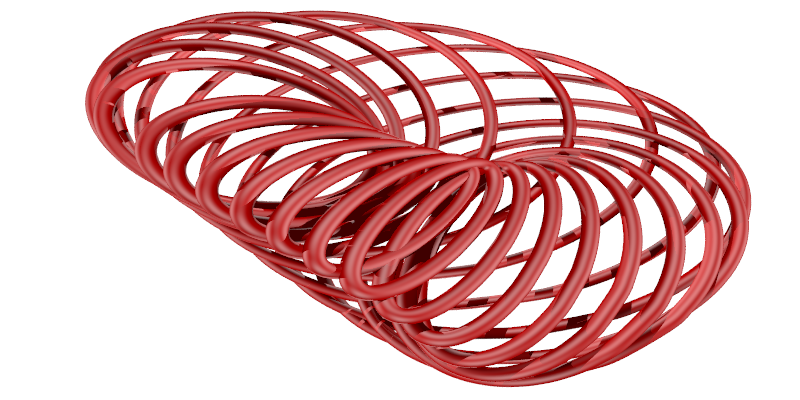}
\Mfig{3}{3.4}{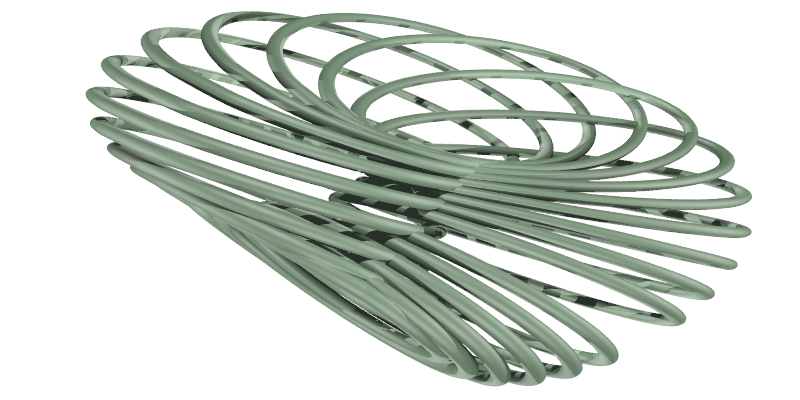}& 
\Mfig{3}{3.4}{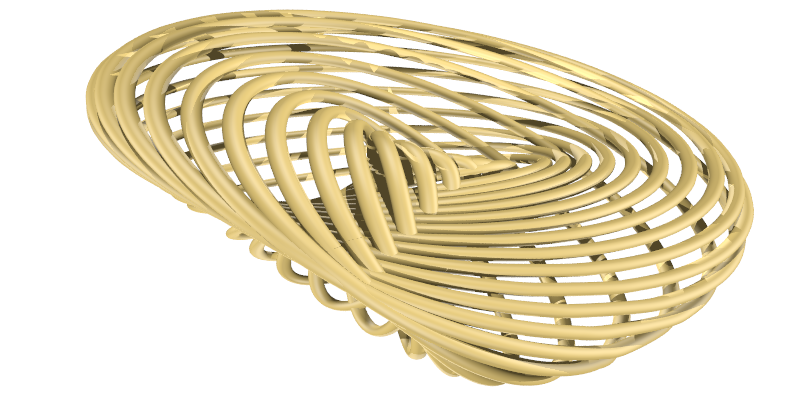} 
\end{tabular}
\caption{
A projection of a smooth surface of degree six in $\R^5$
that is an $\SO(2)\times \SO(2)$-orbit and that contains three circles through each point.
The family of M\"obius equivalence classes of such surfaces is two-dimensional.
%The circles in the surface are projected to ellipses in~$\R^3$.
}
\label{fig:dP6}
\end{figure}

There has been recent interest in the classification of surfaces 
that contain at least two circles through each point \cite{pot2,sko3}.
Surfaces that contain infinitely many circles through a general point
are classified in \cite{kol2} (see \THM{vero}). 
In this article we consider the M\"obius automorphism group as M\"obius invariant
and we use methods from \cite{sch10} (see \THM{D}), a classification from 
\cite{so4} (see \THM{sl2+sl2}) 
and results from \cite{nls-fam-circles} (see \THM{circle}).

% In order to clarify the terminology, we start by considering 
% a classical example, namely a circular cylinder $Z\subset \R^3$.
% Thus $Z$ is covered by a family of lines and circles
% and is translation and rotation invariant.
% These translations and rotations form a 2-dimensional group of M\"obius automorphisms.
% It follows from Liouville's theorem that the conformal transformations of
% $\R^n$ for $n\geq 3$ are exactly the M\"obius transformations.
% A \df{spindle torus} is a quartic surface in~$\R^3$ that is 
% M\"obius equivalent to a circular cyclinder and therefore admits a 2-dimensional group of M\"obius automorphisms. 
% In order to define M\"obius invariants of surfaces 
% we consider a \df{stereographic projection} $\pi\c S^n\dto \R^n$
% from some point in the \df{unit-sphere} $S^n\subset \R^n$.
% For example, the degree of the inverse stereographic projection $\pi^{-1}(Z)$ in $S^3$ is a M\"obius invariant of $Z$.

For our classification result we require
M\"obius invariants that capture geometric aspects at complex infinity. 
To uncover this hidden structure we define a \df{real variety} $X$ to be a complex variety together with
an antiholomorphic involution $\sigma\c X\to X$, 
which is called its \df{real structure} \citep[Section~I.1]{sil1}.
We denote the real points of $X$ by $X(\R):=\set{p\in X}{\sigma(p)=p}$.
Real varieties can always be defined by polynomials with real coefficients 
\citep[Section~6.1]{serre}.
Curves, surfaces and projective spaces $\P^n$ are real algebraic varieties 
and maps between such varieties are compatible with the real structures 
unless explicitly stated otherwise.
Instead of $\R^n$, it is more natural to use the \df{M\"obius quadric} for our space:
\[
\S^n:=\set{x\in\P^{n+1}}{-x_0^2+x_1^2+\ldots+x_{n+1}^2=0}, 
\]
where $\sigma\c\P^{n+1}\to\P^{n+1}$ sends $x$ to $(\overline{x_0}:\ldots:\overline{x_{n+1}})$.
The \df{M\"obius transformations} of $\S^n$ are the biregular automorphisms $\aut(\S^n)$
and they are linear so that $\aut(\S^n)\subset \aut(\P^{n+1})$.
We denote a \df{stereographic projection} from the \df{unit-sphere} $S^n\subset\R^{n+1}$
by $\pi\c S^n\dto \R^n$.
Notice that $\S^n(\R)\cong S^n$ and that the inverse stereographic projection $\pi^{-1}(\R^n)\subset \S^n$
is an isomorphic copy of $\R^n$ \st the M\"obius transformations of $\S^n$
restrict to M\"obius transformations of $S^n$ and $\pi^{-1}(\R^n)$.
In particular, the M\"obius transformations that preserve the projection center of $\pi$,
correspond to the \df{Euclidean similarities} of~$\R^n$.

A conic $C\subset \S^n$ is called a \df{circle} if $C(\R)$ defines a circle in $S^n\subset\R^{n+1}$.
We call a surface in $\S^n$ \df{$\lambda$-circled}
if it contains exactly $\lambda$ circles through a general point.
A \df{celestial surface} is a $\lambda$-circled surface $X\subseteq\S^n$ 
\st $\lambda\geq 2$
and \st $X$ is not contained in a hyperplane section of~$\S^n$.
If in addition $X$ is of degree~$d$, then its \df{celestial type} is defined as 
$\bT(X):=(\lambda,d,n)$.
If the biregular automorphism group $\aut(X)$ is a Lie group,
then its identity component is denoted by $\autc(X)$
and the \df{M\"obius automorphism group} of $X$ is defined as
$\bM(X):=\autc(X)\cap \autc(\S^n)$.
We denote the \df{singular locus}
of $X$ by $\bS(X)$. 
A complex node, real node, complex tacnode and real tacnode
is denoted by $A_1$, $\underline{A_1}$, $A_3$ and $\underline{A_3}$, \resp,
and the union of such nodes is written as a formal sum.
We write $\bS(X)=\emptyset$, if $X$ is smooth.

The \df{M\"obius moduli dimension} $\bD(X)$ is defined as 
the dimension of the space of M\"obius equivalence classes of 
celestial surfaces $Y\subset\S^n$ \st 
$\bT(Y)=\bT(X)$, $\bS(Y)\cong\bS(X)$ as algebraic sets 
and $\bM(Y)\cong\bM(X)$ as groups.

We use the following notation for subgroups of $\autc(\P^1)$. 
Let the real structure $\sigma\c\P^1\to\P^1$ be defined by $(x:y)\mapsto(\overline{x}:\overline{y})$
so that $\autc(\P^1)\cong\PSL(2,\R)$.
If $p,q,r\in\P^1$
\st $p\neq\sigma(p)$, $q=\sigma(q)$, $r=\sigma(r)$ and $q\neq r$, then we denote
\[
\begin{array}{l@{~}c@{~}l}
\PSO(2)&:=&\set{\varphi\in\autc(\P^1)}{\varphi(p)=p,~ \varphi(\sigma(p))=\sigma(p)} ,\\
\PSX(1)&:=&\set{\varphi\in\autc(\P^1)}{\varphi(q)=q,~ \varphi(r)=r}                 ,\\
\PSE(1)&:=&\set{\varphi\in\autc(\P^1)}{\varphi(x:y)=(x+\alpha\,y: y),~ \alpha\in \R}, \text{ and }\\
\PSA(1)&:=&\set{\varphi\in \autc(\P^1)}{\varphi(r)=r}.
\end{array}.
\]
Notice that $\varphi$ in $\PSO(2)$ or $\PSX(1)$
maps $(x:y)$ up to choice of coordinates to 
$(\cos(\alpha)\,x-\sin(\alpha)\,y:\sin(\alpha)\,x+\cos(\alpha)\,y)$
and $(\alpha\,x:y)$, \resp, for some $\alpha\in\R\setminus\{0\}$.
The elements in $\PSA(1)$ are combinations of elements in $\PSX(1)$ and $\PSE(1)$
(see also \RMK{naming}).

\begin{theorem}
\textrm{\bf(M\"obius automorphisms of celestial surfaces)}
\label{thm:M}
\\
If $X\subseteq\S^n$ is a celestial surface \st $\dim\bM(X)\geq 2$,
then 
$X$ is toric and its M\"obius invariants
$\bT(X)$, $\bS(X)$, $\bM(X)$ and $\bD(X)$
are characterized by a row in the following table (see \DEF{names} for the names):
\begin{center}
%\vspace{-5mm}
\footnotesize
\setlength\extrarowheight{8pt}
\begin{tabular}{ccccl}
%\hline
$\bT(X)$       & $\bS(X)$                                  & $\bM(X)$                  & $\bD(X)$  & name                  \\\hline
$(2,8,7)$      & $\emptyset$                               & $\PSO(2)\times \PSO(2)$   & $3$       & double Segre surface  \\
$(2,8,5)$      & $\emptyset$                               & $\PSO(2)\times \PSO(2)$   & $2$       & projected dS          \\
$(3,6,5)$      & $\emptyset$                               & $\PSO(2)\times \PSO(2)$   & $2$       & dP6 (see \FIG{dP6})   \\
$(\infty,4,4)$ & $\emptyset$                               & $\PSO(3)$                 & $0$       & Veronese surface      \\
$(4,4,3)$      & $A_1+A_1+A_1+A_1$                         & $\PSO(2)\times \PSO(2)$   & $1$       & ring cyclide          \\     
$(2,4,3)$      & $\underline{A_1}+\underline{A_1}+A_1+A_1$ & $\PSO(2)\times \PSX(1)$   & $0$       & spindle cyclide       \\ 
$(2,4,3)$      & $\underline{A_3}+A_1+A_1$                 & $\PSO(2)\times \PSE(1)$   & $0$       & horn cyclide          \\ 
$(\infty,2,2)$ & $\emptyset$                               & $\PSO(3,1)$               & $0$       & 2-sphere
\end{tabular}
\end{center}
Moreover, if $\bT(X)\notin\{(2,8,7),~(2,8,5),~(\infty,4,4)\}$, then $\bM(X)=\autc(X)$ and if $\bD(X)=0$, then 
$X$ is unique up to M\"obius equivalence.
\end{theorem}

If we replace $\S^n$ with $S^n\cong\S^n(\R)$,
then \THM{M} holds if we replace $\PSO(3,1)$ by $\SO(3)$
and remove the remaining P's in the $\bM(X)$ column.
The case $\bT(X)=(\infty,4,4)$ was already known
and is revisited in \SEC{vero} (see also \citep[Theorem~23]{kol2} and \citep[Section 2.4.3]{olm1}).

A \df{smooth model} of a surface~$X\subset\P^{n+1}$ is a birational morphism $\X\to X$ 
from a nonsingular surface~$\X$, such that this morphism does not contract $(-1)$-curves.
If $\lambda<\infty$, then the smooth models of the $\lambda$-circled
surfaces in \THM{M} are isomorphic to $\S^1\times\S^1$ blownup up in either 0, 2 or 4
complex conjugate points (see \NTN{S1S1}).
A smooth model of a Veronese surface is isomorphic to~$\P^2$ 
\st $\sigma\c\P^2\to\P^2$ sends $x$ to $(\overline{x_0}:\overline{x_1}:\overline{x_2})$ 
(see \LEM{invo-P1P1}c).

Instead of $\S^n$ one could also consider hyperquadrics of different signature. 
Although we do not pursue this, we cannot resist to mention the following result, 
which will come almost for free during our investigations:

\begin{corollary}
\label{cor:iqf}
If $Q\subset\P^8$ is a quadric hypersurface of signature $(4,5)$
or $(3,6)$, 
then there exists a unique double Segre surface $X\subset Q$
\st $\autc(X)\subset \autc(Q)$ and $X$ is isomorphic to $\S^1\times\S^1$ and $\S^2$, \resp.
\end{corollary}

Our methods are constructive and allow for 
explicit coordinate description of the moduli space of the celestial surfaces.
See \citep[\texttt{moebius\_aut}]{maut} for an implementation using \citep[Sage]{sage}.

\begin{definition}
{\bf (names of surfaces)}
\label{def:names}
\\
A surface $X\subset \S^3$ is called a \df{spindle cyclide}, \df{horn cyclide} or \df{ring cyclide}
if there exists a stereographic projection $\pi\c S^3\dto \R^3$ 
\st $\pi(X(\R))\subset \R^3$ is a circular cone, circular cylinder and ring torus, \resp.
We call $X\subset\P^5$ a \df{Veronese surface} 
if there exists a biregular isomorphism $\P^2\to X$ 
whose components form a basis of the vector space of degree 2 forms.
We call $X\subset\P^8$ a \df{double Segre surface} (or \df{dS} for short)
if there exists a biregular isomorphism $\P^1\times\P^1\to X$ 
whose components form a basis of the vector space of bidegree (2,2) forms.
A \df{projected dS} is a surface $X$ that is a degree preserving 
linear projection of a dS.
We call $X\subset\P^6$ a \df{sextic del Pezzo surface} (or \df{dP6} for short)
if $X$ is an anticanonical model of $\P^1\times\P^1$ blown up in two general complex points.
\END
\end{definition}

\newpage
\begin{remark}
{\bf (overview)}
\\
Suppose that $X$ is a celestial surface of type $\bT(X)=(\lambda,d,n)$
\st $\dim \bM(X)\geq 2$ and $\lambda<\infty$.

In \SEC{L} we classify $X$ under the additional assumption that $X$ is toric.

In \SEC{segre} we give coordinates for a double Segre surface $\Ys\subset\P^8$
and we investigate actions of real structures on $\Ys$.
This will be needed for finding quadratic forms of signature $(1,n+1)$
in the ideal of~$\Ys$.

We establish in \SEC{blowup} that $X$ must be toric.
Moreover, there exists a birational linear projection $\rho\c\Ys\dto X$
and $\bM(X)$ is isomorphic to a subgroup of $\autc(\Ys)$ that leaves the center of $\rho$ invariant.
We characterize the possible configurations for the center of $\rho$ in $\Ys$
and for each such configuration we restrict the possible values for $\bT(X)$ and $\bM(X)$.

In \SEC{pair} we encode, up to M\"obius equivalence, $X$
as an $\bM(X)$-invariant quadratic form in the ideal of $\Ys$.
From the Lie algebra of $\bM(X)$, we recover the 
subspace of $\bM(X)$-invariant quadratic forms and each invariant form
of signature \mbox{$(1,n+1)$}
in this space encodes a possible M\"obius equivalence class for $X$.

In \SEC{lie} we show how toric real structures 
act on the Lie algebra of $\autc(\Ys)$ and we recall 
the classification of Lie algebras of complex subgroups of $\autc(\Ys)$.

In \SEC{iqf-segre} we make a case distinction on the established
configurations for the center of~$\rho$ 
and Lie algebras of $\bM(X)$.
If~$X$ is not a spindle or horn cyclide, then 
$\bM(X)\cong\SO(2)\times\SO(2)$ and we obtain   
coordinates for the $\bD(X)+1$ generators of all 
$\bM(X)$-invariant quadratic forms of signature $(1,n+1)$
in the ideal of $\Ys$. 
This enables us to conclude \SEC{iqf-segre} with a proof for \THM{M} and \COR{iqf}.

Finally, we present in \SEC{vero} an alternative proof 
for the known $\lambda=\infty$ case by applying
the same methods as before, but with $\Ys$ replaced with the Veronese surface $\Yc\subset\P^5$.
\END
\end{remark}

\section{Toric celestial surfaces}
\label{sec:L}

In this section we classify toric celestial surfaces and their real structures.

Suppose that $X\subset\P^n$ is a surface that is not contained in a hyperplane section.
The \df{linear normalization} $X_N\subset\P^m$ of $X$
is defined as the image of its smooth model $\X$ via the map associated to 
the complete linear series of hyperplane sections of $X$.
Thus $m\geq n$, $X$ is a linear projection of $X_N$ and $X_N$ is unique up to $\aut(\P^m)$.

Let $\T^1:=(\C^*,1)$ denote the \df{algebraic torus}.
Recall that $X$ is \df{toric} if 
there exists an embedding $i\c\T^2\hto X$
\st $i(\T^2)$ is dense in $X$ and \st the action of $\T^2$ on itself
extends to an action on $X$.

If $X$ is a toric surface, then
there exists, up to projective equivalence,
a monomial parametrization $\xi\c\T^2\to X_N$.
The \df{lattice polygon of $X$} is defined as 
the convex hull of the points in the lattice $\Z^2\subset\R^2$,
whose coordinates are defined by the exponents of the components of $\xi$. 
The antiholomorphic involution $\sigma\c X\to X$
induces an involution 
$\sigma\c\T^2\to\T^2$.
Consequently, $\sigma$ induces a unimodular involution $\Z^2\to\Z^2$ that
leaves the lattice polygon of~$X$ invariant.

\begin{notation}
\label{ntn:nota}
By abuse of notation we denote 
involutions on algebraic structures, that 
correspond functorially with the real structure $\sigma\c X\to X$, by $\sigma$ as well.
\END 
\end{notation}

A lattice projection 
$\Z^2\subset\R^2\to\Z^1\subset\R^1$
induces a toric map $X_N\to\P^1$. 
We call a family of curves on $X_N$ \df{toric} if the family can be defined 
by the fibers of a toric map.
A family of circles on $X$ is called \df{toric} if it corresponds to a toric 
family on $X_N$ via a linear projection $X_N\to X$.
The toric families of circles 
that cover a toric surface~$X$ that is not covered by complex lines,
correspond to the projections of the lattice polygon of~$X$
to a line segment that is of minimal width among all such projections
\citep[Proposition~31]{nls1}.

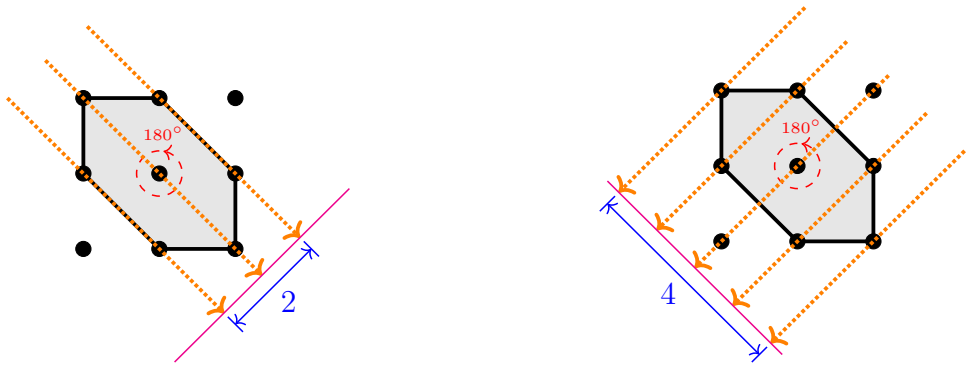
\begin{figure}[!ht]
\centering
\begin{tikzpicture}
%\draw[help lines] (-1,1) grid (1,-1);

\draw[draw=black, fill=black!10, line width=0.5mm] (-1.0,+1.0) to (+0.0,+1.0) 
                                                               to (+1.0,+0.0) 
                                                               to (+1.0,-1.0) 
                                                               to (+0.0,-1.0)
                                                               to (-1.0,+0.0)
                                                               to (-1.0,+1.0);

\draw[draw=black, fill=black] (-1.0,+1.0) circle [radius=0.1];
\draw[draw=black, fill=black] (+0.0,+1.0) circle [radius=0.1];
\draw[draw=black, fill=black] (+1.0,+1.0) circle [radius=0.1];
\draw[draw=black, fill=black] (-1.0,+0.0) circle [radius=0.1];
\draw[draw=black, fill=black] (+0.0,+0.0) circle [radius=0.1];
\draw[draw=black, fill=black] (+1.0,+0.0) circle [radius=0.1];
\draw[draw=black, fill=black] (-1.0,-1.0) circle [radius=0.1];
\draw[draw=black, fill=black] (+0.0,-1.0) circle [radius=0.1];
\draw[draw=black, fill=black] (+1.0,-1.0) circle [radius=0.1];

\draw[draw=red, dashed, line width=0.2mm, ->] (0,0.3) arc (90:360+80:0.3) node[red,above] {\tiny $180^\circ$};

\draw[draw=orange, densely dotted, line width=0.5mm, ->] (-2.5+0.5,+1.5-0.5) to (+0.5+0.35,-1.5-0.35);
\draw[draw=orange, densely dotted, line width=0.5mm, ->] (-1.5+0.0,+1.5-0.0) to (+1.5-0.15,-1.5+0.15);
\draw[draw=orange, densely dotted, line width=0.5mm, ->] (-0.5-0.45,+1.5+0.45) to (+2.5-0.65,-1.5+0.65);

\draw[draw=magenta, line width=0.2mm] (+0.2,-2.5) to (+2.5,-0.2);

\draw[draw=blue, line width=0.2mm] (+0.0+0.9,+0.0-1.9) 
                                to (+0.2+0.9,-0.2-1.9);

\draw[draw=blue, line width=0.2mm] (+0.0+1.9,+0.0-0.9) 
                                to (+0.2+1.9,-0.2-0.9);

\draw[draw=blue, line width=0.2mm, <->] (+0.0+1.0,+0.0-2.0) 
                                     to (+0.0+2.0,+0.0-1.0);                                

\node[blue] at (+1.7,-1.7) {$2$};                                     
\end{tikzpicture}
\hspace{3cm}
\begin{tikzpicture}
%\draw[help lines] (-1,1) grid (1,-1);

\draw[draw=black, fill=black!10, line width=0.5mm] (-1.0,+1.0) to (+0.0,+1.0) 
                                                               to (+1.0,+0.0) 
                                                               to (+1.0,-1.0) 
                                                               to (+0.0,-1.0)
                                                               to (-1.0,+0.0)
                                                               to (-1.0,+1.0);

\draw[draw=black, fill=black] (-1.0,+1.0) circle [radius=0.1];
\draw[draw=black, fill=black] (+0.0,+1.0) circle [radius=0.1];
\draw[draw=black, fill=black] (+1.0,+1.0) circle [radius=0.1];
\draw[draw=black, fill=black] (-1.0,+0.0) circle [radius=0.1];
\draw[draw=black, fill=black] (+0.0,+0.0) circle [radius=0.1];
\draw[draw=black, fill=black] (+1.0,+0.0) circle [radius=0.1];
\draw[draw=black, fill=black] (-1.0,-1.0) circle [radius=0.1];
\draw[draw=black, fill=black] (+0.0,-1.0) circle [radius=0.1];
\draw[draw=black, fill=black] (+1.0,-1.0) circle [radius=0.1];

\draw[draw=red, dashed, line width=0.2mm, ->] (0,0.3) arc (90:360+80:0.3) node[red,above] {\tiny $180^\circ$};

\draw[draw=orange, densely dotted, line width=0.5mm, ->] (-0.5+0.60,+1.5+0.60) to (-3.5+1.15,-1.5+1.15);
\draw[draw=orange, densely dotted, line width=0.5mm, ->] (+0.5+0.15,+1.5+0.15) to (-2.5+0.65,-1.5+0.65);
\draw[draw=orange, densely dotted, line width=0.5mm, ->] (+1.5-0.30,+1.5-0.30) to (-1.5+0.15,-1.5+0.15);
\draw[draw=orange, densely dotted, line width=0.5mm, ->] (+2.5-0.80,+1.5-0.80) to (-0.5-0.35,-1.5-0.35);
\draw[draw=orange, densely dotted, line width=0.5mm, ->] (+3.5-1.30,+1.5-1.30) to (+0.5-0.85,-1.5-0.85);

\draw[draw=magenta, line width=0.2mm] (-0.2,-2.5) to (-2.5,-0.2);

\draw[draw=blue, line width=0.2mm] (-0.0-0.4,+0.0-2.4)
                                to (-0.2-0.4,-0.2-2.4);

\draw[draw=blue, line width=0.2mm] (-0.0-2.4,+0.0-0.4) 
                                to (-0.2-2.4,-0.2-0.4);

\draw[draw=blue, line width=0.2mm, <->] (-0.2-0.3,-0.2-2.3) 
                                     to (-0.2-2.3,-0.2-0.3);                                

\node[blue] at (-1.7,-1.7) {$4$};                                     
\end{tikzpicture}
\caption{Width of lattice polygon along directions $\searrow$ and $\swarrow$.}
\label{fig:width}
\vspace{5mm}
\end{figure}

In \FIG{width} we see two examples of lattice projections of a 
lattice polygon. The width of the polygon in the 
$\swarrow$ direction is $4$.
The lattice polygon
attains its minimal width of $2$ in the directions 
$\rightarrow$, $\downarrow$ and $\searrow$. 
Notice that the lines through the origin, along these three directions,
are invariant under
lattice involution defined by $180^\circ$ rotation around the central lattice point.

The \df{lattice type} $\bL(X)$ of a toric surface $X$
consists of the following data
\begin{Menum}
\item The lattice polygon $\Lambda\subset\R^2$ of $X$.
\item The unimodular involution $\Z^2\to\Z^2$ that is induced by the real structure $\sigma\c X\to X$.
\item The lattice projections that correspond to toric families of circles. We will represent such projections by arrows ($\downarrow$, $\rightarrow$, $\searrow$, $\swarrow$) 
pointing in the corresponding direction. 
\end{Menum}

Lattice types $\bL(X)$ and $\bL(X')$ are equivalent if 
there exists a unimodular isomorphism between
their lattice polygons that is compatible with the 
unimodular involution.
Data 3 is uniquely determined by data 1 and data 2.
The unimodular involutions $\Z^2\to \Z^2$,
defined by
$(x,y)\mapsto(x,y)$, $(x,y)\mapsto(-x,y)$, $(x,y)\mapsto(-x,-y)$ and $(x,y)\mapsto(y,x)$,
are represented by their symmetry axes in the lattice polygons.

\begin{proposition}
\textrm{\bf(classification of toric celestial surfaces)}
\label{prp:L}
\\
If $X\subseteq\S^n$ is a toric surface that is covered 
by at least two toric families of circles, 
then its lattice type $\bL(X)$, together with $\bT(X)$ and the name of~$X$,
is up to equivalence characterized by one of the eight cases in \TAB{L}. 
\end{proposition}

\begin{corollary}
\textrm{\bf(classification of toric celestial surfaces)}
\label{cor:L}
\begin{Mlist}
\item[\bf a)]
The antiholomorphic involutions of the double Segre surface --- that act as unimodular 
involutions as in \TAB{L}a, \TAB{L2}a and \TAB{L2}b --- are 
inner automorphic via $\aut(\P^1\times\P^1)$.

\item[\bf b)]
The antiholomorphic involutions of the Veronese surface --- that act as unimodular 
involutions as in \TAB{L}d and \TAB{L2}c --- are 
inner automorphic via $\aut(\P^2)$.
% 
% \item[\bf c)]
% If $X\subset \S^n$ is a celestial surface \st $n\geq 4$, then $X_N$ is toric.
\end{Mlist}
\end{corollary}

\begin{table}
\caption{See \PRP{L}. 
Lattice types of toric celestial surfaces up to equivalence,
together with the corresponding name and possible celestial types. 
For the celestial types we have $3\leq n\leq 7$ and $4\leq m\leq 5$.
The directions correspond to the toric families of circles.}
\label{tab:L}
%\begin{tabular}{c@{\qquad}c@{\qquad}c@{\qquad}c}
\begin{tabular}{@{}c@{\hspace{1.3cm}}c@{\hspace{1.3cm}}c@{\hspace{1.3cm}}c@{}}
\begin{tikzpicture}[scale=1]
%\draw[help lines] (-1,1) grid (1,-1);

\draw[draw=black, fill=black!10, line width=0.5mm] (-1.0,+1.0) to (+1.0,+1.0) to (+1.0,-1.0) to (-1.0,-1.0) to (-1.0,+1.0);

\draw[draw=black, fill=black] (-1.0,+1.0) circle [radius=0.1];
\draw[draw=black, fill=black] (+0.0,+1.0) circle [radius=0.1];
\draw[draw=black, fill=black] (+1.0,+1.0) circle [radius=0.1];
\draw[draw=black, fill=black] (-1.0,+0.0) circle [radius=0.1];
\draw[draw=black, fill=black] (+0.0,+0.0) circle [radius=0.1];
\draw[draw=black, fill=black] (+1.0,+0.0) circle [radius=0.1];
\draw[draw=black, fill=black] (-1.0,-1.0) circle [radius=0.1];
\draw[draw=black, fill=black] (+0.0,-1.0) circle [radius=0.1];
\draw[draw=black, fill=black] (+1.0,-1.0) circle [radius=0.1];

%\draw[draw=red, dashed, line width=0.5mm, ->] (0,0.3) arc (90:360+80:0.3) node[red, above] {\small $180^\circ$};
\node[red] at (0,0.5) {id};

\node at (0,2) {dS};
\node at (0,1.5) {$(2,8,n)$};
\node at (0,-1.5) {$\rightarrow\downarrow$};
\node at (0,-2.0) {\bf\large a.};
\end{tikzpicture}
&
\begin{tikzpicture}[scale=1]
%\draw[help lines] (-1,1) grid (1,-1);

\draw[draw=black, fill=black!10, line width=0.5mm] (-1.0,+1.0) to (+0.0,+1.0) 
                                                               to (+1.0,+0.0) 
                                                               to (+1.0,-1.0) 
                                                               to (+0.0,-1.0)
                                                               to (-1.0,+0.0)
                                                               to (-1.0,+1.0);

\draw[draw=black, fill=black] (-1.0,+1.0) circle [radius=0.1];
\draw[draw=black, fill=black] (+0.0,+1.0) circle [radius=0.1];
\draw[draw=black, fill=black] (+1.0,+1.0) circle [radius=0.1];
\draw[draw=black, fill=black] (-1.0,+0.0) circle [radius=0.1];
\draw[draw=black, fill=black] (+0.0,+0.0) circle [radius=0.1];
\draw[draw=black, fill=black] (+1.0,+0.0) circle [radius=0.1];
\draw[draw=black, fill=black] (-1.0,-1.0) circle [radius=0.1];
\draw[draw=black, fill=black] (+0.0,-1.0) circle [radius=0.1];
\draw[draw=black, fill=black] (+1.0,-1.0) circle [radius=0.1];

\draw[draw=red, dashed, line width=0.5mm, ->] (0,0.3) arc (90:360+80:0.3) node[red, above left] {\small $180^\circ$};

\node at (0,2) {dP6};
\node at (0,1.5) {$(3,6,m)$};
\node at (0,-1.5) {$\rightarrow\downarrow\searrow$};
\node at (0,-2.0) {\bf\large b.};
\end{tikzpicture}
&
\begin{tikzpicture}[scale=1]
%\draw[help lines] (-1,1) grid (1,-1);

\draw[draw=black, fill=black!10, line width=0.5mm] (-1.0,+0.0) to (+0.0,+1.0) 
                                                               to (+1.0,+0.0) 
                                                               to (+1.0,-1.0) 
                                                               to (-1.0,-1.0)
                                                               to (-1.0,+0.0)
                                                               to (-1.0,+0.0);

\draw[draw=black, fill=black] (-1.0,+1.0) circle [radius=0.1];
\draw[draw=black, fill=black] (+0.0,+1.0) circle [radius=0.1];
\draw[draw=black, fill=black] (+1.0,+1.0) circle [radius=0.1];
\draw[draw=black, fill=black] (-1.0,+0.0) circle [radius=0.1];
\draw[draw=black, fill=black] (+0.0,+0.0) circle [radius=0.1];
\draw[draw=black, fill=black] (+1.0,+0.0) circle [radius=0.1];
\draw[draw=black, fill=black] (-1.0,-1.0) circle [radius=0.1];
\draw[draw=black, fill=black] (+0.0,-1.0) circle [radius=0.1];
\draw[draw=black, fill=black] (+1.0,-1.0) circle [radius=0.1];

%\draw[draw=red, line width=0.5mm, ->] (0,0.3) arc (90:360+80:0.3) node[red, above left] {\small $180^\circ$};
\draw[draw=red, dashed, line width=0.5mm] (0,1) to (0,-1);

\node at (0,2) {weak dP6};
\node at (0,1.5) {$(2,6,m)$};
\node at (0,-1.5) {$\rightarrow\downarrow$};
\node at (0,-2.0) {\bf\large c.};
\end{tikzpicture}
&
\begin{tikzpicture}[scale=1]
%\draw[help lines] (-1,1) grid (1,-1);

\draw[draw=black, fill=black!10, line width=0.5mm] (-1.0,+1.0) to (+1.0,-1.0) 
                                                               to (-1.0,-1.0) 
                                                               to (-1.0,+1.0); 

\draw[draw=black, fill=black] (-1.0,+1.0) circle [radius=0.1];
\draw[draw=black, fill=black] (+0.0,+1.0) circle [radius=0.1];
\draw[draw=black, fill=black] (+1.0,+1.0) circle [radius=0.1];
\draw[draw=black, fill=black] (-1.0,+0.0) circle [radius=0.1];
\draw[draw=black, fill=black] (+0.0,+0.0) circle [radius=0.1];
\draw[draw=black, fill=black] (+1.0,+0.0) circle [radius=0.1];
\draw[draw=black, fill=black] (-1.0,-1.0) circle [radius=0.1];
\draw[draw=black, fill=black] (+0.0,-1.0) circle [radius=0.1];
\draw[draw=black, fill=black] (+1.0,-1.0) circle [radius=0.1];

%\draw[draw=red, line width=0.5mm, ->] (0,0.3) arc (90:360+80:0.3) node[red, above left] {\small $180^\circ$};
%\draw[draw=red, dashed, line width=0.5mm] (0,1) to (0,-1);
\node[red] at (-0.5,-0.5) {id};

\node at (0,2) {Veronese};
\node at (0,1.5) {$(\infty,4,4)$};
\node at (0,-1.5) {$\rightarrow\downarrow\searrow$};
\node at (0,-2.0) {\bf\large d.};
\end{tikzpicture}
\\[2mm]
\begin{tikzpicture}[scale=1]
%\draw[help lines] (-1,1) grid (1,-1);

\draw[draw=black, fill=black!10, line width=0.5mm] (-1.0,+0.0) to (+0.0,+1.0) 
                                                               to (+1.0,+0.0) 
                                                               to (+0.0,-1.0)
                                                               to (-1.0,+0.0); 

\draw[draw=black, fill=black] (-1.0,+1.0) circle [radius=0.1];
\draw[draw=black, fill=black] (+0.0,+1.0) circle [radius=0.1];
\draw[draw=black, fill=black] (+1.0,+1.0) circle [radius=0.1];
\draw[draw=black, fill=black] (-1.0,+0.0) circle [radius=0.1];
\draw[draw=black, fill=black] (+0.0,+0.0) circle [radius=0.1];
\draw[draw=black, fill=black] (+1.0,+0.0) circle [radius=0.1];
\draw[draw=black, fill=black] (-1.0,-1.0) circle [radius=0.1];
\draw[draw=black, fill=black] (+0.0,-1.0) circle [radius=0.1];
\draw[draw=black, fill=black] (+1.0,-1.0) circle [radius=0.1];

\draw[draw=red, dashed, line width=0.5mm, ->] (0,0.3) arc (90:360+80:0.3) node[red, above] {\tiny $180^\circ$};
%\draw[draw=red, dashed, line width=0.5mm] (0,1) to (0,-1);
%\node[red] at (-0.5,-0.5) {id};

\node at (0,2) {ring cyclide};
\node at (0,1.5) {$(4,4,3)$};
\node at (0,-1.5) {$\rightarrow\downarrow\swarrow\searrow$};
\node at (0,-2.0) {\bf\large e.};
\end{tikzpicture}
&
\begin{tikzpicture}[scale=1]
%\draw[help lines] (-1,1) grid (1,-1);

\draw[draw=black, fill=black!10, line width=0.5mm] (-1.0,+0.0) to (+0.0,+1.0) 
                                                               to (+1.0,+0.0) 
                                                               to (+0.0,-1.0)
                                                               to (-1.0,+0.0); 

\draw[draw=black, fill=black] (-1.0,+1.0) circle [radius=0.1];
\draw[draw=black, fill=black] (+0.0,+1.0) circle [radius=0.1];
\draw[draw=black, fill=black] (+1.0,+1.0) circle [radius=0.1];
\draw[draw=black, fill=black] (-1.0,+0.0) circle [radius=0.1];
\draw[draw=black, fill=black] (+0.0,+0.0) circle [radius=0.1];
\draw[draw=black, fill=black] (+1.0,+0.0) circle [radius=0.1];
\draw[draw=black, fill=black] (-1.0,-1.0) circle [radius=0.1];
\draw[draw=black, fill=black] (+0.0,-1.0) circle [radius=0.1];
\draw[draw=black, fill=black] (+1.0,-1.0) circle [radius=0.1];

%\draw[draw=red, dashed, line width=0.5mm, ->] (0,0.3) arc (90:360+80:0.3) node[red, above] {\small $180^\circ$};
\draw[draw=red, dashed, line width=0.5mm] (0,1) to (0,-1);
%\node[red] at (-0.5,-0.5) {id};

\node at (0,2) {spindle cyclide};
\node at (0,1.5) {$(2,4,3)$};
\node at (0,-1.5) {$\rightarrow\downarrow$};
\node at (0,-2.0) {\bf\large f.};
\end{tikzpicture}
&
\begin{tikzpicture}[scale=1]
%\draw[help lines] (-1,1) grid (1,-1);

\draw[draw=black, fill=black!10, line width=0.5mm] (-1.0,-1.0) to (+0.0,+1.0) 
                                                               to (+1.0,-1.0)                                                                
                                                               to (-1.0,-1.0); 

\draw[draw=black, fill=black] (-1.0,+1.0) circle [radius=0.1];
\draw[draw=black, fill=black] (+0.0,+1.0) circle [radius=0.1];
\draw[draw=black, fill=black] (+1.0,+1.0) circle [radius=0.1];
\draw[draw=black, fill=black] (-1.0,+0.0) circle [radius=0.1];
\draw[draw=black, fill=black] (+0.0,+0.0) circle [radius=0.1];
\draw[draw=black, fill=black] (+1.0,+0.0) circle [radius=0.1];
\draw[draw=black, fill=black] (-1.0,-1.0) circle [radius=0.1];
\draw[draw=black, fill=black] (+0.0,-1.0) circle [radius=0.1];
\draw[draw=black, fill=black] (+1.0,-1.0) circle [radius=0.1];

%\draw[draw=red, dashed, line width=0.5mm, ->] (0,0.3) arc (90:360+80:0.3) node[red, above] {\small $180^\circ$};
\draw[draw=red, dashed, line width=0.5mm] (0,1) to (0,-1);
%\node[red] at (-0.5,-0.5) {id};

\node at (0,2) {horn cyclide};
\node at (0,1.5) {$(2,4,3)$};
\node at (0,-1.5) {$\rightarrow\downarrow$};
\node at (0,-2.0) {\bf\large g.};
\end{tikzpicture}
&
\begin{tikzpicture}[scale=1]
%\draw[help lines] (-1,1) grid (1,-1);

\draw[draw=black, fill=black!10, line width=0.5mm] (-1.0,-1.0) to (-1.0,+0.0) 
                                                               to (+0.0,+0.0)
                                                               to (+0.0,-1.0)
                                                               to (-1.0,-1.0); 

\draw[draw=black, fill=black] (-1.0,+1.0) circle [radius=0.1];
\draw[draw=black, fill=black] (+0.0,+1.0) circle [radius=0.1];
\draw[draw=black, fill=black] (+1.0,+1.0) circle [radius=0.1];
\draw[draw=black, fill=black] (-1.0,+0.0) circle [radius=0.1];
\draw[draw=black, fill=black] (+0.0,+0.0) circle [radius=0.1];
\draw[draw=black, fill=black] (+1.0,+0.0) circle [radius=0.1];
\draw[draw=black, fill=black] (-1.0,-1.0) circle [radius=0.1];
\draw[draw=black, fill=black] (+0.0,-1.0) circle [radius=0.1];
\draw[draw=black, fill=black] (+1.0,-1.0) circle [radius=0.1];

%\draw[draw=red, dashed, line width=0.5mm, ->] (0,0.3) arc (90:360+80:0.3) node[red, above] {\small $180^\circ$};
\draw[draw=red, dashed, line width=0.5mm] (-1,-1) to (1,1);
%\node[red] at (-0.5,-0.5) {id};

\node at (0,2) {2-sphere};
\node at (0,1.5) {$(\infty,2,2)$};
\node at (0,-1.5) {$\swarrow\searrow$};
\node at (0,-2.0) {\bf\large h.};
\end{tikzpicture}
\end{tabular}
\end{table}

\begin{table}[!ht]
\caption{See \COR{L} and the proof of \PRP{L}.}
\centering
\label{tab:L2}
\begin{tabular}{@{}c@{\hspace{1.3cm}}c@{\hspace{1.3cm}}c@{\hspace{1.3cm}}c@{}}
\begin{tikzpicture}[scale=1]
%\draw[help lines] (-1,1) grid (1,-1);

\draw[draw=black, fill=black!10, line width=0.5mm] (-1.0,+1.0) to (+1.0,+1.0) to (+1.0,-1.0) to (-1.0,-1.0) to (-1.0,+1.0);

\draw[draw=black, fill=black] (-1.0,+1.0) circle [radius=0.1];
\draw[draw=black, fill=black] (+0.0,+1.0) circle [radius=0.1];
\draw[draw=black, fill=black] (+1.0,+1.0) circle [radius=0.1];
\draw[draw=black, fill=black] (-1.0,+0.0) circle [radius=0.1];
\draw[draw=black, fill=black] (+0.0,+0.0) circle [radius=0.1];
\draw[draw=black, fill=black] (+1.0,+0.0) circle [radius=0.1];
\draw[draw=black, fill=black] (-1.0,-1.0) circle [radius=0.1];
\draw[draw=black, fill=black] (+0.0,-1.0) circle [radius=0.1];
\draw[draw=black, fill=black] (+1.0,-1.0) circle [radius=0.1];

%\draw[draw=red, dashed, line width=0.5mm, ->] (0,0.3) arc (90:360+80:0.3) node[red, above] {\small $180^\circ$};
\draw[draw=red, dashed, line width=0.5mm] (0,-1) to (0,1);

\node at (0,2) {dS};
\node at (0,1.5) {$(2,8,n)$};
\node at (0,-1.5) {$\rightarrow\downarrow$};
\node at (0,-2.0) {\bf\large a.};
\end{tikzpicture} 
&
\begin{tikzpicture}[scale=1]
%\draw[help lines] (-1,1) grid (1,-1);

\draw[draw=black, fill=black!10, line width=0.5mm] (-1.0,+1.0) to (+1.0,+1.0) to (+1.0,-1.0) to (-1.0,-1.0) to (-1.0,+1.0);

\draw[draw=black, fill=black] (-1.0,+1.0) circle [radius=0.1];
\draw[draw=black, fill=black] (+0.0,+1.0) circle [radius=0.1];
\draw[draw=black, fill=black] (+1.0,+1.0) circle [radius=0.1];
\draw[draw=black, fill=black] (-1.0,+0.0) circle [radius=0.1];
\draw[draw=black, fill=black] (+0.0,+0.0) circle [radius=0.1];
\draw[draw=black, fill=black] (+1.0,+0.0) circle [radius=0.1];
\draw[draw=black, fill=black] (-1.0,-1.0) circle [radius=0.1];
\draw[draw=black, fill=black] (+0.0,-1.0) circle [radius=0.1];
\draw[draw=black, fill=black] (+1.0,-1.0) circle [radius=0.1];

\draw[draw=red, dashed, line width=0.5mm, ->] (0,0.3) arc (90:360+80:0.3) node[red, above] {\small $180^\circ$};
%\draw[draw=red, dashed, line width=0.5mm] (-1,-1) to (1,1);

\node at (0,2) {dS};
\node at (0,1.5) {$(2,8,n)$};
\node at (0,-1.5) {$\rightarrow\downarrow$};
\node at (0,-2.0) {\bf\large b.};
\end{tikzpicture} 
&
\begin{tikzpicture}[scale=1]
%\draw[help lines] (-1,1) grid (1,-1);

\draw[draw=black, fill=black!10, line width=0.5mm] (-1.0,+1.0) to (+1.0,-1.0) 
                                                               to (-1.0,-1.0) 
                                                               to (-1.0,+1.0); 

\draw[draw=black, fill=black] (-1.0,+1.0) circle [radius=0.1];
\draw[draw=black, fill=black] (+0.0,+1.0) circle [radius=0.1];
\draw[draw=black, fill=black] (+1.0,+1.0) circle [radius=0.1];
\draw[draw=black, fill=black] (-1.0,+0.0) circle [radius=0.1];
\draw[draw=black, fill=black] (+0.0,+0.0) circle [radius=0.1];
\draw[draw=black, fill=black] (+1.0,+0.0) circle [radius=0.1];
\draw[draw=black, fill=black] (-1.0,-1.0) circle [radius=0.1];
\draw[draw=black, fill=black] (+0.0,-1.0) circle [radius=0.1];
\draw[draw=black, fill=black] (+1.0,-1.0) circle [radius=0.1];

%\draw[draw=red, line width=0.5mm, ->] (0,0.3) arc (90:360+80:0.3) node[red, above left] {\small $180^\circ$};
\draw[draw=red, dashed, line width=0.5mm] (-1,-1) to (1,1);
%\node[red] at (-0.5,-0.5) {id};

\node at (0,2) {Veronese};
\node at (0,1.5) {$(\infty,4,4)$};
\node at (0,-1.5) {$\searrow$};
\node at (0,-2.0) {\bf\large c.};
\end{tikzpicture}
&
\begin{tikzpicture}[scale=1]
%\draw[help lines] (-1,1) grid (1,-1);

\draw[draw=black, fill=black!10, line width=0.5mm] (-1.0,-1.0) to (-1.0,+0.0) 
                                                               to (+1.0,+1.0) 
                                                               to (+0.0,-1.0)
                                                               to (-1.0,-1.0); 

\draw[draw=black, fill=black] (-1.0,+1.0) circle [radius=0.1];
\draw[draw=black, fill=black] (+0.0,+1.0) circle [radius=0.1];
\draw[draw=black, fill=black] (+1.0,+1.0) circle [radius=0.1];
\draw[draw=black, fill=black] (-1.0,+0.0) circle [radius=0.1];
\draw[draw=black, fill=black] (+0.0,+0.0) circle [radius=0.1];
\draw[draw=black, fill=black] (+1.0,+0.0) circle [radius=0.1];
\draw[draw=black, fill=black] (-1.0,-1.0) circle [radius=0.1];
\draw[draw=black, fill=black] (+0.0,-1.0) circle [radius=0.1];
\draw[draw=black, fill=black] (+1.0,-1.0) circle [radius=0.1];

%\draw[draw=red, line width=0.5mm, ->] (0,0.3) arc (90:360+80:0.3) node[red, above left] {\small $180^\circ$};
\draw[draw=red, dashed, line width=0.5mm] (-1,-1) to (1,1);
%\node[red] at (-0.5,-0.5) {id};

\node at (0,2) {}; %{CP};
\node at (0,1.5) {$(1,4,3)$};
\node at (0,-1.5) {$\swarrow$};
\node at (0,-2.0) {\bf\large d.};
\end{tikzpicture}
\end{tabular}
%\vspace{-6mm}
\end{table}

Before we prove \PRP{L} we state in \LEM{invo-P1P1} and \LEM{invo-T2}
the known classification of real structures of $\P^1$, $\P^1\times\P^1$, $\P^2$ and $\T^2$.
We include proofs in case we could not find a suitable reference.
\THM{circle} collects results from \cite{nls-fam-circles}
that we need for \PRP{L}, \LEM{I2} and \PRP{blowup}.

It will follow from \PRP{blowup} in \SEC{blowup} that \PRP{L} also holds with
the following hypothesis:
{\it``If $X\subseteq\S^n$ is a toric celestial surface, then\ldots''}.

\begin{notation}
\label{ntn:S1S1}
We consider the following normal forms for real structures:
\begin{gather*}
\sigma_+\c\P^1\to\P^1,\quad (x:y)\mapsto (\overline{x}:\overline{y}),
\qquad
\sigma_-\c\P^1\to\P^1,\quad (x:y)\mapsto (-\overline{y}:\overline{x}),
\\[2mm]
\sigma_{s}\c\P^1\times\P^1\to\P^1\times\P^1,\quad (s:t;u:w)\mapsto (\overline{u}:\overline{w};\overline{s}:\overline{t}).
\end{gather*}
Notice that $\P^1\times\P^1$ with real structure $\sigma_+\times \sigma_+$
or $\sigma_s$ is isomorphic to $\S^1\times\S^1$ and $\S^2$, \resp.
\END
\end{notation}

\newpage
\begin{lemma}
\textrm{\bf(real structures for $\P^1$, $\P^1\times\P^1$ and $\P^2$)} 
\label{lem:invo-P1P1}
\begin{Mlist}
\item[\bf a)]
If $\sigma\c\P^1\to\P^1$ is an antiholomorphic involution,
then
there exists $\gamma\in\aut(\P^1_\C)$ \st 
$(\gamma^{-1}\circ\sigma\circ\gamma)$ is equal to either $\sigma_+$ or $\sigma_-$.

\item[\bf b)]
If $\sigma\c\P^1\times\P^1\to\P^1\times\P^1$ 
is an antiholomorphic involution,
then 
there exists $\gamma\in\aut(\P^1_\C\times\P^1_\C)$ \st 
$(\gamma^{-1}\circ\sigma\circ\gamma)$ is equal to either
\\
$\sigma_+\times\sigma_+$,~~
$\sigma_+\times\sigma_-$,~~
$\sigma_-\times\sigma_-$ ~~or~~  $\sigma_{s}$.

\item[\bf c)]
If $\sigma\c\P^2\to\P^2$ is an antiholomorphic involution,
then there exists $\gamma\in\aut(\P^2_\C)$ \st 
$(\gamma^{-1}\circ\sigma\circ\gamma)$ is equal to
$\sigma_0:(s:t:u)\mapsto (\overline{s}:\overline{t}:\overline{u})$.

\end{Mlist}
\end{lemma}

\begin{proof}
{\it Claim 1.
If $X$ is a variety with antiholomorphic involution $\sigma\c X\to X$
and very ample anticanonical class $-\kk$, 
then the following diagram commutes 
\[
\begin{tikzcd}
X \arrow[d,"\sigma"']\arrow[r, "\varphi_{-\kk}"] & Y\subset \P^{h^0(-\kk)-1}\arrow[d,"\sigma_0"]\\
X \arrow[r, "\varphi_{-\kk}"] & Y\subset \P^{h^0(-\kk)-1}
\end{tikzcd}
\]
where $\sigma_0: (x_0:\ldots:x_n) \mapsto (\overline{x_0}:\ldots:\overline{x_n})$ 
and $Y$ is the image of $X$ under the 
birational morphism $\varphi_{-\kk}$ associated to $-\kk$.}
% The linear antiholomorphic
% action of $\sigma$ on the global sections $V:=H^0(\cO_X(\delta))$ 
% is --- up to equivalence --- defined as 
% \[
% \Mfun{\sigma}{V}{V}{v}{j\circ v\circ \sigma},
% \]
% where $\Mrow{j}{\C}{\C}$ denotes complex conjugation,
% \st the induced antiholomorphic involution $\Marrow{\P(V^\star)}{}{\P(V^\star)}$
% is complex conjugation.
% We verify that $\sigma(v_1+v_2)=\sigma(v_1)+\sigma(v_2)$
% and $\sigma(\lambda v_1)=\overline{\lambda}\sigma(v_1)$ for all $v_1,v_2\in V$ \st 
% the $\sigma$ acts as complex conjugation on the constant sections \citep[I.(1.2)]{sil1}.
% For the equivalence relation, we recall that antiholomorphic involutions 
% $\Mrow{\sigma}{X}{X}$ and $\Mrow{\sigma'}{X}{X}$ are equivalent \Iff
% there exists $\gamma\in\aut(X)$ \st $\sigma=\gamma^{-1}\circ\sigma'\circ\gamma$.
% The details for the proof of this claim are left to the reader. 
This claim is a straightforward consequence of \citep[I.(1.2) and I.(1.4)]{sil1}.

a)
We apply claim~1 with $X=\P^1$ 
so that $Y\subset\P^2$ is a real conic. 
We know that $Y$ has signature either $(3,0)$ or $(2,1)$. 
Thus there are, up to inner automorphism, two 
antiholomorphic involutions of $\P^1$. Moreover we have that
$|\set{p\in\P^1}{\sigma(p)=p}|\in\{0,\infty\}$.
This concludes the proof, since the $\sigma$ must be inner automorphic 
to either $\sigma_+$ or $\sigma_-$.

b)
If $\sigma$ does not flip the components of $\P^1\times\P^1$,
then it follows from a) that $\sigma$ is inner automorphic to 
either $\sigma_+\times\sigma_+$, $\sigma_+\times\sigma_-$ or $\sigma_-\times\sigma_-$.
Now suppose that $\sigma$ flips the components of $\P^1\times\P^1$.
Let $\pi_1$ and $\pi_2$ be the complex first and second projections of~$\P^1\times\P^1$ 
to~$\P^1_\C$, \resp.
The composition $\pi_2\circ \sigma\circ \pi_1^{-1}$ defines an antiholomorphic
isomorphism $\tau\c\P^1_\C\to\P^1_\C$.
If $\sigma$ and $\sigma'$ are inner automorphic, 
then there exist complex $\alpha,\beta\in \aut(\P^1_\C)$ \st 
$
\pi_2\circ \sigma\circ \pi_1^{-1}$
is equal to
$
\beta\circ\pi_2\circ \sigma'\circ \pi_1^{-1}\circ \alpha
$.
Conversely, an antiholomorphic isomorphism $\tau\c\P^1_\C\to\P^1_\C$
defines an antiholomorphic involution $(p;q)\mapsto(\tau^{-1}(q);\tau(p))$
that flips the components of~$\P^1\times\P^1$.
There exists $\alpha,\beta\in \aut(\P^1_\C)$ \st $\beta\circ\tau\circ\alpha$ is 
defined by $(s:t)\mapsto (\overline{s}:\overline{t})$.
We conclude that $\sigma$ is unique up to inner automorphisms and thus \Wlog inner automorphic to 
the real structure~$\sigma_s$.

c)
We apply claim~1 with $X=\P^2$
so that $Y\subset\P^9$ is a surface of degree nine. 
Since the degree is odd, we obtain infinitely many real points on $Y$ 
and thus also infinitely many real points on $\P^2$.
We know that $-\kk=3\hh$, where $-\kk$ is the anticanonical class
and $\hh$ is the divisor class of lines in $\P^2$.
We can construct two different real lines in $\P^2$, since 
a line through two real points is real.
The linear subseries of $|-\kk|$ that consists of all 
cubics that contain these real two lines, 
is $|-\kk-2\hh|=|\hh|$.
Notice that choosing a real subsystem of $|-\kk|$ is geometrically 
a real linear projection of $Y$.
It follows that the map $\varphi_\hh\c\P^2\to\P^2$
associated to $\hh$ is real \st $\sigma_0\circ\varphi_\hh=\varphi_\hh\circ\sigma$.
We conclude that $\sigma$ is inner automorphic to $\sigma_0$ as was claimed.
\end{proof}

\begin{lemma}
\textrm{\bf(real structures for $\T^2$)}
\label{lem:invo-T2}
\\
If $\sigma\c\T^2\to\T^2$ 
is a toric antiholomorphic involution,
then 
there exists $\gamma\in\aut(\T^2_\C)$ \st 
$(\gamma^{-1}\circ\sigma\circ\gamma)$ is equal to either
one of the following:
\[
\begin{array}{ll}
\sigma_{0}: (s,u) \mapsto (\overline{s},\overline{u}),   
& 
\sigma_{1}: (s,u) \mapsto (\frac{1}{\overline{s}},\overline{u}),
\\
\sigma_{2}: (s,u) \mapsto (\frac{1}{\overline{s}},\frac{1}{\overline{u}}), 
& 
\sigma_{3}: (s,u) \mapsto (\overline{u},\overline{s}),
\end{array}
\]
and $\sigma_i\c\T^2\to\T^2$ induces, up to unimodular equivalence, 
the following unimodular involution $\sigma_i\c\Z^2\to\Z^2$:
\[
\begin{array}{ll}
\sigma_{0}: (x,y)\mapsto(x,y),   
& 
\sigma_{1}: (x,y)\mapsto(-x,y),
\\
\sigma_{2}: (x,y)\mapsto(-x,-y), 
& 
\sigma_{3}: (x,y)\mapsto(y,x).
\end{array}
\]
The corresponding real points 
$\Gamma_{\sigma_i}:=\set{(s,u)\in\T^2}{\sigma_i(s,u)=(s,u)}$ are:
\[
\begin{array}{l@{}l}
\Gamma_{\sigma_0}=\{(s,u)\in \T^2 ~|~ s=\overline{s},~ u=\overline{u}   & \}\cong (\R^\star)^2     ,    \\
\Gamma_{\sigma_1}=\{(s,u)\in \T^2 ~|~ s\overline{s}=1,~ u=\overline{u}  & \}\cong S^1\times\R^\star,    \\
\Gamma_{\sigma_2}=\{(s,u)\in \T^2 ~|~ s\overline{s}=1,~ u\overline{u}=1 & \}\cong S^1\times S^1    ,    \\
\Gamma_{\sigma_3}=\{(s,u)\in \T^2 ~|~ s=\overline{u}                    & \}\cong \R^2\setminus\{(0,0)\}.
\end{array}
\]
\end{lemma}
 
\begin{proof}
Since $\sigma\c\T^2\to\T^2$ extends to an antiholomorphic involution of an algebraic surface
we may assume that $\sigma$ is defined by 
$(s,u)\mapsto\overline{f(s,u)}$ where $f$ is 
some bivariate rational function in $\C(s,u)$.
From $\sigma(1,1)=(1,1)$ it follows that $f(s,u)=(s^au^b,s^cu^d)$
with $a,b,c,d\in\Z$.
From $(\sigma\circ\sigma)(s,u)=(s,u)$ it follows that 
$ad-bc=\pm1$ and thus the induced unimodular involution
$(x,y)\mapsto (ax+by,cx+dy)$
is unimodular equivalent to $\sigma_i\c\Z^2\to\Z^2$ for some $i\in\{0,1,2,3\}$ as asserted.
For $\sigma_1$ we find that $f(s,u)=(\frac{1}{s},u)$,
and thus $\overline{f(s,u)}=(s,u)$ \Iff $s\overline{s}=1$ and $u=\overline{u}$ 
so that $\Gamma_{\sigma_1}\cong S^1\times\R^\star$.
The proofs for $\Gamma_{\sigma_0}$, $\Gamma_{\sigma_2}$ and $\Gamma_{\sigma_3}$ are similar.
\end{proof}

For convenience of the reader, \THM{circle} below extracts result from
\citep[Theorem~1, Theorem~3, Theorem~4, Corollary~5, Lemma~1 and Lemma~3]{nls-fam-circles} 
that are needed
for \PRP{L}, \LEM{I2} and \PRP{blowup}.

\begin{notation}
\label{ntn:N}
The \df{Neron-Severi lattice} of a surface~$X\subset \P^{n+1}$ 
consists of a unimodular lattice $N(X)$ that is 
defined by the divisor classes on the smooth model~$\X$ up to numerical equivalence.
In this article, $N(X)$ will be a sublattice of 
$\bas{\l_0,~\l_1,~\p_1,~\p_2,~\p_3,~\p_4}_\Z$,
where
$\l_0\cdot\l_1=1$, $\p_i^2=-1$ and $\l_0^2=\l_1^2=\l_0\cdot\p_i=\l_1\cdot\p_i=0$ for $1\leq i\leq 4$.
The real structure $\sigma$ induces a unimodular involution $\sigma\c N(X)\to N(X)$
\st $\sigma(\l_0)=\l_0$, $\sigma(\l_1)=\l_1$, $\sigma(\p_1)=\p_2$
and $\sigma(\p_3)=\p_4$ (see \NTN{nota}).
The function $h^0\c N(X)\to \Z_{\geq0}$ assigns 
to a divisor class the dimension of the vector space of its associated global sections.
The two distinguished elements $\hh,\kk\in N(X)$ correspond to
the \df{class of hyperplane sections} and the \df{canonical class}, \resp.
We call a divisor class $c\in N(X)$ \df{indecomposable}
if $h^0(c)>0$ and if there do not exist nonzero $a,b\in N(X)$ such that
$c=a+b$, $h^0(a)>0$ and $h^0(b)>0$. 
The subset of \df{indecomposable $(-2)$-classes} in $N(X)$
are defined as
\[B(X):=\set{c\in N(X)}{-\kk\cdot c=0, ~c^2=-2 \text{ and } c \text{ is indecomposable}  }.\]
We use the following shorthand notation for elements in $B(X)$:
\[
b_{ij} := \l_0-\p_i-\p_j,\quad
b'_{ij} := \l_1-\p_i-\p_j,\quad
b_1 := \p_1-\p_3
\quad\text{and}\quad 
b_2 := \p_2-\p_4,
\]
and we underline the classes in $\set{b\in B(X)}{\sigma(b)=b}$.
A \df{(projected) weak dP6} is (a degree preserving linear projection of) 
an anticanonical model of $\P^1\times\P^1$ blown up in two complex points
that lie in a fiber of a projection $\P^1\times\P^1\to\P^1$.
We call $X\subset \S^3$ a \df{CH1 cyclide}, if $X(\R)\subset S^3$ 
is an inverse stereographic projection of a circular hyperboloid of 1 sheet.
\END
\end{notation}

\begin{theoremext}
\label{thm:circle}
{\bf [2019]} We use \NTN{N}.
\\
Suppose that $X$ is a celestial surface of type $\bT(X)=(\lambda,d,n)$.
\begin{Mlist}
\item[\bf a)]
If either $\lambda=\infty$, $d>4$, $|\bS(X_N)|> 2$ or $|B(X)|>3$, 
then $\bT(X)$, $\bS(X_N)$, $B(X)$ and the name of~$X$ is characterized  
by a row in the following table, where $3\leq n\leq 7$ and $4\leq m\leq 5$:
\begin{center}
\footnotesize
\setlength\extrarowheight{7pt}
\begin{tabular}{cccl}
$\bT(X)$       & $\bS(X_N)$              & $B(X)$                                                            & name                 \\\hline
$(\infty,4,4)$ & $\emptyset$             & $\emptyset$                                                       & Veronese surface     \\
$(\infty,2,2)$ & $\emptyset$             & $\emptyset$                                                       & 2-sphere             \\
$(2,8,n)$      & $\emptyset$             & $\emptyset$                                                       & (projected) dS       \\
$(3,6,m)$      & $\emptyset$             & $\emptyset$                                                       & (projected) dP6      \\
$(2,6,m)$      & $\underline{A_1}$       & $\{\underline{b_{12}}\}$                                          & (projected) weak dP6 \\
$(4,4,3)$      & $4A_1$                  & $\{b_{13},~b_{24},~b'_{14},~b'_{23}\}$                            & ring cyclide         \\            
$(2,4,3)$      & $2\underline{A_1}+2A_1$ & $\{b_{13},~b_{24},~\underline{b'_{12}},~\underline{b'_{34}}\}$    & spindle cyclide      \\                                                  
$(2,4,3)$      & $\underline{A_3}+2A_1$  & $\{b_{13},~b_{24},~\underline{b'_{12}},~b_1,~b_2\}$               & horn cyclide         \\                                              
$(3,4,3)$      & $\underline{A_1}+2A_1$  & $\{b_{13},~b_{24},~\underline{b_{12}'}\}$                         & CH1 cyclide          \\                           
\end{tabular}
\end{center}

\item[\bf b)]
If $\lambda<\infty$, then the class $\hh$ of hyperplane sections of $X$ is 
equal to the anticanonical class $-\kk$ and  \Wlog equal to either \\
$2\,\l_0+2\,\l_1$, $2\,\l_0+2\,\l_1-\p_1-\p_2$ or $2\,\l_0+2\,\l_1-\p_1-\p_2-\p_3-\p_4$.
\\
If $\lambda=\infty$, then either $\hh=-\frac{2}{3}\kk$ and $\kk^2=9$, or 
$\hh=-\frac{1}{2}\kk$ and $\kk^2=8$.

\item[\bf c)] 
If $\lambda<\infty$,
then the smooth model $\X$ is isomorphic to the blowup of $\P^1\times\P^1$
in either 0, 2 or 4 nonreal complex conjugate points.
These points may be infinitely near, but 
at most two of the non-infinitely near points lie in the same fiber of a projection from
$\P^1\times\P^1$ to $\P^1$.
The pullback into $\X$ of a fiber that contains two points is contracted to an isolated 
singularity of the linear normalization $X_N$.

\end{Mlist}
\end{theoremext}

\begin{proof}[Proof of \PRP{L} and \COR{L}.]
Let $X\subseteq\S^n$ be a $\lambda$-circled toric celestial surface of degree~$d$
that is covered by at least two toric families of circles. The lattice polygon of $X$ 
contains $i$ interior and $b$ boundary lattice points.

{\it
Claim 1: $(i,b,d)\in\{(0,4,2),(0,6,4),(1,4,4),(1,6,6),(1,8,8)\}$
and $\lambda=\infty$ if $i=0$.}
We know from \citep[Propositions~10.5.6 and 10.5.8]{toric} 
that $i$ and $b$ are equal to the sectional genus
$p_a(\hh)=\frac{1}{2}(\hh^2+\kk\cdot\hh)+1$
and anticanonical degree $-\kk\cdot \hh$, \resp~(see also \cite{sch5}).
This claim now follows from \THM{circle}b.

{\it Claim 2:
The lattice polygon of $X$ is, 
up to unimodular equivalence, 
preserved by the unimodular involution $\sigma\c\Z^2\to \Z^2$
that is defined by either 
$\sigma_0$, $\sigma_1$, $\sigma_2$ or $\sigma_3$.}
This claim follows from \LEM{invo-T2}.

{\it Claim 3: 
A boundary line segment of the lattice polygon of~$X$, 
that contains no more and no less than two lattice points, 
is not left invariant by the unimodular involution $\sigma\c\Z^2\to\Z^2$.}
Up to unimodular equivalence, we may assume that the two lattice 
points on a boundary line segment have coordinates $(0,0)$ and $(0,1)$, and that the
remaining lattice points of the polygon lie strictly on the right side of these two points.
Without loss of generality the two lattice points
correspond to the first two components $s^0 u^0$ and $s^0 u^1$ 
of a monomial parametrization $\xi\c\T^2\to X_N$
so that $\xi(0,u)=(1:u:0:\ldots:0)$ parametrizes a line in~$X_N$. 
This line is either linear equivalent 
or linearly projected to a line in~$X$.
We conclude that claim~3 holds,
since $X(\R)\subseteq S^n$ does not contain real lines.

{\it Claim 4:
If $d\neq 2$, then the lattice polygon of $X$ attains its minimal width 
along at least two directions that are left invariant
by the unimodular involution $\sigma\c\Z^2\to\Z^2$.}
Indeed, recall that such a direction corresponds to a toric family of circles.

{\it Claim 5: 
If the lattice polygon of~$X$ is contained in a $3\times 3$ grid centered at the origin,
then $\bL(X)$ together with $\bT(X)$ and the name of~$X$ is in \TAB{L} or \TAB{L2}.}
For each real structure $\sigma$ and pair $(i,b)$
listed at claims~1 and~2,
we list up to equivalence all lattice polygons in the $3\times 3$ grid that are 
left invariant by $\sigma$ and
that have $i$ interior and $b$ boundary lattice points.
Of these polygons we discard those that 
contradict claim~3 or claim~4.
For example, we exclude the lattice types in \TAB{L2}[c,d] as they contradict claim~4
and the lattice polygon of \TAB{L2}d together with unimodular involution $\sigma_0$
would contradict claim~3.
We find that a candidate for $\bL(X)$ is equivalent to 
one of \TAB{L} or \TAB{L2}[a,b].
We recover $|\bS(X_N)|$ from the
monomial parametrization associated to the lattice polygon.
For each lattice type in \TAB{L} and \TAB{L2}[a,b], we apply claim~1 and find 
that either $\lambda=\infty$, $d>4$ or $|\bS(X_N)|>2$.
It follows that the name and celestial type of $X$ 
correspond to a row of \THM{circle}a.
If $\bL(X)$ is \TAB{L}d or \TAB{L2}c, 
then $X$ is a Veronese surface by \DEF{names} as 
the monomials associated to its lattice polygon span 
a basis for vector space of degree 2 forms on $\P^2$.
If $\bL(X)$ is \TAB{L}h, then $X$ is a 2-sphere by claim~1.
If $\bL(X)$ is \TAB{L}e, then there are 4 directions and 
thus $4\leq \lambda<\infty$ so that $X$ must be a ring cyclide.
If $\bL(X)$ is \TAB{L}f or \TAB{L}g,
then $X$ is a spindle cyclide and horn cyclide, \resp~(see forward \EXM{spindle-horn}). 
This concludes the proof of claim~5.  

{\it Claim 6: 
The lattice polygon of~$X$ is contained in a $3\times3$ grid centered at the origin.}
If $\lambda=\infty$ or $d>4$, then it follows from \THM{circle}a, 
that $X_N$ is unique up to projective equivalence 
and thus its lattice type is already realized in \TAB{L} or \TAB{L2} at claim~5.
If $\lambda<\infty$, then  
the lattice polygon of~$X$ must be one of \citep[Theorem~8.3.7]{toric}
\st $(i,b)$ is as in claim~1.
It follows that claim~6 holds.

{\it Claim 7: \COR{L} holds.}
Notice that $X$ is covered by two families of conics that contain real points.
Therefore the real structure of a celestial double Segre surface 
must be inner automorphic to $\sigma_+\times\sigma_+$ 
by \LEM{invo-P1P1}b so that \COR{L}a holds. 
\COR{L}b is a consequence of \LEM{invo-P1P1}c.

We concluded the proof of \PRP{L},
since it follows from claims~5, 6 and 7 that 
$\bL(X)$, $\bT(X)$ and the name of $X$ is 
up to equivalence characterized by one of the eight cases in \TAB{L}.
\end{proof}

\section{Embeddings of \texorpdfstring{$\P^1\times\P^1$}{P1xP1}}
\label{sec:segre}

In this section we give explicit coordinates for double Segre surfaces, which
are embeddings of $\P^1\times\P^1$ into $\P^8$.
We describe how real structures and projective automorphisms act on these embeddings.

We denote the vector space of quadratic forms in the ideal $I(X)$ of a 
surface $X$ as
\[
I_2(X):=\bas{q \in I(X) ~|~ \deg q=2 }_\C. 
\]

\begin{lemma}
%\textrm{\bf(the number of generators)}
\label{lem:I2}
%\\
If $X$ is a toric celestial surface, then $\dim I_2(X_N)$ for the 
linear normalization $X_N$ is as follows: 
\begin{center}
\begin{tabular}{rcccccccc}
%\hline
\TAB{L}:         & a  & b  & c & d & e & f & g & h \\\hline
$\dim I_2(X_N)$: & 20 & 9  & 9 & 6 & 2 & 2 & 2 & 1
\end{tabular}.
\end{center}
% The corresponding celestial types in \TAB{L} are 
% $a:(2,8,7)$, 
% $b:(3,6,5)$, 
% $c:(2,6,5)$, 
% $d:(\infty,4,4)$,
% $e:(4,4,3)$, 
% $f:(2,4,3)$ 
% and 
% $g:(2,4,3)$.
\end{lemma}

\begin{proof}
The dimension of the space $U$ of quadratic forms vanishing on $X_N\subset\P^m$ is equal to $h^0(2\hh)$, 
where $\hh$ is the class of hyperplane sections.
The dimension of the space $W$ of quadratic forms in $\P^m$ is equal to $\binom{2+m}{2}$. 
Thus we find that
\[
\dim I_2(X_N)=\dim W/U=\dim W - \dim U=\binom{2+m}{2}-h^0(2\hh). 
\]
We obtain $h^0(2\hh)=\frac{1}{2}(4\,\hh^2-2\,\hh\cdot \kk)+1$
as a straightforward consequence of
\THM{circle}c,
Riemann-Roch theorem and Kawamata-Viehweg vanishing theorem.
The main assertion now follows from \THM{circle}b.
\end{proof}

\begin{table}[!ht]
\centering
\caption{Coordinates for lattice points.}
\label{tab:coord}
\begin{tabular}{c@{\hspace{3cm}}c}
double Segre surface & Veronese surface
\\
\begin{tikzpicture}
\draw[draw=black, fill=white ] (-1,1) circle [radius=3mm] node[black] {$y_8$};
\draw[draw=black, fill=white ] ( 0,1) circle [radius=3mm] node[black] {$y_3$};
\draw[draw=black, fill=white ] ( 1,1) circle [radius=3mm] node[black] {$y_5$};

\draw[draw=black, fill=white ] (-1,0) circle [radius=3mm] node[black] {$y_2$};
\draw[draw=black, fill=white ] ( 0,0) circle [radius=3mm] node[black] {$y_0$};
\draw[draw=black, fill=white ] ( 1,0) circle [radius=3mm] node[black] {$y_1$};

\draw[draw=black, fill=white ] (-1,-1) circle [radius=3mm] node[black] {$y_6$};
\draw[draw=black, fill=white ] ( 0,-1) circle [radius=3mm] node[black] {$y_4$};
\draw[draw=black, fill=white ] ( 1,-1) circle [radius=3mm] node[black] {$y_7$};

\node at (-1.8,1.5) {\tiny $(-1,1)$};
\node at (0,1.5) {\tiny $(0,1)$};
\node at (1.8,1.5) {\tiny $(1,1)$};

\node at (-1.8,0) {\tiny $(-1,0)$};
%\node at (0,0.5) {\tiny $(0,0)$};
\node at (1.8,0) {\tiny $(1,0)$};

\node at (-1.8,-1.5) {\tiny $(-1,-1)$};
\node at (0,-1.5) {\tiny $(0,-1)$};
\node at (1.8,-1.5) {\tiny $(1,-1)$};
\end{tikzpicture}
&
\begin{tikzpicture}
\draw[draw=black, fill=white ] (-1,1) circle [radius=3mm] node[black] {$y_5$};
\draw[draw=black, fill=white ] ( 0,1) circle [radius=3mm] node[black] {$y_7$};
\draw[draw=black, fill=white ] ( 1,1) circle [radius=3mm] node[black] {$y_8$};

\draw[draw=black, fill=white ] (-1,0) circle [radius=3mm] node[black] {$y_3$};
\draw[draw=black, fill=white ] ( 0,0) circle [radius=3mm] node[black] {$y_1$};
\draw[draw=black, fill=white ] ( 1,0) circle [radius=3mm] node[black] {$y_6$};

\draw[draw=black, fill=white ] (-1,-1) circle [radius=3mm] node[black] {$y_0$};
\draw[draw=black, fill=white ] ( 0,-1) circle [radius=3mm] node[black] {$y_2$};
\draw[draw=black, fill=white ] ( 1,-1) circle [radius=3mm] node[black] {$y_4$};

\node at (-1.8,1.5) {\tiny $(2,0)$};
\node at (0,1.5) {\tiny $(2,1)$};
\node at (1.8,1.5) {\tiny $(2,2)$};

\node at (-1.8,0) {\tiny $(1,0)$};
\node at (1.8,0) {\tiny $(1,2)$};

\node at (-1.8,-1.5) {\tiny $(0,0)$};
\node at (0,-1.5) {\tiny $(0,1)$};
\node at (1.8,-1.5) {\tiny $(0,2)$};
\end{tikzpicture}
\end{tabular}
\end{table}

Let $\Ys\subset\P^8$ denote the linear normalization of the double Segre surface 
with lattice polygon as in \TAB{L}a.
We consider the left coordinates in \TAB{coord} so that we obtain the parametric map 
\begin{gather*}
\xi\c\T^2\to \Ys\subset\P^8,
\quad
(s,u)\mapsto
\\
\begin{array}{l@{~}c@{\colon}c@{\colon}c@{\colon}c@{\colon}c@{\colon}c@{\colon}c@{\colon}c@{\colon}c@{~}l}
( & 1 & s & s^{-1} & u & u^{-1} & su & s^{-1}u^{-1} & su^{-1} & s^{-1}u & ) =
\\
( & y_0 & y_1 & y_2 & y_3 & y_4 & y_5 & y_6 & y_7 & y_8 & ).
\end{array}
\end{gather*}

Using $\xi$ we find the following 20 generators for
the vector space of quadratic forms on $\Ys$
and it follows from \LEM{I2} that these form a basis:
\begin{gather*}
I_2(\Ys)=
\langle
y_0^2  - y_1y_2,~ 
y_0^2  - y_3y_4,~ 
y_0^2  - y_5y_6,~
y_0^2  - y_7y_8,~ 
y_1^2  - y_5y_7,~ 
y_2^2  - y_6y_8,~
\\
y_3^2  - y_5y_8,~ 
y_4^2  - y_6y_7,~ 
y_0y_1 - y_4y_5,~
y_0y_2 - y_3y_6,~
y_0y_3 - y_2y_5,~
y_0y_4 - y_1y_6,~
\\
y_0y_1 - y_3y_7,~
y_0y_2 - y_4y_8,~
y_0y_3 - y_1y_8,~
y_0y_4 - y_2y_7,~
y_0y_5 - y_1y_3,~
\\
y_0y_6 - y_2y_4,~
y_0y_7 - y_1y_4,~
y_0y_8 - y_2y_3 
\rangle_\C.
\end{gather*}

\begin{lemma}
\textrm{\bf(real structures for $\P^8$)}
\label{lem:invo-P8}
\\
Let $\Mi$ denote the imaginary unit. 
The maps $\sigma_i\c\P^8\to\P^8$
and $\mu_i\c\P^8\to\P^8$ 
which are defined by
{\footnotesize%
\begin{gather*}
\begin{array}{l@{}l}
\sigma_0
&\colon
y \mapsto 
(
\overline{y_0}:
\overline{y_1}:
\overline{y_2}:
\overline{y_3}:
\overline{y_4}:
\overline{y_5}:
\overline{y_6}:
\overline{y_7}:
\overline{y_8}
), 
\\
\sigma_1
&\colon
y \mapsto 
(
\overline{y_0}:
\overline{y_2}:
\overline{y_1}:
\overline{y_3}:
\overline{y_4}:
\overline{y_8}:
\overline{y_7}:
\overline{y_6}:
\overline{y_5}
), 
\\
\sigma_2
&\colon
y \mapsto 
(
\overline{y_0}:
\overline{y_2}:
\overline{y_1}:
\overline{y_4}:
\overline{y_3}:
\overline{y_6}:
\overline{y_5}:
\overline{y_8}:
\overline{y_7}
), 
\\
\sigma_3
&\colon
y \mapsto 
(
\overline{y_0}:
\overline{y_3}:
\overline{y_4}:
\overline{y_1}:
\overline{y_2}:
\overline{y_5}:
\overline{y_6}:
\overline{y_8}:
\overline{y_7}
), 
\\
\mu_0
&\colon
x
\mapsto 
x,
\\
\mu_1
&\colon
x
\mapsto
( 
x_0
:x_1 + \Mi x_2  
:x_1 - \Mi x_2  
:x_3            
:x_4            
:x_5 + \Mi x_8  
:x_7 - \Mi x_6  
:x_7 + \Mi x_6  
:x_5 - \Mi x_8   
),
\\
\mu_2
&\colon
x
\mapsto
(
\text{\scalebox{0.90}{$
\frac{x_0}{2}
:x_1 + \Mi x_2  
:x_1 -\Mi x_2   
:x_3 + \Mi x_4  
:x_3 -\Mi x_4   
:x_5 + \Mi x_6  
:x_5 -\Mi x_6   
:x_7 -\Mi x_8   
:x_7 + \Mi x_8   
$}}
),
\\
\mu_3
&\colon
x
\mapsto
( 
x_0
:x_3 - \Mi x_1  
:x_2 + \Mi x_4  
:x_3 + \Mi x_1            
:x_2 - \Mi x_4           
:x_5  
:x_6  
:x_8 - \Mi x_7  
:x_8 + \Mi x_7   
),
\end{array}
\end{gather*}}%
make the following diagram commute for all $0\leq i\leq 3$:
\[
\begin{tikzcd}
\T^2\arrow[d,"\sigma_i"']\arrow[r,"\xi"] & \Ys\arrow[d,"\sigma_i"]\arrow[r,"\mu_i^{-1}"] & X_i\arrow[d,"\sigma_0"] \\
\T^2                     \arrow[r,"\xi"] & \Ys                    \arrow[r,"\mu_i^{-1}"] & X_i                
\end{tikzcd}
\]
where $X_i:=\mu_i^{-1}(\Ys)$ and
real structure
$\sigma_i\c\T^2\to\T^2$
is defined in \LEM{invo-T2}.
\end{lemma}

\begin{proof}
The specification of $\sigma_i\c\Ys\to\Ys$ for $0\leq i\leq 3$ follows from the action on the lattice coordinates in \TAB{coord}
(recall \NTN{nota}).
It is straightforward to verify that $\mu_i^{-1}$ makes the diagram commute.
\end{proof}

\begin{remark}
\label{rmk:invo-P8}
Recall from \COR{L}, \NTN{S1S1} and \LEM{invo-P1P1} that the real structures 
$\sigma_i\c\Ys\to\Ys$ for $0\leq i\leq 2$
are inner automorphic to $\sigma_+\times \sigma_+$
via $\aut(\P^1\times\P^1)$ so that $X_i\cong\S^1\times\S^1$ in these cases.
Notice that $\sigma_3\c\Ys\to\Ys$ is 
via $\aut(\P^1\times\P^1)$ inner automorphic to $\sigma_s$
so that $X_3\cong\S^2$.
\END
\end{remark}

The surface $X_2$ from \LEM{invo-P8} is contained in $\S^7$. 
Indeed, if we compose the first four generators of $I_2(\Ys)$ with $\mu_2$, then
\[
\begin{array}{c@{\quad}c}
(y_0^2-y_1y_2)\circ\mu_2=\frac{1}{4}\,x_0^2-x_1^2-x_2^2, & (y_0^2-y_5y_6)\circ\mu_2=\frac{1}{4}\,x_0^2-x_5^2-x_6^2,\\
(y_0^2-y_3y_4)\circ\mu_2=\frac{1}{4}\,x_0^2-x_3^2-x_4^2, & (y_0^2-y_7y_8)\circ\mu_2=\frac{1}{4}\,x_0^2-x_7^2-x_8^2,
\end{array}
\]
and their sum is the equation of $\S^7$.

We extend $\xi\c\T^2\to\Ys$ \st we obtain the biregular isomorphism
\begin{gather*}
\tilde{\xi}\c\P^1\times\P^1\to\Ys,\quad(s:t;u:w)\mapsto
\\
\begin{array}{l@{~}c@{\colon}c@{\colon}c@{\colon}c@{\colon}c@{\colon}c@{\colon}c@{\colon}c@{\colon}c@{~}l}
( 
& s t u w
& s^2 u w
& t^2 u w
& s t u^2
& s t w^2
& s^2 u^2
& t^2 w^2 
& s^2 w^2
& t^2 u^2
& ) =
\\
( & y_0 & y_1 & y_2 & y_3 & y_4 & y_5 & y_6 & y_7 & y_8 & ),
\end{array}
\end{gather*}
and thus    
$
\autc(\Ys)\cong  \autc(\P^1\times\P^1).
$

\begin{lemma}
%\textrm{\bf(automorphisms of $\P^1\times\P^1$)} 
\label{lem:autP1P1}
%\\
$\autc(\P^1\times\P^1)\cong\autc(\P^1)\times\autc(\P^1)$.
%and the Lie algebra of $\autc(\P^1\times\P^1)$ is $\Msl_2\oplus\Msl_2$.
\end{lemma}

\begin{proof}
Automorphisms in the identity component $\autc(\P^1\times\P^1)$ act 
trivially on the Neron-Severi lattice $N(\P^1\times\P^1)=\bas{\l_0,\l_1}_\Z$,
where generators 
$\l_0$ and $\l_1$
are the classes of the fibers of the first and second projection of $\P^1\times\P^1$
to $\P^1$.
Thus a fiber of $\pi_i$ is mapped by $\varphi\in\autc(\P^1\times\P^1)$
as a whole to a fiber of $\pi_i$ for all $1\leq i\leq 2$
so that the main assertion is concluded.
We remark that 
$\aut_\circ(\P^1)\cong\operatorname{PSL}(2,\R)
\subsetneq
\operatorname{PGL}(2,\R)\cong\aut(\P^1)$.
\end{proof}

We associate to $\varphi=(\varphi_1,\varphi_2)\in \autc(\P^1)\times\autc(\P^1)$ an automorphism 
\[
\cS(\varphi):=\sym_2(\varphi_1)\otimes \sym_2(\varphi_2)\in \autc(\Ys)\subset \autc(\P^8). 
\]
We can compute $\cS(\varphi)$ via the following specification:
\begin{equation}
\label{eqn:S}
\cS(\varphi)\c\Ys\to\Ys,\quad
\tilde{\xi}(p) 
\mapsto
(\tilde{\xi}\circ\varphi)(p),
\end{equation}
for all $p\in\P^1\times\P^1\cong \Ys$.

\begin{example}
\textrm{\bf(toric automorphisms of $\P^1\times\P^1$)}
\label{exm:toric}
\\
Since $\varphi\circ\xi\c\T^2\hto\Ys$ defines an embedding of the algebraic torus $\T^2$
for all automorphisms $\varphi \in \aut(\Ys)$,
the double Segre surface $\Ys$ does not have a unique toric structure. 
Let $\aut^\T_\circ(\Ys)$ denote the identity component of the toric 
automorphisms \wrt $\xi$. We have the following parametrization:
\begin{gather*}
\phi\c\T^2\stackrel{\cong}{\to}\aut^\T_\circ(\Ys),\quad (s,u)\mapsto B(u)\circ A(s),
\qquad\text{where}
\\
A(\alpha):=\cS
\left(
\begin{bmatrix}
\alpha & 0\\ 
0        & \alpha^{-1} 
\end{bmatrix}
,
\begin{bmatrix}
1 & 0\\ 
0 & 1
\end{bmatrix}
\right) 
~~~\text{and}~~~
B(\alpha):=\cS
\left(
\begin{bmatrix}
1 & 0\\ 
0 & 1 
\end{bmatrix}
,
\begin{bmatrix}
\alpha & 0\\ 
0        & \alpha^{-1}   
\end{bmatrix}
\right).
\end{gather*}
Suppose that the real structure of $\Ys$ is defined by $\sigma_2$ in \LEM{invo-P8}.
It follows from \LEM{invo-T2} that 
$\set{p\in\T^2}{\sigma_2(p)=p}\cong S^1\times S^1$ 
and thus
\begin{gather*}
\phi\colon S^1\times S^1\stackrel{\cong}{\to}\aut^\T_\circ(\Ys),\quad
\Bigl(~(\cos(\alpha),\sin(\alpha)),~(\cos(\beta),\sin(\beta))~\Bigr)
\mapsto
\\
\text{\footnotesize{$
\cS
\left(
\begin{bmatrix}
\cos(\alpha)+\Mi\sin(\alpha) & 0\\ 
0        & \cos(\alpha)-\Mi\sin(\alpha)
\end{bmatrix}
,
\begin{bmatrix}
\cos(\beta)+\Mi\sin(\beta) & 0\\ 
0        & \cos(\beta)-\Mi\sin(\beta) 
\end{bmatrix}
\right).
$}}
\end{gather*}
Let $\mu_2\c\P^8\to\P^8$ be as defined in \LEM{invo-P8}.
From the composition of $\phi$ with the pullback $\mu_2^*\c\aut^\T_\circ(\Ys)\to\aut^\T_\circ(X_2)$
we obtain
\begin{gather*}
\mu_2^*\circ\phi\colon S^1\times S^1\to \aut^\T_\circ(X_2),
%\\[2mm]
\Bigl(~(\cos(\alpha),\sin(\alpha)),~(\cos(\beta),\sin(\beta))~\Bigr)
\mapsto
\\
\cS
\left(
\begin{bmatrix}
\cos(\alpha)& -\sin(\alpha) \\ 
\sin(\alpha)& \cos(\alpha)
\end{bmatrix}
,
\begin{bmatrix}
\cos(\beta)& -\sin(\beta) \\ 
\sin(\beta)& \cos(\beta)
\end{bmatrix}
\right)
.
\end{gather*}
Notice that the real structure of $X_2\subset\S^7$ is defined by $\sigma_0$ in \LEM{invo-P8}
and that $\autc^\T(X_2)\cong \PSO(2)\times \PSO(2)$. 
\END
\end{example}

\section{Blowups of \texorpdfstring{$\P^1\times\P^1$}{P1xP1}}
\label{sec:blowup}

The smooth model of a celestial surface that is not $\infty$-circled
is either $\P^1\times\P^1$ or the blowup of $\P^1\times\P^1$ 
in two or four points. 
Such a blowup is realized by a linear projection of the double Segre surface $\Ys$ in $\P^8$. 
The automorphisms of the image surface induce 
automorphisms of $\P^1\times\P^1$ that leave the center of blowup invariant. 
This allows us to formulate restrictions on the possible M\"obius automorphism groups
of celestial surfaces. 
In particular, we find that celestial surfaces with many symmetries must be toric.

\begin{proposition}
\textrm{\bf(blowups of $\P^1\times\P^1$)}
\label{prp:blowup}
\\
If a celestial surface $X\subset \S^n$ is not $\infty$-circled
and $\dim\autc(X)\geq 2$,  
then 
its linear normalization $X_N$ is a toric surface,
each family of circles on $X$ is toric,
and 
$\autc(X)$ embeds as a subgroup into $\autc(\P^1)\times\autc(\P^1)$.
Moreover, there exists a birational linear projection
\[
\rho\c\Ys\subset\P^8\dto X\subset\P^{n+1},
\]
whose center of projection is characterized by a row in \TAB{P1P1} 
together with $\bT(X)$, $\bS(X_N)$ and
the projections of  
$\autc(X)$
to a subgroup of $\autc(\P^1)$.
\end{proposition}

\begin{table}
\caption{
See \PRP{blowup}.
The possible configurations of the center of blowup $\Lambda\subset\P^1\times\P^1$ realized by 
the birational linear projection $\rho\c\Ys\dto X$ 
via the isomorphism $\P^1\times\P^1\cong\Ys$.
At the entries for $\bT(X)$ we have $3\leq n\leq 7$ and $4\leq m\leq 5$.
Since $\autc(X)$ embeds into $\autc(\P^1)\times\autc(\P^1)$, we find that 
the projection $\pi_{i}(\autc(X))$ is a subgroup of $\autc(\P^1)$ for $i\in\{1,2\}$.
An entry for $\pi_1(\autc(X))$ and $\pi_2(\autc(X))$ denotes the maximal possible subgroup.
The vertical and horizontal line segments in the diagrams represent fibers of the projections 
$\pi_i\c\P^1\times\P^1\to\P^1$
for $i\in\{1,2\}$.
The complex conjugate points $q$ and $\overline{q}$ in diagram {\bf (f)} 
are infinitely near to $p$ and $\overline{p}$, \resp.    
A fiber that contains two centers of blowup is contracted by $\rho$ to an isolated singularity of~$X$.}
\label{tab:P1P1}
\centering
\vspace{-5mm}
\setlength\extrarowheight{2mm}
{\footnotesize%
\begin{tabular}{cccccl}
$\Lambda$   & $\bT(X)$   & $\bS(X_N)$                & $\pi_{1}(\autc(X))$  & $\pi_{2}(\autc(X))$ & name                  \\\hline
{\bf(a)}    & $(2,8,n)$  & $\emptyset$               & $\PSL(2)$            & $\PSL(2)$           & (projected) dS        \\
{\bf(b)}    & $(3,6,m)$  & $\emptyset$               & $\PSO(2)$            & $\PSO(2)$           & (projected) dP6       \\
{\bf(c)}    & $(2,6,m)$  & $A_1$                     & $\PSO(2)$            & $\PSA(1)$           & (projected) weak dP6  \\
{\bf(d)}    & $(4,4,3)$  & $4A_1$                    & $\PSO(2)$            & $\PSO(2)$           & ring cyclide          \\
{\bf(e)}    & $(2,4,3)$  & $2\underline{A_1}+2A_1$   & $\PSO(2)$            & $\PSX(1)$           & spindle cyclide       \\
{\bf(f)}    & $(2,4,3)$  & $\underline{A_3}+2A_1$    & $\PSO(2)$            & $\PSE(1)$           & horn cyclide          \\
%\hline
\end{tabular}}%
\\[5mm]
\begin{tabular}{@{}c@{~~}c@{~~}c@{}}
\begin{tikzpicture}[scale=0.7]
%\draw[help lines] (-3,3) grid (3,-3);
%\draw[draw=orange, fill=yellow, thick] (0,0) circle [radius=0.1];

\draw[line width=0.2mm, draw=black!40!green] (-2.5,+3.0) -- (-2.5,-1.0);
\draw[line width=0.2mm, draw=black!40!green] (+2.8,-2.0) -- (-1.0,-2.0);
\node[black!40!green] at (-2.5,-1.0-0.3) {$\P^1$};
\node[black!40!green] at (-1.0-0.3,-2.0) {$\P^1$};

\draw[line width=0.3mm, draw=gray ] (-1,3) -- (-1,-1);
\draw[line width=0.3mm, draw=gray ] (+0,3) -- (+0,-1);
\draw[line width=0.3mm, draw=gray ] (+1,3) -- (+1,-1);
\draw[line width=0.3mm, draw=gray ] (+2,3) -- (+2,-1);

\draw[line width=0.3mm, draw=gray ] (-1.5,+2.5) -- (2.8,+2.5);
\draw[line width=0.3mm, draw=gray ] (-1.5,+1.5) -- (2.8,+1.5);
\draw[line width=0.3mm, draw=gray ] (-1.5,+0.5) -- (2.8,+0.5);
\draw[line width=0.3mm, draw=gray ] (-1.5,-0.5) -- (2.8,-0.5);

\draw[line width=0.4mm, draw=black, -> ] (-1.8,+1.5) -- (-2.4,+1.5);
\draw[line width=0.4mm, draw=black, -> ] (+1.0,-1.2) -- (+1.0,-1.8);
\node at (-2.4+0.5,+1.5+0.3) {$\pi_1$};
\node at (+1.0+0.4,-1.8+0.4) {$\pi_2$}; 

% for making the height equal to the other pictures
\node[white] at (+0.0+0.1,-2.0-0.5) {$\pi_2(\overline{p})$}; 
\node[white] at (+2.0+0.1,-2.0-0.5) {$\pi_2(p)$};

\draw[draw=black, fill=white] (-2.5,-2.3) circle [radius=0.4];
\node at (-2.5,-2.3) {\large\bf a};
\end{tikzpicture}
&
\begin{tikzpicture}[scale=0.7]
%\draw[help lines] (-3,3) grid (3,-3);
%\draw[draw=orange, fill=yellow, thick] (0,0) circle [radius=0.1];

\draw[line width=0.2mm, draw=black!40!green] (-2.5,+3.0) -- (-2.5,-1.0);
\draw[line width=0.2mm, draw=black!40!green] (+2.8,-2.0) -- (-1.0,-2.0);
\node[black!40!green] at (-2.5,-1.0-0.3) {$\P^1$};
\node[black!40!green] at (-1.0-0.3,-2.0) {$\P^1$};

\draw[line width=0.3mm, draw=gray ] (-1,3) -- (-1,-1);
\draw[line width=0.3mm, draw=gray ] (+0,3) -- (+0,-1);
\draw[line width=0.3mm, draw=gray ] (+1,3) -- (+1,-1);
\draw[line width=0.3mm, draw=gray ] (+2,3) -- (+2,-1);

\draw[line width=0.3mm, draw=gray ] (-1.5,+2.5) -- (2.8,+2.5);
\draw[line width=0.3mm, draw=gray ] (-1.5,+1.5) -- (2.8,+1.5);
\draw[line width=0.3mm, draw=gray ] (-1.5,+0.5) -- (2.8,+0.5);
\draw[line width=0.3mm, draw=gray ] (-1.5,-0.5) -- (2.8,-0.5);

\draw[line width=0.4mm, draw=black, -> ] (-1.8,+1.5) -- (-2.4,+1.5);
\draw[line width=0.4mm, draw=black, -> ] (+1.0,-1.2) -- (+1.0,-1.8);
\node at (-2.4+0.5,+1.5+0.3) {$\pi_1$};
\node at (+1.0+0.4,-1.8+0.4) {$\pi_2$};

\draw[line width=0.8mm, draw=red, dashed] (+0.0,+3.0) -- (+0.0,-2.0);
\draw[line width=0.8mm, draw=red, dashed] (+2.0,+3.0) -- (+2.0,-2.0);

\draw[line width=0.8mm, draw=blue, dashed ] (-2.5,+2.5) -- (+2.8,+2.5);
\draw[line width=0.8mm, draw=blue, dashed ] (-2.5,+0.5) -- (+2.8,+0.5);

\draw[draw=black, fill=purple!20] (+0.0,+0.5) circle [radius=0.2] node[below right] {$\overline{p}$};
%\draw[draw=black, fill=purple!20] (+0.0,+2.5) circle [radius=0.2] node[below right] {$q$};
%\draw[draw=black, fill=purple!20] (+2.0,+0.5) circle [radius=0.2] node[below right] {$\overline{q}$};
\draw[draw=black, fill=purple!20] (+2.0,+2.5) circle [radius=0.2] node[below right] {$p$};

\draw[draw=black, fill=green!20] (-2.5,+2.5) circle [radius=0.2];
\draw[draw=black, fill=green!20] (-2.5,+0.5) circle [radius=0.2];
\node[black] at (-2.5+0.1,+2.5+0.5) {$\pi_1(p)$};
\node[black] at (-2.5+0.1,+0.5-0.5) {$\pi_1(\overline{p})$};

\draw[draw=black, fill=green!20] (+0.0,-2.0) circle [radius=0.2];
\draw[draw=black, fill=green!20] (+2.0,-2.0) circle [radius=0.2];
\node[black] at (+0.0+0.1,-2.0-0.5) {$\pi_2(\overline{p})$};
\node[black] at (+2.0+0.1,-2.0-0.5) {$\pi_2(p)$};

\draw[draw=black, fill=white] (-2.5,-2.3) circle [radius=0.4];
\node at (-2.5,-2.3) {\large\bf b};
\end{tikzpicture}
&
\begin{tikzpicture}[scale=0.7]
%\draw[help lines] (-3,3) grid (3,-3);
%\draw[draw=orange, fill=yellow, thick] (0,0) circle [radius=0.1];

\draw[line width=0.2mm, draw=black!40!green] (-2.5,+3.0) -- (-2.5,-1.0);
\draw[line width=0.2mm, draw=black!40!green] (+2.8,-2.0) -- (-1.0,-2.0);
\node[black!40!green] at (-2.5,-1.0-0.3) {$\P^1$};
\node[black!40!green] at (-1.0-0.3,-2.0) {$\P^1$};

\draw[line width=0.3mm, draw=gray ] (-1,3) -- (-1,-1);
\draw[line width=0.3mm, draw=gray ] (+0,3) -- (+0,-1);
\draw[line width=0.3mm, draw=gray ] (+1,3) -- (+1,-1);
\draw[line width=0.3mm, draw=gray ] (+2,3) -- (+2,-1);

\draw[line width=0.3mm, draw=gray ] (-1.5,+2.5) -- (2.8,+2.5);
\draw[line width=0.3mm, draw=gray ] (-1.5,+1.5) -- (2.8,+1.5);
\draw[line width=0.3mm, draw=gray ] (-1.5,+0.5) -- (2.8,+0.5);
\draw[line width=0.3mm, draw=gray ] (-1.5,-0.5) -- (2.8,-0.5);

\draw[line width=0.4mm, draw=black, -> ] (-1.8,+1.5) -- (-2.4,+1.5);
\draw[line width=0.4mm, draw=black, -> ] (+1.0,-1.2) -- (+1.0,-1.8);
\node at (-2.4+0.5,+1.5+0.3) {$\pi_1$};
\node at (+1.0+0.4,-1.8+0.4) {$\pi_2$};

\draw[line width=0.8mm, draw=red, dashed] (+2.0,+3.0) -- (+2.0,-2.0);

\draw[line width=0.8mm, draw=blue, dashed ] (-2.5,+2.5) -- (+2.8,+2.5);
\draw[line width=0.8mm, draw=blue, dashed ] (-2.5,+0.5) -- (+2.8,+0.5);

\draw[draw=black, fill=purple!20] (+2.0,+0.5) circle [radius=0.2] node[below right] {$\overline{p}$};
\draw[draw=black, fill=purple!20] (+2.0,+2.5) circle [radius=0.2] node[below right] {$p$};

\draw[draw=black, fill=green!20] (-2.5,+2.5) circle [radius=0.2];
\draw[draw=black, fill=green!20] (-2.5,+0.5) circle [radius=0.2];
\node[black] at (-2.5+0.1,+2.5+0.5) {$\pi_1(p)$};
\node[black] at (-2.5+0.1,+0.5-0.5) {$\pi_1(\overline{p})$};

\draw[draw=black, fill=green!20] (+2.0,-2.0) circle [radius=0.2];
\node[black] at (+2.0+0.1,-2.0-0.5) {$\pi_2(p)$};

\draw[draw=black, fill=white] (-2.5,-2.3) circle [radius=0.4];
\node at (-2.5,-2.3) {\large\bf c};
\end{tikzpicture}
\\[5mm]
\begin{tikzpicture}[scale=0.7]
%\draw[help lines] (-3,3) grid (3,-3);
%\draw[draw=orange, fill=yellow, thick] (0,0) circle [radius=0.1];

\draw[line width=0.2mm, draw=black!40!green] (-2.5,+3.0) -- (-2.5,-1.0);
\draw[line width=0.2mm, draw=black!40!green] (+2.8,-2.0) -- (-1.0,-2.0);
\node[black!40!green] at (-2.5,-1.0-0.3) {$\P^1$};
\node[black!40!green] at (-1.0-0.3,-2.0) {$\P^1$};

\draw[line width=0.3mm, draw=gray ] (-1,3) -- (-1,-1);
\draw[line width=0.3mm, draw=gray ] (+0,3) -- (+0,-1);
\draw[line width=0.3mm, draw=gray ] (+1,3) -- (+1,-1);
\draw[line width=0.3mm, draw=gray ] (+2,3) -- (+2,-1);

\draw[line width=0.3mm, draw=gray ] (-1.5,+2.5) -- (2.8,+2.5);
\draw[line width=0.3mm, draw=gray ] (-1.5,+1.5) -- (2.8,+1.5);
\draw[line width=0.3mm, draw=gray ] (-1.5,+0.5) -- (2.8,+0.5);
\draw[line width=0.3mm, draw=gray ] (-1.5,-0.5) -- (2.8,-0.5);

\draw[line width=0.4mm, draw=black, -> ] (-1.8,+1.5) -- (-2.4,+1.5);
\draw[line width=0.4mm, draw=black, -> ] (+1.0,-1.2) -- (+1.0,-1.8);
\node at (-2.4+0.5,+1.5+0.3) {$\pi_1$};
\node at (+1.0+0.4,-1.8+0.4) {$\pi_2$};

\draw[line width=0.8mm, draw=red, dashed] (+0.0,+3.0) -- (+0.0,-2.0);
\draw[line width=0.8mm, draw=red, dashed] (+2.0,+3.0) -- (+2.0,-2.0);

\draw[line width=0.8mm, draw=blue, dashed ] (-2.5,+2.5) -- (+2.8,+2.5);
\draw[line width=0.8mm, draw=blue, dashed ] (-2.5,+0.5) -- (+2.8,+0.5);

\draw[draw=black, fill=purple!20] (+0.0,+0.5) circle [radius=0.2] node[below right] {$\overline{p}$};
\draw[draw=black, fill=purple!20] (+0.0,+2.5) circle [radius=0.2] node[below right] {$q$};
\draw[draw=black, fill=purple!20] (+2.0,+0.5) circle [radius=0.2] node[below right] {$\overline{q}$};
\draw[draw=black, fill=purple!20] (+2.0,+2.5) circle [radius=0.2] node[below right] {$p$};

\draw[draw=black, fill=green!20] (-2.5,+2.5) circle [radius=0.2];
\draw[draw=black, fill=green!20] (-2.5,+0.5) circle [radius=0.2];
\node[black] at (-2.5+0.1,+2.5+0.5) {$\pi_1(p)$};
\node[black] at (-2.5+0.1,+0.5-0.5) {$\pi_1(\overline{p})$};

\draw[draw=black, fill=green!20] (+0.0,-2.0) circle [radius=0.2];
\draw[draw=black, fill=green!20] (+2.0,-2.0) circle [radius=0.2];
\node[black] at (+0.0+0.1,-2.0-0.5) {$\pi_2(\overline{p})$};
\node[black] at (+2.0+0.1,-2.0-0.5) {$\pi_2(p)$};

\draw[draw=black, fill=white] (-2.5,-2.3) circle [radius=0.4];
\node at (-2.5,-2.3) {\large\bf d};
\end{tikzpicture}
&
\begin{tikzpicture}[scale=0.7]
%\draw[help lines] (-3,3) grid (3,-3);
%\draw[draw=orange, fill=yellow, thick] (0,0) circle [radius=0.1];

\draw[line width=0.2mm, draw=black!40!green] (-2.5,+3.0) -- (-2.5,-1.0);
\draw[line width=0.2mm, draw=black!40!green] (+2.8,-2.0) -- (-1.0,-2.0);
\node[black!40!green] at (-2.5,-1.0-0.3) {$\P^1$};
\node[black!40!green] at (-1.0-0.3,-2.0) {$\P^1$};

\draw[line width=0.3mm, draw=gray ] (-1,3) -- (-1,-1);
\draw[line width=0.3mm, draw=gray ] (+0,3) -- (+0,-1);
\draw[line width=0.3mm, draw=gray ] (+1,3) -- (+1,-1);
\draw[line width=0.3mm, draw=gray ] (+2,3) -- (+2,-1);

\draw[line width=0.3mm, draw=gray ] (-1.5,+2.5) -- (2.8,+2.5);
\draw[line width=0.3mm, draw=gray ] (-1.5,+1.5) -- (2.8,+1.5);
\draw[line width=0.3mm, draw=gray ] (-1.5,+0.5) -- (2.8,+0.5);
\draw[line width=0.3mm, draw=gray ] (-1.5,-0.5) -- (2.8,-0.5);

\draw[line width=0.4mm, draw=black, -> ] (-1.8,+1.5) -- (-2.4,+1.5);
\draw[line width=0.4mm, draw=black, -> ] (+1.0,-1.2) -- (+1.0,-1.8);
\node at (-2.4+0.5,+1.5+0.3) {$\pi_1$};
\node at (+1.0+0.4,-1.8+0.4) {$\pi_2$};

\draw[line width=0.8mm, draw=red, dashed] (+0.0,+3.0) -- (+0.0,-2.0);
\draw[line width=0.8mm, draw=red, dashed] (+2.0,+3.0) -- (+2.0,-2.0);

\draw[line width=0.8mm, draw=blue, dashed ] (-2.5,+2.5) -- (+2.8,+2.5);
\draw[line width=0.8mm, draw=blue, dashed ] (-2.5,+0.5) -- (+2.8,+0.5);

\draw[draw=black, fill=purple!20] (+0.0,+0.5) circle [radius=0.2] node[below right] {$\overline{p}$};
\draw[draw=black, fill=purple!20] (+0.0,+2.5) circle [radius=0.2] node[below right] {$p$};
\draw[draw=black, fill=purple!20] (+2.0,+0.5) circle [radius=0.2] node[below right] {$\overline{q}$};
\draw[draw=black, fill=purple!20] (+2.0,+2.5) circle [radius=0.2] node[below right] {$q$};

\draw[draw=black, fill=green!20] (-2.5,+2.5) circle [radius=0.2];
\draw[draw=black, fill=green!20] (-2.5,+0.5) circle [radius=0.2];
\node[black] at (-2.5+0.1,+2.5+0.5) {$\pi_1(p)$};
\node[black] at (-2.5+0.1,+0.5-0.5) {$\pi_1(\overline{p})$};

\draw[draw=black, fill=green!20] (+0.0,-2.0) circle [radius=0.2];
\draw[draw=black, fill=green!20] (+2.0,-2.0) circle [radius=0.2];
\node[black] at (+0.0+0.1,-2.0-0.5) {$\pi_2(p)$};
\node[black] at (+2.0+0.1,-2.0-0.5) {$\pi_2(q)$};

\draw[draw=black, fill=white] (-2.5,-2.3) circle [radius=0.4];
\node at (-2.5,-2.3) {\large\bf e};
\end{tikzpicture}
&
\begin{tikzpicture}[scale=0.7]
%\draw[help lines] (-3,3) grid (3,-3);
%\draw[draw=orange, fill=yellow, thick] (0,0) circle [radius=0.1];

\draw[line width=0.2mm, draw=black!40!green] (-2.5,+3.0) -- (-2.5,-1.0);
\draw[line width=0.2mm, draw=black!40!green] (+2.8,-2.0) -- (-1.0,-2.0);
\node[black!40!green] at (-2.5,-1.0-0.3) {$\P^1$};
\node[black!40!green] at (-1.0-0.3,-2.0) {$\P^1$};

\draw[line width=0.3mm, draw=gray ] (-1,3) -- (-1,-1);
\draw[line width=0.3mm, draw=gray ] (+0,3) -- (+0,-1);
\draw[line width=0.3mm, draw=gray ] (+1,3) -- (+1,-1);
\draw[line width=0.3mm, draw=gray ] (+2,3) -- (+2,-1);

\draw[line width=0.3mm, draw=gray ] (-1.5,+2.5) -- (2.8,+2.5);
\draw[line width=0.3mm, draw=gray ] (-1.5,+1.5) -- (2.8,+1.5);
\draw[line width=0.3mm, draw=gray ] (-1.5,+0.5) -- (2.8,+0.5);
\draw[line width=0.3mm, draw=gray ] (-1.5,-0.5) -- (2.8,-0.5);

\draw[line width=0.4mm, draw=black, -> ] (-1.8,+1.5) -- (-2.4,+1.5);
\draw[line width=0.4mm, draw=black, -> ] (+1.0,-1.2) -- (+1.0,-1.8);
\node at (-2.4+0.5,+1.5+0.3) {$\pi_1$};
\node at (+1.0+0.4,-1.8+0.4) {$\pi_2$};

\draw[line width=0.8mm, draw=red, dashed] (+2.0,+3.0) -- (+2.0,-2.0);

\draw[line width=0.8mm, draw=blue, dashed ] (-2.5,+2.5) -- (+2.8,+2.5);
\draw[line width=0.8mm, draw=blue, dashed ] (-2.5,+0.5) -- (+2.8,+0.5);

\draw[draw=black, fill=purple!20] (+2.0,+0.5) circle [radius=0.2] node[below right] {$\overline{p}$};
\draw[draw=black, fill=purple!20] (+2.0,+2.5) circle [radius=0.2] node[below right] {$p$};

\draw[draw=black, fill=orange!30] (+2.0-0.4,+0.5+0.15) rectangle (+2.0-0.1,+0.5-0.15) node[below left] {$\overline{q}$};
\draw[draw=black, fill=orange!30] (+2.0-0.4,+2.5+0.15) rectangle (+2.0-0.1,+2.5-0.15) node[below left] {$q$};

\draw[draw=black, fill=green!20] (-2.5,+2.5) circle [radius=0.2];
\draw[draw=black, fill=green!20] (-2.5,+0.5) circle [radius=0.2];
\node[black] at (-2.5+0.1,+2.5+0.5) {$\pi_1(p)$};
\node[black] at (-2.5+0.1,+0.5-0.5) {$\pi_1(\overline{p})$};

\draw[draw=black, fill=green!20] (+2.0,-2.0) circle [radius=0.2];
\node[black] at (+2.0+0.1,-2.0-0.5) {$\pi_2(p)$};

\draw[draw=black, fill=white] (-2.5,-2.3) circle [radius=0.4];
\node at (-2.5,-2.3) {\large\bf f};
\end{tikzpicture}   
\end{tabular}
\end{table}

\begin{proof}
By \THM{circle}[c,b], 
the smooth model $\X$ is isomorphic to the blowup of $\P^1\times\P^1$ 
in a center $\Lambda$ \st $|\Lambda|\in\{0,2,4\}$
and $X_N$ is its anticanonical model.
Hence $\autc(X_N)\cong\autc(\X)$ and thus 
$\autc(X)$ defines a subgroup of $\autc(\X)$.
Moreover, $\autc(\X)$ is isomorphic to a subgroup of $\autc(\P^1\times\P^1)$
whose action leaves the blowup center $\Lambda$ invariant. 
It now follows from \LEM{autP1P1} that $\autc(X)$ is isomorphic to a subgroup of 
\[
\left\{ \varphi \in\autc(\P^1)\times\autc(\P^1)
~|~
\varphi
\bigl(
\pi_1(\Lambda),\pi_2(\Lambda)
\bigr)
=
\bigl(
\pi_1(\Lambda),\pi_2(\Lambda)
\bigr)
\right\},
\]
where $\pi_1$ and $\pi_2$
denote the projections of $\P^1\times\P^1$ to its $\P^1$ factors. 
We will denote the projections of $\autc(\P^1)\times\autc(\P^1)$
to its $\autc(\P^1)$ factors, by $\pi_1$ and $\pi_2$ as well.

We claim that $|\pi_1(\Lambda)|\leq 2$ and $|\pi_2(\Lambda)|\leq 2$.
Suppose by contradiction that $|\pi_1(\Lambda)|>2$ so that $|\Lambda|=4$.
Let $\pi_1(\autc(X))$ and $\pi_2(\autc(X))$
denote the subgroups of $\autc(\P^1)$
that preserve $\pi_1(\Lambda)$ and $\pi_2(\Lambda)$, \resp.  
Recall that $\autc(\P^1)$ is 3-transitive and thus $\dim\pi_1(\autc(X))=0$.
By assumption $\dim\autc(X)\geq 2$, hence 
$|\pi_2(\Lambda)|=1$ and $\dim\autc(X)=\dim\pi_2(\autc(X))=2$.
By \THM{circle}c at most two non-infinitely near points lie in the same fiber
and these points are nonreal.
Thus $\Lambda$ consist of 
two complex conjugate points and two infinitely near points. 
We arrived at a contradiction as $\pi_2(\autc(X))$
must be a proper subgroup of $\PSA(1)$
so that $\pi_2(\autc(X))\cong\PSE(1)$. 

\newpage
As $|\pi_1(\Lambda)|\leq 2$, $|\pi_2(\Lambda)|\leq 2$
and $\Lambda$ does not contain real points,
it follows that all possible configurations of 
$\Lambda$ together with  $\pi_1(\autc(X))$ and $\pi_2(\autc(X))$
are listed in \TAB{P1P1}.
Moreover, since the algebraic torus $\T^1$ embeds into $\P^1\setminus\pi_i(\Lambda)$
\st $\autc(\T^1)$ extends to a subgroup of 
$\pi_{i}(\autc(X))$ for $i\in\{1,2\}$,
we deduce that $X_N$ must be toric.

The bidegree $(2,2)$ forms define an isomorphism $\P^1\times\P^1\to\Ys\subset\P^8$.
Since $X_N$ is an anticanonical model,
the bidegree $(2,2)$ forms that pass through the blowup center $\Lambda$,
define a birational map $\P^1\times\P^1\dto X_N\subset\P^{r}$ for some $n+1\leq r\leq 8$.
Assigning linear conditions to the forms, so that they pass through $\Lambda$, 
corresponds to a linear projection $f\c\Ys\subset\P^8\dto X_N\subset\P^r$.
It follows from the definition of linear normalization that 
there exists a degree preserving linear projection $g\c X_N\to X$.
We now define $\rho\c\Ys\dto X$ as the composition of $f$ with $g$.
Notice that $\Ys\cong\P^1\times\P^1$ and thus we may interpret the blowup center
$\Lambda$ instead as the center of projection of $\rho$.
It follows from \THM{circle}c that the fibers that contain
two points in $\Lambda$ are contracted via $\rho$ to an isolated singularity of $X$.

We assume \Wlog that the generators $\l_0$ and $\l_1$ as defined at \NTN{N} 
are the classes of the pullbacks to $\X$ of the fibers of $\pi_1$ and $\pi_2$, \resp.
The generators $\p_1$, $\p_2$, $\p_3$ and $\p_4$ 
are the classes of the pullbacks of $(-1)$-curves that contract to
the points $p$, $\overline{p}$, $q$ and $\overline{q}$, \resp.
For each configuration of $\Lambda$ we obtain an explicit description of~$B(X)$.
For example, if $\Lambda$ is as in \TAB{P1P1}f, then 
$b_1\in B(X)$ since  $q$ is infinitely near to $p$, 
and $\underline{b_{12}} \in B(X)$ since $p$ and $\overline{p}$ lie in a real fiber of $\pi_2$,
and so on.
Notice that $\bS(X_N)$ 
corresponds to the Dynkin diagram 
with vertex set $B(X)$ and edge set $\set{(a,b)}{a\cdot b>0}$. 
We find that $|B(X)|>3$ in each case and thus
the values at the $\bT(X)$, $\bS(X_N)$ and name columns
are a direct consequence of \THM{circle}a.

If any two choices for the blowup center $\Lambda$
are characterized by the same row of \TAB{P1P1},
then these choices are equivalent up to $\aut(\P^1\times\P^1)$. 
It follows that the linear normalization $X_N\subset \P^r$
is up to $\aut(\P^r)$ uniquely determined by the name of the celestial surface $X$
and thus $\bL(X)$ is up to equivalence determined by
this name as well. We now apply \PRP{L} and
find by comparing \TAB{P1P1} with \TAB{L},
that each family of circles on $X$ is realized by some toric family.
\end{proof}

% \begin{remark}
% \label{rmk:tf}
% If $X\subset\S^n$ is a toric celestial surface
% that is not $\infty$-circled, then
% we find by comparing \TAB{P1P1} with \TAB{L},
% that each family of circles on $X$ 
% is projectively equivalent to, or a linear projection of, a toric family on $X_N$.
% \END
% \end{remark}

\newpage
\begin{remark}
\textrm{\bf(toric projections of the double Segre surface)}
\label{rmk:proj}
\\
Recall that a lattice polygon in \TAB{L} defines, 
up to projective isomorphism, 
a monomial parametrization of the linear normalization of a toric celestial surface. 
The inclusion of lattice polygons with the same unimodular involution,
defines an arrow reversing projection between the corresponding toric models.
The corresponding \df{toric projection} is defined by omitting
components of the monomial parametrization associated to the bigger lattice polygon
\st the exponents of the remaining components define the 
lattice points of the smaller lattice polygon.
Thus toric surfaces with lattice types 
b, c, e, f, g in \TAB{L} are toric projections of 
b, a, b, a, a in \TAB{L2}, \resp.
We will use this concept in \EXM{proj-dP6}, \EXM{spindle-horn}
and in the proof of \LEM{M1}.
\END
\end{remark}

\begin{example}
\textrm{\bf(dP6 as the image of a toric projection)}
\label{exm:proj-dP6}
\\
Suppose that $\Ys\subset\P^8$ and $Z\subset\P^6$ have 
lattice types as in \TAB{L2}b and  \TAB{L}b, \resp.
Thus the real structure of $\Ys$ is defined by $\sigma_2$ in \LEM{invo-P8}.
We use the left coordinates of \TAB{coord} and omit the monomial components corresponding to $y_5$ and $y_6$ coordinates \st
\[
\begin{array}{l@{~}l@{~}c@{\colon}c@{\colon}c@{\colon}c@{\colon}c@{\colon}c@{\colon}c@{~}l}
\xi_b\c\T^2\to Z\subset\P^6,\quad (s,u)\mapsto & ( & 1 & s & s^{-1} & u & u^{-1} & su^{-1} & s^{-1}u & ) =\\
                                               & ( & y_0 & y_1 & y_2 & y_3 & y_4 &  y_7 & y_8 & ).
\end{array}
\]
Let the projective isomorphism $\mu_2\c\P^6\to\P^6$ be 
a restriction of $\mu_2$ as defined in \LEM{invo-P8}.
We find that $X:=\mu_2(Z)$ is contained in $\S^5$ and has celestial type $(3,6,5)$ (see \FIG{dP6}).
The center of the linear projection $\rho\c\P^8\dto\P^6$
is a line that intersects $\Ys$ transversaly in
$p=(0:0:0:0:0:1:0:0:0)$ and its complex conjugate $\overline{p}=\sigma_2(p)=(0:0:0:0:0:0:1:0:0)$.
We remark that $\sigma_2\c\Ys\to\Ys$ is inner automorphic to $\sigma_0$ via $\aut(\Ys)$ by \COR{L}a.
The projection $\rho$ realizes a blowup of $\P^1\times\P^1\cong\Ys$ with centers $p$ and $\overline{p}$
as in \TAB{P1P1}b.
Notice that
$
\autc(Z)\cong\set{ \varphi\in\autc(\P^1\times\P^1) 
}{ 
\varphi(p)=p,~ \varphi(\overline{p})=\overline{p}}
$, since $\rho$ is an isomorphism almost everywhere.
Recall that $\autc(\P^1\times\P^1)\cong\autc(\P^1)\times\autc(\P^1)$ by \LEM{autP1P1}
and thus $\pi_1(\autc(Z))\cong\pi_2(\autc(Z))\cong \PSO(2)$ as it is stated in \TAB{P1P1}b.
\END
\end{example}

% \begin{example}
% \textrm{\bf(toric projection to horn cyclide)}
% \label{exm:proj-horn}
% \\
% Suppose that $\Ys\subset\P^8$ and $Z\subset\P^4$ have 
% lattice types as in \TAB{L2}a and  \TAB{L}g, \resp.
% The center of the linear projection $\Mdashrow{\rho}{\P^8}{\P^4}$
% is a 3-space that intersects $\Ys$ tangentially at two complex conjugate points $p$ and $\overline{p}$.
% This is diagrammatically represented in \TAB{P1P1}f by infinitely near points $q$ and $\overline{q}$.
% Let $N(Z)=\bas{\l_0,\l_1,\p_1,\p_2,\p_3,\p_4}_\Z$ be the Neron-Severi lattice of $Z$,
% where the nonzero intersections of the generators are $\l_0\cdot\l_1=1$ and $\p_i^2=-1$ for $1\leq i\leq 4$.
% The classes $\l_0$ and $\l_1$ correspond to the classes 
% of the fibers of the projections $\Mrow{\pi_i}{\P^1\times\P^1}{\P^1}$ for $i\in\{1,2\}$.
% We assume that $\p_1$, $\p_2$, $\p_3$ and $\p_4$
% correspond to exceptional divisors with center $p$, $\overline{p}$, $q$ and $\overline{q}$ 
% in \TAB{P1P1}f, \resp. 
% Thus we find that the effective (-2)-classes are 
% $\l_0-\p_1-\p_3$, 
% $\l_0-\p_2-\p_4$, 
% $\p_1-\p_3$, 
% $\p_2-\p_4$
% and $\l_1-\p_1-\p_2$ (see also \citep[Lemma~4]{nls-fam-circles}). 
% The real (vertical) fiber corresponding to $\l_1-\p_1-\p_2$ is contracted 
% to a real isolated singularity of type $\underline{A_3}$. 
% The complex conjugate (horizontal) fibers with classes $\l_0-\p_1-\p_3$ and $\l_0-\p_2-\p_4$
% are contracted to two complex conjugate isolated singularities of type $A_1$. 
% It follows that $\bS(Z)=\underline{A_3}+2A_1$ conform with \TAB{P1P1}.
% \END
% \end{example}

\section{Invariant quadratic forms on \texorpdfstring{$\P^1\times\P^1$}{P1xP1}}
\label{sec:pair}

In this section we reformulate 
the problem of classifying M\"obius automorphism groups of celestial surfaces,
into the problem of finding invariant quadratic forms 
of given signature in a vector space.

Suppose that $Y\subset\P^m$ is a surface \st $\aut(Y)\subset \aut(\P^m)$. 
For example, $Y\subset\P^8$ is the double Segre surface or $Y\subset\P^5$ is the Veronese surface.
Suppose that we have a birational linear projection with $m\geq n+1\geq 3$:
\[
\rho\c Y\subset\P^m\dto X\subset\S^n\subset\P^{n+1}.
\]
The \df{M\"obius pair} of $X$ \wrt $\rho$ is defined as 
\[
(Y,Q)\quad\text{where}\quad Q\subset \P^m\quad\text{is the Zariski closure of}\quad \rho^{-1}(\S^n).
\]
Notice that $Q$ is a hyperquadric of signature $(1,n+1)$ 
\st $Y\subset Q$ and \st the singular locus of $Q$ coincides with the 
center of the linear projection~$\rho$.
We define the following equivalence relation on M\"obius pairs:
\[
(Y,Q)\sim (Y,Q')~~:\Leftrightarrow~~
\exists\,\varphi\in\aut(\P^m)\c~~ \varphi(Y)=Y \quad\text{and}\quad \varphi(Q)=Q'.
\]
Suppose that $G\subseteq \autc(Y)$ is a subgroup. 
The vector space of \df{$G$-invariant quadratic forms} in the ideal $I(Y)$ of $Y$
is defined as
\[
I_2^G(Y):=\bas{ q\in I_2(Y) ~|~ q\circ \varphi = q \text{ for all } \varphi\in G }_\C,
\]
where we assume that $\varphi\in G\subset \PSL(m+1)$ is normalized to have determinant one.
Notice that the real structure $\sigma\c Y\to Y$ induces an
antiholomorphic involution on $I_2^G(Y)$.
We denote the \df{zeroset} of a form $q\in I(Y)$ by $V(q)$. 

\begin{proposition}
\textrm{\bf(properties of M\"obius pairs)}
\label{prp:moeb}
\\
Let $(Y,Q)$ and $(Y,Q')$ 
be the M\"obius pairs of surfaces
$X\subset\S^n$ and $X'\subset\S^n$, \resp.
\begin{Mlist}
\item[\bf a)]
There exists $\alpha\in\aut(\S^n)$ with $\alpha(X)=X'$
\Iff
$(Y,Q)$ and $(Y,Q')$ are equivalent.
In particular, we have that
\[
\bM(X)\cong\set{\varphi\in\autc(\P^m)}{\varphi(Y)=Y \text{ and } \varphi(Q)=Q}. 
\]

\item[\bf b)]
The subgroup $G\subseteq \autc(Y)$ is isomorphic to a subgroup of $\bM(X)$
\Iff $Q=V(q)$ for some $q\in I_2^G(Y)$.

\item[\bf c)]
If $G,G'\subset\autc(Y)$ are inner automorphic subgroups 
and $q\in I_2^{G}(Y)$, then there exists $q'\in I_2^{G'}(Y)$ \st $(Y,V(q))$ and $(Y,V(q'))$ 
are equivalent as M\"obius pairs.

\end{Mlist}
\end{proposition}

\newpage
\begin{proof}
a)
Let $\rho\c\P^m\dto\P^{n+1}$ 
be a birational linear projection \st $\rho(Q)=\S^n$ and $\rho(Y)=X$. 
Similarly, let $\rho'\c\P^m\dto\P^{n+1}$ be \st 
$\rho'(Q')=\S^n$ and $\rho'(Y)=X'$. 

$\Rightarrow$:
We show that there exists $\varphi \in \aut(\P^m)$ \st the following diagram 
commutes:
\[
\begin{tikzcd}
Y\arrow[d,"\rho"']\arrow[r,hook] & Q\arrow[d,"\rho"']\arrow[r,"\varphi"] & Q'\arrow[d,"\rho'"]\arrow[r,hookleftarrow] & Y\arrow[d,"\rho'"]\\ 
X\arrow[r,hook] & \S^n\arrow[r,"\alpha"] & \S^n\arrow[r,hookleftarrow] & X'
\end{tikzcd}
\]
If $m=n+1$, then $\rho$ and $\rho'$ are projective isomorphisms and the claim follows immediately.
If $m>n+1$, then the centers of the linear projections $\rho$ and~$\rho'$ coincide
with the singular loci $\bS(Q)$ and $\bS(Q')$ of $Q$ and $Q'$, \resp. 
Let $\Lambda,\Lambda'\subset Y$ be the centers of the projections $\rho|_Y$ and $\rho'|_Y$, \resp.
The linear isomorphism $\alpha$ induces via the projections $\rho$ and $\rho'$ the
algebraic isomorphisms
$\varphi|_Q\c Q\setminus\bS(Q)\to Q'\setminus\bS(Q')$
and
$\varphi|_Y\c Y\setminus\Lambda\to Y\setminus\Lambda'$.
The automorphism $\alpha$ leaves the union of exceptional curves that contract to points in $\Lambda$ invariant
and thus we can extend $\varphi|_Y$ so that $\varphi\in\autc(Y)$ and $\varphi(\Lambda)=\Lambda'$.
Since $\aut(Y)\subset\aut(\P^m)$ 
and since  $Q$ contains $\rho^{-1}(\S^n)$
by assumption, we find that 
$\varphi\in\aut(\P^m)$ \st $\varphi(Q)=Q'$ and $\varphi(Y)=Y$ as was to be shown.

$\Leftarrow$:
For the converse, we need to show that for given $\varphi \in \aut(\P^m)$ there exists $\alpha\in\aut(\S^n)$ \st the above diagram commutes.
This is immediate, since we define $\alpha$ as the composition $\rho'|_{Q'}\circ \varphi\circ(\rho|_Q)^{-1}$.

The remaining assertion follows if we set $Q':=Q$ and $X':=X$ in the above diagram.

b)
We first show the $\Leftarrow$ direction. By assumption $q\circ\varphi=q$ for all $\varphi\in G$.
Since 
$
\varphi^{-1}(V(q))
=
\set{\varphi^{-1}(x)\in\P^m}{ q(x)=0}
=
V(q\circ\varphi)
$
we find that $G\subseteq \set{\varphi\in\autc(\P^m)}{\varphi(Y)=Y \text{ and } \varphi(Q)=Q}$. 
It now follows from  a) that $G$ embeds as a subgroup into $\bM(X)$.
For the $\Rightarrow$ direction we again apply the characterization of $\bM(X)$ in a)
and find that $\varphi^{-1}(Q)=Q$ and thus $q\circ\varphi=q$ for all $\varphi\in G$ so that $q\in I_2^G(Y)$.

c)
Since $q\in I_2^G(Y)$ the following holds for all $\varphi\in G$ and $\alpha\in\aut(Y)$:
\[
q\circ \varphi=q 
~~\Leftrightarrow~~
q\circ \varphi\circ\alpha =q\circ\alpha
~~\Leftrightarrow~~
q\circ\alpha\circ\alpha^{-1}\circ\varphi\circ\alpha =q\circ\alpha
.
\]
By assumption $G'=\alpha^{-1}\circ G\circ \alpha$ for some $\alpha\in\aut(Y)$
and thus for all $\varphi'\in G'$ there exists $\varphi\in G$
\st $\varphi'=\alpha^{-1}\circ\varphi\circ\alpha$.
It follows that $q'\circ\varphi'=q'$ for all $\varphi'\in G'$, 
where $q':=q\circ\alpha$ so that $q'\in I_2^{G'}(Y)$. 
Thus 
\[
\alpha^{-1}(V(q))
=\set{\alpha^{-1}(x)\in\P^m}{ q(x)=0}
=V(q\circ\alpha)
=V(q')
,
\]
so that $(Y,V(q))$ is equivalent to $(Y,V(q'))$.
\end{proof}

The following theorem is essentially 
\citep[Theorem~2.5]{sch10} and allows us to compute $G$-invariant 
quadratic forms via the \df{Lie algebra} $\lie(G)$ of $G$.

\begin{theoremext}
{\bf [DeGraaf-P{\'{\i}}lnikov{\'a}-Schicho, 2009]}
\label{thm:D}
\\
Suppose that $Y\subset\P^{m-1}$ is a variety \st $\autc(Y)\subset \autc(\P^{m-1})$.
Let the 1-parameter subgroup $H\subset \autc(Y)$
be represented by an $m\times m$ matrix whose entries are smooth functions
in the parameter $\alpha$ \st $\det H(\alpha)=1$ for all $\alpha$
and \st $H(0)$ is the identity matrix.
Let the $m\times m$ matrix~$D$ in $\lie(H)$ be the tangent vector $(\partial_\alpha H)(0)$ of $H$ at the identity.
Then the $H$-invariant quadratic forms are
\begin{equation}
I_2^H(Y)=\bas{ q_A\in I_2(Y) ~|~ D^T\cdot A + A\cdot D=0}_\C,
\end{equation}
where $q_A$ denotes the quadratic form $x^\top\cdot A\cdot x$ associated to 
the symmetric $m\times m$ matrix $A$.
\end{theoremext}

\begin{proof}
We observe that
$I_2^H(Y)=\bas{ q_A\in I_2(Y) ~|~ H^\top\cdot A\cdot H = A }_\C$.
Let us first assume that $H^\top\cdot A\cdot H = A$. 
We differentiate both sides of the equation \wrt $\alpha$
and evaluate at $0$ so that we obtain the necessary condition $D^\top\cdot A + A\cdot D=0$.
For the converse we assume that $D^\top\cdot A + A\cdot D=0$.
Thus $G^\top\cdot A\cdot G = B$ for some matrix $B$ where $G=\exp(\alpha D)$ is a one-parameter subgroup. 
We differentiate both sides of the equivalent equation $G^\top\cdot A\cdot G\cdot \exp(\alpha) = B\cdot \exp(\alpha)$ 
\wrt $\alpha$ and evaluate at $0$ so that we obtain $D^\top\cdot A + A\cdot D + A = B$.
It follows that $A=B$ and we know from Lie theory that $H=G$ so that $H^\top\cdot A\cdot H = A$ as is required.
\end{proof}

\begin{remark}
\textrm{\bf(goal)}
\label{rmk:goal}
Our goal is to classify subgroups 
$G\subseteq \autc(Y)$ up to inner automorphism
\st $\dim G\geq 2$
and $I_2^G(Y)$ contains quadratic forms~$q$
of signature $(1,n+1)$ with $n\geq 3$.
It follows from \PRP{moeb} that 
the M\"obius pairs $(Y,V(q))$ for such $q$, 
correspond to the celestial surfaces $X\subset\S^n$ \st $G$ is isomorphic to 
a subgroup of $\bM(X)$.
% Suppose that $G,G'\subseteq \autc(Y)$ are inner automorphic subgroups
% so that $G'=\alpha^{-1}\circ G\circ \alpha$ for some $\alpha\in\aut(Y)$.
% Notice that the surface $X\subset\S^n$ corresponding to $q\in I_2^G(Y)$
% is M\"obius equivalent to $X'\subset\S^n$
% corresponding to $q\circ \alpha\in I_2^{G'}(Y)$.
\END
\end{remark}

\section{Automorphisms of \texorpdfstring{$\P^1\times\P^1$}{P1xP1}}
\label{sec:lie}

Motivated by \RMK{goal} with $Y$ the double Segre surface $\Ys\cong\P^1\times\P^1$, 
we would like to classify Lie subgroups of $\autc(\P^1\times\P^1)$ up to inner automorphism.
By \THM{D}, it is sufficient to classify Lie subalgebras of $\Msl_2\oplus\Msl_2$.
%and \citep[Section 5.9]{hab1}

Let us first investigate real structures of $\Msl_2\oplus\Msl_2$.
Consider the toric involutions $\sigma_i\c\T^2\to\T^2$ 
in \LEM{invo-T2} with $0\leq i\leq 3$.
By \LEM{invo-P8}, these toric involutions induce involutions
on related algebraic structures (recall \NTN{nota}): 
\quad
$\sigma_i\c\P^1\times\P^1\to\P^1\times\P^1$,
\quad
$\sigma_i\c\autc(\P^1\times\P^1)\to\autc(\P^1\times\P^1)$
\quad and \quad
$\sigma_i\c\Msl_2\oplus\Msl_2\to\Msl_2\oplus\Msl_2$. 
% We know from \LEM{invo-P1P1}b that, up to inner automorphism in $\aut(\P^1\times\P^1)$, the only real involution on 
% $\P^1\times\P^1$ that leaves a 2-dimensional set of real points invariant is $\sigma_+\times\sigma_+$.

% Celestial surfaces of degrees 6 and 4 that are not $\infty$-circled
% are blowups of $\P^1\times\P^1$ in either 2 or 4 points.
% The automorphism group of such blowups are isomorphic to subgroups of $\autc(\P^1\times\P^1)$
% that stabilize the center of the blowup. 
% Real structures of $\P^1\times\P^1$ that are
% inner automorphic to $\sigma_+\times\sigma_+$ in \LEM{invo-P1P1}b
% are not necessarily inner automorphic in such a subgroup of $\autc(\P^1\times\P^1)$.

\begin{lemma}
\textrm{\bf(real structures for $\Msl_2\oplus\Msl_2$)}
\label{lem:invo-sl2sl2}
\\
If
{\footnotesize%
\[
m:=
\left(
\begin{bmatrix}
a&b\\c&d 
\end{bmatrix} 
,
\begin{bmatrix}
e&f\\g&h 
\end{bmatrix} 
\right)\in \Msl_2\oplus\Msl_2,
\]}%
then
{\footnotesize%
\begin{gather*}
\sigma_0(m)=
\left(
\begin{bmatrix}
\overline{a}&\overline{b}\\
\overline{c}&\overline{d} 
\end{bmatrix} 
,
\begin{bmatrix}
\overline{e}&\overline{f}
\\
\overline{g}&\overline{h} 
\end{bmatrix} 
\right)
,\quad
\sigma_1(m)=
\left(
\begin{bmatrix}
\overline{d}&\overline{c}\\
\overline{b}&\overline{a} 
\end{bmatrix} 
,
\begin{bmatrix}
\overline{e}&\overline{f}
\\
\overline{g}&\overline{h} 
\end{bmatrix} 
\right)
\\[2mm]
\sigma_2(m)=
\left(
\begin{bmatrix}
\overline{d}&\overline{c}\\
\overline{b}&\overline{a} 
\end{bmatrix} 
,
\begin{bmatrix}
\overline{h}&\overline{g}
\\
\overline{f}&\overline{e} 
\end{bmatrix} 
\right)
,\quad
\sigma_3(m)=
\left(
\begin{bmatrix}
\overline{e}&\overline{f}
\\
\overline{g}&\overline{h} 
\end{bmatrix} 
,
\begin{bmatrix}
\overline{a}&\overline{b}\\
\overline{c}&\overline{d} 
\end{bmatrix} 
\right),
\end{gather*}}%
where $\sigma_0$, $\sigma_1$, $\sigma_2$ and $\sigma_3$
are real structures $\Msl_2\oplus\Msl_2\to\Msl_2\oplus\Msl_2$
induced by the corresponding involutions in \LEM{invo-T2} and \LEM{invo-P8}.
\end{lemma}

\begin{proof}
Suppose that $M\subset\autc(\P^1\times\P^1)$ is a 1-parameter subgroup
such that $m$ is the tangent vector of $M$ at the identity.
The 1-parameter subgroup $\sigma_i(M)$ has tangent vector 
$\sigma_i(m)$ for all $0\leq i\leq 3$.
We compute the representation $\cS(M) \in\aut(\P^8)$ using \EQN{S},
where the entries of $M$ are set as indeterminates.
Let $L_i$ denote the $9\times 9$ permutation matrix
corresponding to the induced antiholomorphic involution
$\sigma_i\c\P^8\to\P^8$
as stated in \LEM{invo-P8}.
It is immediate to verify that
$\overline{L_i^{-1}\circ\cS(M)\circ L_i}=\cS(\sigma_i(M))$,
where $\sigma_i$ acts on $\autc(\P^1\times\P^1)$ as an 
involution and $\overline{\,\cdot\,}$ denotes complex conjugation. 
We conclude that 
the lemma holds, since the action of~$\sigma_i$ on~$m$ is the same as the action 
of~$\sigma_i$ on~$M$.
\end{proof}

\newpage
\begin{remark}
The real structure $\sigma_2\c\Msl_2\oplus\Msl_2\to\Msl_2\oplus\Msl_2$
is inner automorphic to
$\sigma_H:=\alpha\circ\sigma_2\circ \alpha^{-1}$, where
\[
\alpha
:=
\begin{bmatrix}
 \Mi &  0\\ 
 0 & -\Mi
\end{bmatrix}
\text{so that }
\sigma_H
\left(
\begin{bmatrix}
a&b\\c&d 
\end{bmatrix} 
,
\begin{bmatrix}
e&f\\g&h 
\end{bmatrix} 
\right)
=
\left(
\begin{bmatrix}
\overline{d}&-\overline{c}\\
-\overline{b}&\overline{a} 
\end{bmatrix} 
,
\begin{bmatrix}
\overline{h}&-\overline{g}
\\
-\overline{f}&\overline{e} 
\end{bmatrix} 
\right).
\]
The Lie algebra $\Msl_2\oplus\Msl_2$ with real structure $\sigma_H$
is usually denoted by $\Msu_2\oplus\Msu_2$.
The real elements in $\Msu_2$ are skew Hermitian matrices.
Similarly, $\Msl_2\oplus\Msl_2$ with real structure $\sigma_1$ 
can be identified with $\Msu_2\oplus\Msl_2(\R)$.
\END
\end{remark}

\begin{notation}
\label{ntn:lie}
We consider the following elements in $\Msl_2$:
\[
t:=
\begin{bmatrix}
0&1\\0&0 
\end{bmatrix}
,~~
q:=
\begin{bmatrix}
0&0\\1&0 
\end{bmatrix}
,~~
s:=
\begin{bmatrix}
1&0\\0&-1 
\end{bmatrix}
,~~
r:=
\begin{bmatrix}
0&-1\\1&0 
\end{bmatrix}
,~~
e:=
\begin{bmatrix}
0&0\\0&0 
\end{bmatrix}.
\]
Recall that $\Msl_2$ over the complex numbers is generated by $\bas{t,q,s}$, 
where the Lie brackets of the generators are $[t,q]=s$, $[t,s]=-2t$ and $[q,s]=2q$.
We shall denote $(g,e)\in \Msl_2\oplus\Msl_2$ by $g_1$ and 
$(e,g)\in \Msl_2\oplus\Msl_2$ by $g_2$ for all $g\in\Msl_2$. 
Notice that $\Msl_2\oplus\Msl_2=\bas{t_1,q_1,s_1, t_2,q_2,s_2}$ 
where the Lie bracket acts componentwise.~\END
\end{notation}

\begin{remark}
%\textrm{\bf(generators of Lie algebra)}
\label{rmk:naming}
Suppose that the real structure of $\Msl_2\oplus\Msl_2$
is defined by $\sigma_0$ in \LEM{invo-sl2sl2}.
In this case, 
\begin{gather*}
\lie(\PSE(1))=\bas{t},\qquad 
\lie(\PSX(1))=\bas{s},\qquad 
\lie(\PSO(2))=\bas{r},\qquad
\\
\lie(\PSA(1))=\bas{t,s} \qquad\text{and}\qquad
\lie(\PSL(2))=\bas{t,q,s}.
\end{gather*}
These groups correspond to 
translations, 
scalings, 
rotations, 
affine transformations, and projective transformations, \resp. 
Indeed the generators are the tangent vectors 
at the identity of the following 1-parameter subgroups:
\[
t\leftrightsquigarrow
\begin{bmatrix}
1&\alpha\\
0&1 
\end{bmatrix}
,~~
s\leftrightsquigarrow
\begin{bmatrix}
\alpha+1 &0\\0& (\alpha+1)^{-1}
\end{bmatrix}
,~~
r\leftrightsquigarrow
\begin{bmatrix}
\cos(\alpha) & -\sin(\alpha)\\
\sin(\alpha) & \cos(\alpha)
\end{bmatrix}. 
\]
Now suppose that the real structure of $\Msl_2\oplus\Msl_2$
is defined by $\sigma_2$ in \LEM{invo-sl2sl2}.
In this case $\lie(\PSO(2))=\bas{\Mi s}$, since
\[
\Mi s
=
\begin{bmatrix}
\Mi & 0\\
0 & -\Mi
\end{bmatrix} 
\leftrightsquigarrow
\begin{bmatrix}
\cos(\alpha)+\Mi\sin(\alpha) & 0\\
0 & \cos(\alpha)-\Mi\sin(\alpha)
\end{bmatrix}. 
\]
See also \EXM{toric}.
\END
\end{remark}

Suppose that $F$ is a Lie group.
We call two Lie subalgebras $\Mg,\Mh\subset\lie(F)$ 
\df{(complex) inner automorphic} if there exists (complex) $M\in F$ \st $\Mg=M^{-1}\cdot \Mh\cdot M$.
\THM{sl2+sl2} and thus \COR{sl2+sl2} follow from \cite{so4}.

\begin{theoremext}
\label{thm:sl2+sl2}
{\bf [Douglas-Repka, 2016]}  
\\
A Lie subalgebra $0\subsetneq \Mg\subseteq \Msl_2\oplus\Msl_2$
is, up to flipping the left and right factor, complex inner automorphic to either one 
of the following with $\alpha\in\C^*$:
{\footnotesize%
\begin{gather*}
\bas{t_1},~
\bas{s_1},~
\bas{t_1+t_2},~
\bas{t_1+s_2},~ 
\bas{s_1+\alpha s_2},~ 
\bas{t_1, s_1},~
\bas{t_1, t_2},~ 
\bas{t_1, s_2},~ 
\bas{s_1, s_2},~\\ 
\bas{s_1+t_2,t_1},~
\bas{t_1+t_2,s_1+s_2},~
\bas{s_1+\alpha s_2,t_1},~
\bas{t_1,q_1,s_1},~
\bas{t_1,s_1,t_2},~
\bas{t_1,s_1,s_2},~\\
\bas{s_1+\alpha s_2, t_1, t_2 },~
\bas{t_1+t_2,q_1+q_2,s_1+s_2},~
\bas{t_1,s_1,t_2,s_2},~
\bas{t_1,q_1,s_1,t_2},~
\bas{t_1,q_1,s_1,s_2},~\\
\bas{t_1,q_1,s_1,t_2,s_2},~
\bas{t_1,q_1,s_1,t_2,q_2,s_2}.
\end{gather*}}%
\end{theoremext}

%\vspace{-2mm}
%We obtain \COR{sl2+sl2} as a direct consequence of \THM{sl2+sl2}.
%and we will use this corollary in \LEM{M}. 

\begin{corollary}
\label{cor:sl2+sl2}
If $\Mg\subseteq \Msl_2\oplus\Msl_2$ is a Lie subalgebra \st $\dim\Mg\geq 2$ and $\Mg$ is not complex inner automorphic to 
$\bas{s_1,s_2}$, then $\Mg$ contains a subalgebra that is complex inner automorphic to either $\bas{t_1}$, $\bas{t_2}$ or $\bas{t_1+t_2}$.
\end{corollary}

\section{The classification of \texorpdfstring{$\P^1\times\P^1$}{P1xP1}}
\label{sec:iqf-segre}

In a perfect world we directly use
\THM{D} to compute for each Lie subalgebra $\lie(G)\subseteq\Msl_2\oplus\Msl_2$, 
the vector space $I_2^G(\Ys)$ generated by $G$-invariant quadratic forms
on the double Segre surface $\Ys\cong\P^1\times\P^1$.
We would then proceed by classifying quadratic forms 
in $I_2^G(\Ys)$ of signature $(1,n+1)$ as was suggested in \RMK{goal}. 

Unfortunately, there are two problems.
\THM{sl2+sl2} only provides the classification of subalgebras of $\Msl_2\oplus\Msl_2$ up to complex inner automorphisms
and thus the real structure is not preserved.
The second problem is that it is in general difficult to classify quadratic forms in $I_2^G(\Ys)$ of fixed signature.
For example, for what $n\geq 3$ do there exist quadratic forms of signature $(1,n+1)$ in the vector space $I_2^G(X_1)$ 
at \LEM{iqf}c?

\LEM{M} plays a crucial role in circumventing these two problems using geometric arguments.
We are able to prove \LEM{M1}, since the invariant quadratic forms in \LEM{iqf}a 
have a particularly nice basis.
This section will end with a proof for \THM{M} and \COR{iqf}.
In particular, we will see that the answer to the 
question in the previous paragraph is $n\in \{3\}$.

\begin{lemma}
\textrm{\bf(invariant quadratic forms for $\P^1\times\P^1$)}
\label{lem:iqf}
\\
Let the real structures $\sigma_0$, $\sigma_1$, $\sigma_2$ and $\sigma_3$ for Lie algebras 
be as defined in \LEM{invo-sl2sl2}.
Let $\Ys$, $X_0$, $X_1$, $X_2$ and $X_3$ be the double Segre surfaces in $\P^8$
as defined in \LEM{invo-P8}.
We suppose that $G$ is a Lie subgroup of $\autc(\Ys)$.
\begin{Mlist}

\item[\bf a)]  
If $\lie(G)=\bas{\Mi s_1,\Mi s_2}$ with real structure $\sigma_2$, then
$G\cong \PSO(2)\times \PSO(2)$,
{\footnotesize%
\begin{gather*}
I_2^G(\Ys)=
\left\langle
y_0^2-y_1y_2,~
y_0^2-y_3y_4,~
y_0^2-y_5y_6,~
y_0^2-y_7y_8
\right\rangle_\C 
,\text{ and}\\
I_2^G(X_2)=
\left\langle
\tfrac{1}{4}x_0^2-x_1^2-x_2^2,~
\tfrac{1}{4}x_0^2-x_3^2-x_4^2,~
\tfrac{1}{4}x_0^2-x_5^2-x_6^2,~
\tfrac{1}{4}x_0^2-x_7^2-x_8^2
\right\rangle_\C.
\end{gather*}}%

\item[\bf b)]
If $\lie(G)=\bas{\Mi s_1, s_2}$ with real structure $\sigma_1$, then $G\cong \PSO(2)\times \PSX(1)$,
{\footnotesize%
\begin{gather*}
I_2^G(\Ys)=
\left\langle
y_0^2-y_1y_2,~
y_0^2-y_3y_4,~
y_0^2-y_5y_6,~
y_0^2-y_7y_8
\right\rangle_\C
,\text{ and}\\
I_2^G(X_1)=
\left\langle
x_0^2 - x_1^2 - x_2^2,~  
x_0^2 - x_3x_4,~        
x_5x_6 - x_7x_8,~ 
x_0^2 - x_5x_7 - x_6x_8
\right\rangle_\C 
.
\end{gather*}}%
 
\item[\bf c)]
If $\lie(G)=\bas{\Mi s_1, t_2}$ with real structure $\sigma_1$, then $G\cong \PSO(2)\times \PSE(1)$,
{\footnotesize%
\begin{gather*}
I_2^G(\Ys)=
\left\langle~
y_0^2 - y_3y_4 ,~ 
y_4^2 - y_6y_7 ,~ 
y_1y_6 - y_2y_7,~ 
2y_1y_2 - y_5y_6 - y_7y_8
~\right\rangle_\C
,\text{ and}\\
I_2^G(X_1)=
\left\langle~
x_0^2 - x_3x_4,~
x_4^2 - x_6^2 - x_7^2,~ 
x_1x_6 - x_2x_7,~ 
x_1^2 + x_2^2 - x_5x_7 - x_6x_8 
~\right\rangle_\C. 
\end{gather*}}%

\item[\bf d)]
If $\lie(G)=\bas{t_1, q_1, s_1, t_2, q_2, s_2}$ with real structure either $\sigma_0$ or $\sigma_3$, 
then $G\cong \PSL(2)\times \PSL(2)$,
{\footnotesize%
\begin{gather*}
I_2^G(\Ys)=
\langle~
2y_0^2 - 2y_1y_2 - 2y_3y_4 + y_5y_6 + y_7y_8
~\rangle_\C
,\\
I_2^G(X_0)=
\langle~
2x_0^2 - 2x_1x_2 - 2x_3x_4 + x_5x_6 + x_7x_8
~\rangle_\C
,\text{ and}\\
I_2^G(X_3)=
\langle~
2x_0^2 - 4x_2x_3 - 4x_1x_4 + x_5x_6 + x_7^2 + x_8^2
~\rangle_\C
\end{gather*}}%
\end{Mlist}
\end{lemma}

%     ----------
%      involution        =identity
%      group             =< t1+t2, g1+q2, s1+s2 >
%      invariant ideal   = <
%         (-2)*x0^2 + 2*x1*x2 + 2*x3*x4 - x5*x6 - x7*x8       [4, 5]
%         -x0^2 + x1*x2 - x2*x3 - x1*x4 + x3*x4 + x0*x7 + x0*x8 - x7*x8       [2, 2]
%          >
%      random signatures =
%          [[2, 2], [4, 5]]
%     ----------
%
%     ----------
%      involution        =identity
%      group             =< t1+t2, s1+s2 >
%      invariant ideal   = <
%         (-2)*x0^2 + 2*x1*x2 + 2*x3*x4 - x5*x6 - x7*x8       [4, 5]
%         -x0^2 + x1*x2 - x2*x3 - x1*x4 + x3*x4 + x0*x7 + x0*x8 - x7*x8       [2, 2]
%          >
%      random signatures =
%          [[2, 2], [4, 5]]
%     ----------
%
%     ----------
%      involution        =leftright
%      group             =< t1, s1 >
%      invariant ideal   = <
%         x4^2 - x6^2 - x7^2      [1, 2]
%         2*x3*x4 + (-2)*x5*x7 + (-2)*x6*x8       [3, 3]
%         2*x0*x4 + (-2)*x2*x6 + (-2)*x1*x7       [3, 3]
%         x3^2 - x5^2 - x8^2      [1, 2]
%         2*x0*x3 + (-2)*x1*x5 + (-2)*x2*x8       [3, 3]
%         x0^2 - x1^2 - x2^2      [1, 2]
%          >
%      random signatures =
%          [[1, 2], [2, 4], [3, 3], [3, 6], [4, 5]]
%         [1, 2]

\begin{proof}
a)
We know from \RMK{naming} that $G\cong \PSO(2)\times \PSO(2)$, since
$\Mi s_1$ and $\Mi s_2$ are the tangent vectors of the following two 1-parameter subgroups
of $\autc(\P^1\times\P^1)$:
\begin{center}
\scriptsize
$
\left(
\begin{bmatrix}
\cos(\alpha)+\Mi\sin(\alpha) & 0\\
0 & \cos(\alpha)-\Mi\sin(\alpha)
\end{bmatrix}
,
\begin{bmatrix}
1 & 0\\
0 & 1
\end{bmatrix}
\right)
\& %,~%{\normalsize and}
\left(
\begin{bmatrix}
1 & 0\\
0 & 1
\end{bmatrix}
,
\begin{bmatrix}
\cos(\alpha)+\Mi\sin(\alpha) & 0\\
0 & \cos(\alpha)-\Mi\sin(\alpha)
\end{bmatrix}
\right).
$
\end{center}
Via the map $\phi\c S^1\times S^1\to \autc^\T(\Ys)\subset\aut(\P^8)$
from \EXM{toric}, we obtain 1-parameter subgroups $H_1$ and $H_2$ of $\autc(\P^8)$.
We use \THM{D} to compute the vector spaces 
$I_2^{H_1}(\Ys)$ and $I_2^{H_2}(\Ys)$
of invariant quadratic forms.
Since $\lie(G)=\bas{\Mi s_1,\Mi s_2}$ we have $I_2^G(\Ys)=I_2^{H_1}(\Ys)\cap I_2^{H_2}(\Ys)$.
In order to compute $I_2^G(X_2)$ we compose the generators of $I_2^G(\Ys)$ with $\mu_2$ from \LEM{invo-P8}.
The proofs of b), c) and d) are similar.
The invariant quadratic forms can be computed automatically with \citep[\texttt{moebius-aut}]{maut}.
\end{proof}

\begin{remark}
\label{rmk:sigma3}
It follows from \RMK{invo-P8} that $X_3$ is a 2-uple embedding of~$\S^2$ into~$\P^8$.
Hence, $X_3$ cannot be projectively equivalent to a celestial surface. 
We nevertheless included the real structure $\sigma_3\c\Ys\to\Ys$ 
at \LEM{iqf}d in order to prove \COR{iqf}.
If $G\cong\PSL(2)\times\PSL(2)$, 
then 
$\sigma_3\c\Msl_2\oplus\Msl_2\to\Msl_2\oplus\Msl_2$,
and thus 
$\sigma_3\c G\to G$,
is specified in \LEM{invo-sl2sl2}.
Notice that in this case $\set{\varphi\in G}{\sigma_3(\varphi)=\varphi}$
is three-dimensional and not six-dimensional.
\END
\end{remark}

% \begin{definition}
% \textrm{\bf(models for Euclidean similarities)}
% \label{def:euclid}
% \\
% We fix a point $c\in\S^n$ which we refer to as the 
% \df{point at infinity}. 
% The \df{M\"obius model for Euclidean similarities} \wrt $c$ is defined as
% \[
% \Psi_c:=\set{\varphi\in\aut(\S^n)}{\varphi(c)=c}.
% \]
% Let $\pi_c\c\S^n\dto\P^n$ denote the stereographic projection 
% with center at $c$ (see also \citep[Section~2]{nls-fam-circles}).
% The birational inverse $\pi_c^{-1}\c\P^n\dto\S^n$ 
% is not defined at the \df{absolute quadric} $A\subset\P^n$,
% which is a quadric of codimension two.
% For example, if $c=(1:0:\ldots:0:1)$, then
% $A=\set{x\in\P^n}{x_0=x_1^2+\ldots+x_n^2=0}$.
% The \df{Cayley-Klein model for Euclidean similarities} is defined as
% \[
% \Phi_A:=\set{\varphi\in\aut(\P^n)}{\varphi(A)=A}.
% \]
% Notice that $\pi^{-1}_c\circ\varphi\circ\pi_c \in \Psi_c$ for all $\varphi\in\Phi_A$.
% The \df{Euclidean model for Euclidean similarities} is defined 
% by the Euclidean isometries and scalings of 
% \[
% \R^n\cong \set{x\in\P^n}{x_0\neq 0}.
% \]
% The \df{Euclidean translations} 
% are in this model realized by vector additions.
% \END
% \end{definition}

\begin{example}
\textrm{\bf(spindle cyclide and horn cyclide)}
\label{exm:spindle-horn}
\\
Let $Z_s\subset \P^4$ be the image of 
the monomial parametrization defined by \TAB{L}f (spindle cyclide) using the left coordinates in \TAB{coord}.
Let $Z_h\subset \P^4$ denote the image of the monomial parametrization defined by \TAB{L}g (horn cyclide).
Recall from \RMK{proj}, that both $Z_s$ and $Z_h$ are toric projections of $\Ys$.
It follows from \LEM{I2} that
$I_2(Z_s)$ is generated by the generators of $I_2(\Ys)$ that do not contain $y_i$ for $i\in\{5,6,7,8\}$.
Similarly, $I_2(Z_h)$ is generated by generators of $I_2(\Ys)$ that do not contain $y_i$ for $i\in\{1,2,5,8\}$. 
Thus
\[
I_2(Z_s)=\bas{y_0^2-y_1y_2, y_0^2-y_3y_4}_\C
\quad\text{and}\quad
I_2(Z_h)=\bas{y_0^2-y_3y_4, y_4^2-y_6y_7}_\C.
\]
Let $\mu_s$ and $\mu_h$ be restrictions of $\mu_1\c\P^8\to\P^8$ in \LEM{invo-P8} as follows:
\[
\begin{array}{llcl}
\mu_s\colon\P^4\to\P^4,&(x_0: x_1: x_2: x_3: x_4)&\mapsto&(x_0:x_1+\Mi x_2:x_1-\Mi x_2:x_3:x_4),                                                                   \\[1mm]   
\mu_h\colon\P^4\to\P^4,&(x_0: x_3: x_4: x_6: x_7)&\mapsto&(x_0:x_3:x_4:x_7-\Mi x_6:x_7+\Mi x_6).                                                                   \\[1mm]
\end{array}
\]
We set 
$X_s:=(\alpha_s\circ\mu_s)^{-1}(Z_s)$ and $X_h:=(\alpha_h\circ\mu_h)^{-1}(Z_h)$, where
\[
\begin{array}{@{}l@{}l@{}c@{}l@{}}
\alpha_s\colon\P^4\to\P^4,&(x_0: x_1: x_2: x_3: x_4)&\mapsto&(x_4: x_1: x_2: {\scriptstyle\frac{1}{\sqrt{2}}}(x_0 -x_3): {\scriptstyle\frac{1}{\sqrt{2}}}(x_0 + x_3)),\\[1mm]
\alpha_h\colon\P^4\to\P^4,&(x_0: x_3: x_4: x_6: x_7)&\mapsto&({\scriptstyle\frac{1}{\sqrt{2}}}x_4: x_3: -x_0-x_3: x_6: x_7).                                          
\end{array}
\]
Notice that $X_s, X_h\subset\S^3$, since 
\quad
$I_2(X_s)=\bas{
x_1^2 + x_2^2 - x_4^2,~
x_0^2 -x_3^2 -2x_4^2
}_\C$ 
\quad
and  
\[
I_2(X_h)=\bas{
x_4^2 + 2x_0x_3 + 2x_3^2
,~ 
x_0^2+ 2x_0x_3 +  x_3^2  - x_6^2 - x_7^2 
}_\C.
\]
%It follows that 
Let $\pi^h,\pi^s\c S^3\dto \R^3$ be stereographic projections
so that the projective closures of these projections with the above coordinates are
\begin{gather*}
\tilde{\pi}^h\c\S^3\dto\P^3,\quad (x_0: x_1: x_2: x_3: x_4)\mapsto(x_0-x_3:x_1:x_2:x_4),
\\
\tilde{\pi}^s\c\S^3\dto\P^3,\quad (x_0: x_3: x_4: x_6: x_7)\mapsto(x_0+x_3:x_4:x_7:x_6).
\end{gather*}
We verify that $\pi^h(X_s(\R))$ is a circular cone and that $\pi^s(X_h)$ is a circular cylinder
so that, by \DEF{names}, $X_s$ and $X_h$ are indeed a spindle cyclide and horn cyclide, \resp.
Notice that both $\pi^h$ and $\pi^s$ have a real isolated singularity as center of projection. 
Since these isolated singular points have to be preserved by
the M\"obius automorphisms, it follows from \SEC{intro-defs} that 
$\bM(X_s)$ and $\bM(X_h)$ are subgroups of Euclidean similarities.
The circular cone and the circular cylinder are unique up to Euclidean similarities and thus 
$\bD(X_s)=\bD(X_h)=0$.
By \LEM{iqf}b the generators of the vector space $I_2(Z_s)$ are $\PSO(2)\times \PSX(1)$ invariant.
Similarly, by \LEM{iqf}c, the generators of the vector space $I_2(Z_h)$ are $\PSO(2)\times \PSE(1)$ invariant. 
\PRP{blowup} characterizes the projections from 
$\autc(X_s)$ and $\autc(X_h)$ to $\autc(\P^1)$.
We conclude that
$\bM(X_s)\cong\PSO(2)\times \PSX(1)$ and $\bM(X_h)\cong\PSO(2)\times \PSE(1)$
so that $\autc(X_s)=\bM(X_s)$ and $\autc(X_h)=\bM(X_h)$.
\END
\end{example}

\begin{lemma}
\textrm{\bf(M\"obius automorphism groups)}
\label{lem:M}
\\
If $X\subset\S^n$ is a $\lambda$-circled
celestial surface
\st $\lambda<\infty$ and $\dim \bM(X)\geq 2$,
then either 
\begin{Menum}
\item $\bM(X)\cong \PSO(2)\times \PSO(2)$ and $\lie(\bM(X))\subset\Msl_2\oplus\Msl_2$ is, up to inner automorphism, equal to $\bas{\Mi s_1,\Mi s_2}$
with real structure $\sigma_2$ in \LEM{invo-sl2sl2}, 
\item  
$\bM(X)\cong \PSO(2)\times \PSX(1)$,
$\bT(X)=(2,4,3)$, 
$\bS(X)=2\underline{A_1}+2A_1$, 
$\bD(X)=0$,
$\bM(X)=\autc(X)$
and $X$ is a spindle cyclide, or
\item 
$\bM(X)\cong \PSO(2)\times \PSE(1)$, 
$\bT(X)=(2,4,3)$, 
$\bS(X)=\underline{A_3}+2A_1$, \\
$\bD(X)=0$,
$\bM(X)=\autc(X)$
and $X$ is a horn cyclide.
\end{Menum}
\end{lemma}

\begin{proof}
In \PRP{blowup} we related $\P^1\times\P^1$ to $X$, 
via a birational linear projection $\rho\c\Ys\dto X$,
where $\Ys\cong\P^1\times\P^1$.
Recall from \RMK{sigma3} that the real structure of $\Ys$, 
that is via $\rho$ compatible with the real structure of $X$, 
cannot be~$\sigma_3$.
We know from \LEM{autP1P1} that $\autc(\P^1\times\P^1)\cong\autc(\P^1)\times\autc(\P^1)$.
The first and second projection are denoted by $\pi_i\c\P^1\times\P^1\to\P^1$ with $i\in\{1,2\}$
and we denote the projections of $\autc(\P^1)\times\autc(\P^1)$ 
to $\autc(\P^1)$ by $\pi_1$ and $\pi_2$ as well.
The automorphisms of $X$ in the identity component
factor via $\rho$ through automorphisms of $\Ys$ that leave
the center of projection $\Lambda$ invariant.
We make a case distinction on the configurations of $\Lambda$ in \TAB{P1P1}
were we identified $\Ys$ with $\P^1\times\P^1$. 
By \PRP{blowup} these are all possible configurations for $\Lambda$.

We first suppose that $\Lambda$ is the empty-set as in \TAB{P1P1}a.

We consider the action of subgroups of the M\"obius automorphism group $\bM(X)$ on $\P^1\times\P^1$. 
We start by showing that either \LEM{M}.1 holds or there exists a 
one-dimensional subgroup of $\bM(X)$
whose action on $\P^1\times\P^1$
leaves a real fiber~$L$ of $\pi_2$ invariant and leaves a real point $\hat{c}\in L$ on this fiber invariant. 
We write $\Mg \sim_\C \Mh$ and $\Mg \sim \Mh$
if Lie subalgebras $\Mg,\Mh\subset \Msl_2\oplus\Msl_2$ 
are complex and real inner automorphic, \resp.
Recall from \COR{sl2+sl2} that if $\lie(\bM(X))\nsim_\C\bas{s_1,s_2}$,
then there exists a one-dimensional Lie subgroup
$H\subset\bM(X)$ \st \Wlog either 
$\lie(H)\sim_\C \bas{t_2}$ or
$\lie(H)\sim_\C \bas{t_1+t_2}$.

Suppose that $\lie(\bM(X))\sim_\C\bas{s_1,s_2}$ \st both $\pi_1(\bM(X))$ and $\pi_2(\bM(X))$ 
leave complex conjugate basepoints invariant while acting on $\P^1$.
By \COR{L}a we may assume \Wlog that the real structure of~$\Ys$ is defined by $\sigma_2$ in \LEM{invo-P8}.
The induced real structure on $\Msl_2\oplus\Msl_2$ is as in \LEM{invo-sl2sl2}.
It follows from \RMK{naming} that $\lie(\bM(X))\sim\bas{\Mi s_1,\Mi s_2}$,
so that $\pi_i(\bM(X))$ consists of all automorphisms in $\autc(\P^1)$ 
that preserve two complex conjugate basepoints for $1\leq i\leq 2$. 
Hence, $\bM(X)\subseteq \PSO(2)\times \PSO(2)$ and
since $\dim \bM(X)\geq 2$ by assumption
this must be an inclusion of connected Lie groups of the same dimension
so that $\bM(X)\cong \PSO(2)\times \PSO(2)$. We conclude that \LEM{M}.1 holds in this case.

Suppose that $\lie(\bM(X))\sim_\C\bas{s_1,s_2}$ \st $\pi_2(\bM(X))$ 
leaves real points invariant while acting on $\P^1$.
Thus there exists a subgroup $H\subset\bM(X)$ \st $\lie(H)\sim\bas{s_2}$
and the action of $H$ on $\P^1\times\P^1$
leaves two real fibers~$L$ and~$L'$ of $\pi_2$ pointwise invariant.

Suppose that $\lie(H)\sim_\C \bas{t_2}$. 
It follows from \RMK{naming} that the action of~$H$ on~$\P^1\times\P^1$ 
leaves exactly one fiber $L:=\pi_2^{-1}(u)$ pointwise invariant for some point $u\in\P^1$. 
The number of fibers that are preserved are invariant under 
complex inner automorphisms and thus this fiber must be real
so that $\lie(H)\sim \bas{t_2}$.
% Indeed, $\pi_2(H)\subset \autc(\P^1)$ leaves up to inner automorphism
% the real point $u=(1:0)$ invariant while acting on $\P^1$.

Suppose that $\lie(H)\sim_\C \bas{t_1+t_2}$.
Analogously as before we find that $\lie(H)\sim\bas{t_1+t_2}$, since 
the action of $H$ on $\P^1\times\P^1$ 
leaves $M$ and $L$ invariant, where $M$ and $L$ are real fibers of $\pi_1$ and $\pi_2$, \resp.
Moreover, the action leaves the real point $\hat{c}\in L$ invariant \st $\{\hat{c}\}=M\cap L$.

Now suppose by contradiction that \LEM{M}.1 does not hold.
Notice that we are still in the case where $\Lambda=\emptyset$ as in \TAB{P1P1}a.
We showed that there exists a one-dimensional subgroup $H\subset\bM(X)$ 
\st either $\lie(H)\sim\bas{s_2}$, $\lie(H)\sim\bas{t_2}$ or $\lie(H)\sim\bas{t_1+t_2}$.
Moreover, the action of $H$
on $\P^1\times\P^1$ leaves a real fiber $L$ of $\pi_2$ invariant as a whole 
and leaves a real point $\hat{c}\in L$ invariant.
The point~$\hat{c}$ corresponds via $\rho$ to a point $c\in X$ \st $c\in X(\R)\subset S^n$.
We assume \Wlog that $c$ is the center of a stereographic projection $\pi:S^n\dto \R^n$.
Recall from \SEC{intro-defs} that $H$ induces a one-dimensional subgroup of the Euclidean similarities of~$\R^n$
that leaves $\pi(X(\R))$ invariant as a whole.
We call fibers of $\pi_1$ \df{horizontal} and fibers of $\pi_2$ \df{vertical},
since they correspond to horizontal and vertical line segments in \TAB{P1P1}, \resp.
A horizontal/vertical fiber that meets $\hat{c}$ correspond via $\rho$ and $\pi$ to a \df{horizontal/vertical line} in $\pi(X(\R))$.
The horizontal/vertical fibers that do not meet $\hat{c}$ correspond to \df{horizontal/vertical circles} in $\pi(X(\R))$.
Let $\fL\subset \pi(X(\R))$ be the vertical line corresponding to the vertical fiber $L$.
Thus $\fL$ is the stereographic projection of the set of real points in~$\rho(L)$
and the action of $H$ on $\pi(X(\R))$ leaves the line $\fL$ invariant as a whole.
% Recall that Lie($H$) is \Wlog complex inner automorphic to either
% $\bas{t_2}$ , $\bas{t_1+t_2}$ or $\bas{s_1,s_2}$.
The $H$-orbits of a general point on $L$ and a general point on a horizontal fiber 
corresponds to the $H$-orbits of points on $\fL$ and some horizontal circle, \resp.
The directions of such orbits in a small neighborhood are illustrated in  
\FIG{M}a if $\lie(H)\sim\bas{s_2}$ or $\lie(H)\sim\bas{t_2}$,
and in \FIG{M}b if $\lie(H)\sim\bas{t_1+t_2}$.
Suppose that $\varphi$ is a general Euclidean similarity in $H$.
Recall that an Euclidean similarity of $\R^n$ factors as a rotation, translation and/or scaling.
Notice that $\pi(X(\R))$ is not covered by lines and thus the scaling component of $\varphi$ is trivial.
It follows from \FIG{M}[a,b] that $\varphi$ has a nontrivial rotational component.
We arrived at a contradiction, since the line $\fL$ meets the horizontal circles 
and thus cannot be left invariant by the action of $H$.
We established that if $\Lambda=\emptyset$ as in \TAB{P1P1}a, then \LEM{M}.1 holds.
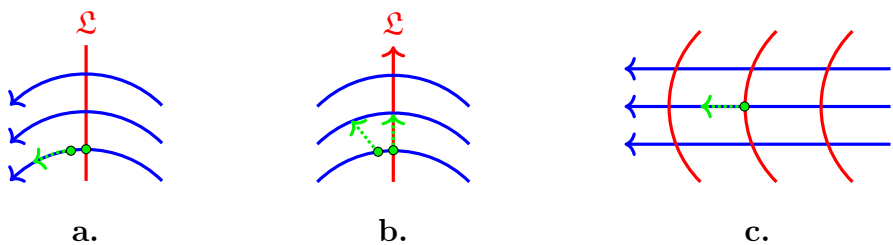
\begin{figure}[!ht]
\centering
\begin{tabular}{c@{\hspace{2cm}}c@{\hspace{2cm}}c}
\begin{tikzpicture}
%\draw[help lines] (-2,2) grid (2,-2);
% 
\draw[very thick, red] (0,-0.5) -- (0,1.3) node[above, red] {$\fL$};
\draw[very thick, blue,<-] (-1,0.5) to [out=45, in=135]  (1,0.5);
\draw[very thick, blue,<-] (-1,0) to [out=45, in=135]    (1,0);
\draw[very thick, blue,<-] (-1,-0.5) to [out=45, in=135] (1,-0.5);

\draw[very thick, black!5!green, densely dotted, ->] (-0.2,-0.1) to [out=180, in=20] (-0.7,-0.25);
\draw[draw=black, fill=black!5!green] (-0.2,-0.1) circle [radius=0.06];
\draw[draw=black, fill=black!5!green] (0,-0.08) circle [radius=0.06];

\end{tikzpicture}
&
\begin{tikzpicture}
%\draw[help lines] (-1,1) grid (1,-1);
% 
\draw[very thick, red, ->] (0,-0.5) -- (0,1.3) node[above, red] {$\fL$};
\draw[very thick, blue] (-1,0.5) to [out=45, in=135]  (1,0.5);
\draw[very thick, blue] (-1,0) to [out=45, in=135]    (1,0);
\draw[very thick, blue] (-1,-0.5) to [out=45, in=135] (1,-0.5);

\draw[very thick, black!5!green, densely dotted, ->] (-0.2,-0.1) to [out=90+45, in=-45] (-0.55,0.31);
\draw[very thick, black!5!green, densely dotted, ->] (0,-0.1) -- (0,0.4);
\draw[draw=black, fill=black!5!green] (-0.2,-0.1) circle [radius=0.06];
\draw[draw=black, fill=black!5!green] (0,-0.08) circle [radius=0.06];

\end{tikzpicture}
% &
% \begin{tikzpicture}
% %\draw[help lines] (-2,2) grid (2,-2);
% % 
% \draw[line width=2, red] (-1.2,0.8) to [out=180+45, in=135] (-1.2,-0.8);
% \draw[line width=2, red] (-0.2,0.55) to [out=180+45, in=135] (-0.2,-0.55);
% \draw[line width=2, red] (0.6,0.35) to [out=180+45, in=135] (0.6,-0.35);
% \draw[very thick, blue, <-] (-2,1)  -- (2,0);
% \draw[very thick, blue, <-] (-2,0)  -- (2,0);
% \draw[very thick, blue, <-] (-2,-1) -- (2,0);
% 
% 
% \draw[line width=1, black!5!green, densely dotted, ->] (-0.5,0) -- (-1,0);
% \draw[draw=black, fill=black!5!green] (-0.42,0) circle [radius=0.06];
% 
% \draw[line width=2, draw=red] (2,0) circle [radius=0.04] node[above, red] {$L_U$};
% 
% \end{tikzpicture}
&
\begin{tikzpicture}
%\draw[help lines] (-2,2) grid (2,-2);
% 
\draw[very thick, blue, <-] (-2,0.5)  -- (1.5, 0.5);
\draw[very thick, blue, <-] (-2,0)    -- (1.5, 0);
\draw[very thick, blue, <-] (-2,-0.5) -- (1.5,-0.5);
\draw[very thick, red] (-1,1) to [out=180+45, in=135] (-1,-1);
\draw[very thick, red] (0,1) to [out=180+45, in=135] (0,-1);
\draw[very thick, red] (1,1) to [out=180+45, in=135] (1,-1);

\draw[very thick, black!5!green, densely dotted, ->] (-0.5,0) -- (-1,0);
\draw[draw=black, fill=black!5!green] (-0.42,0) circle [radius=0.06];

\end{tikzpicture}
\\
\CAP{a}&
\CAP{b}&
\CAP{c}
\end{tabular}
\vspace{-4mm}
\caption{Euclidean similarities acting on a stereographic projection of $X(\R)$.
A dotted arrow depicts the direction of the orbit of the point at the tail.}
\label{fig:M}
\end{figure}

For the next case we suppose that $\Lambda$ is as in \TAB{P1P1}b or \TAB{P1P1}d. 
It follows from \PRP{blowup} that $\bM(X)\subseteq \PSO(2)\times \PSO(2)$
and since $\dim \bM(X)\geq 2$ we find as before that $\bM(X)\cong \PSO(2)\times \PSO(2)$.
We know from \PRP{L} and \TAB{L}[b,e] that $X$ has real structure $\sigma_2$.
Hence \LEM{M}.1 holds as well for these cases.

If $\Lambda$ is as in \TAB{P1P1}e or \TAB{P1P1}f, then 
$X\subset\S^3$ is either the spindle cyclide or the horn cyclide.
We showed in \EXM{spindle-horn} that $\bM(X)$ and $\bD(X)$
are as asserted in \LEM{M}.2 and \LEM{M}.3, \resp. 
The assertions for $\bT(X)$ and $\bS(X)$ follow from \PRP{blowup}.
We remark that the fiber
corresponding to $L$ as considered for the case $\Lambda=\emptyset$,
is in this case contracted to an isolated singularity $c\in X$ 
so that $\pi(X(\R))$ is covered by lines.

Finally, we suppose by contradiction that $\Lambda$ is as in \TAB{P1P1}c.

Let $L$ be the real vertical fiber of $\pi_2$ spanned by the complex conjugate points
$p$ and $\overline{p}$ as depicted in \TAB{P1P1}c.
We first show that there exists a subgroup $H\subset\bM(X)$ whose action on $\P^1\times\P^1$
leaves $L$ and the horizontal fibers invariant.

If $\lie(\bM(X))\sim_\C\bas{s_1,s_2}$, then 
the action of $\pi_2(\bM(X))$ on $\P^1$ leaves $\pi_2(p)$ and some other point $r\in\P^1$
invariant.
Thus in this case there exists a subgroup $H\subset \bM(X)$
whose action on $\P^1\times\P^1$ leaves the vertical fibers $L':=\pi_2^{-1}(r)$ and $L$ pointwise invariant,
and leaves each horizontal fiber invariant as a whole.
Now suppose that $\lie(\bM(X))\nsim_\C\bas{s_1,s_2}$.
It follows from \COR{sl2+sl2} that there exists a subgroup $H\subset \bM(X)$
\st $\lie(H)\sim_\C\bas{t_1+t_2}$
or 
$\lie(H)\sim_\C\bas{t_i}$ for $i\in\{1,2\}$.
Since automorphisms of $\P^1$ are 3-transitive and $|\pi_1(\Lambda)|=2$,
it follows that $\dim\pi_1(\bM(X))\leq 1$ and $\lie(H)\nsim_\C \bas{t_1}$. 
Since $\dim\bM(X)\geq 2$ by assumption, we find that $\dim\pi_2(\bM(X))\geq 1$.
Therefore there exists a subgroup $H\subset\bM(X)$ 
\st $\lie(H)\sim_\C \bas{t_2}$.
In this case, the action of $H$ on $\P^1\times\P^1$ 
leaves $L$ and the horizontal fibers invariant.

Since $|\Lambda\cap L|=2$ it follows that $\rho(L)$ is an isolated singularity of~$X$.
We assume \Wlog that this isolated singularity $c\in X(\R)$ 
is the center of stereographic projection $\pi\c S^n\to \R^n$. 
We use the same notation as before and find that, except for $L$, 
the horizontal fibers and vertical fibers correspond via $\rho$
to horizontal lines and vertical circles in $\pi(X(\R))\subset\R^n$, \resp.
We showed that there exists a subgroup $H\subset\bM(X)$
of Euclidean similarities whose action on $\pi(X(\R))$ 
leaves the horizontal lines invariant and sends vertical circles to vertical circles as in \FIG{M}c.
Thus the orbit of a point in a vertical circle is a horizontal line.
If we let the subgroup of scalings or translations act on the spanning plane of
a circle contained in $\pi(X(\R))$, then we obtain $\R^3$ so that $X\subset\S^3$.
We arrived at a contradiction as $\bT(X)$ is equal to $(2,6,m)$, where $m>3$ by \PRP{blowup}.

We concluded the proof, as we considered all cases for $\Lambda$ in \TAB{P1P1}.
\end{proof}

\newpage
\begin{lemma}
\textrm{\bf(rotational M\"obius automorphism group)}
\label{lem:M1}
\\
If $X\subset\S^n$ is a $\lambda$-circled celestial surface
\st $\lambda<\infty$ and \st $\bM(X)\cong \PSO(2)\times \PSO(2)$,
then \THM{M} holds for $X$.
\end{lemma}

\begin{proof}
Let $(\Ys,Q_c)$ denote the M\"obius pair of $X$,
where $Q_c$ is a hyperquadric of signature $(1,n+1)$.
The existence of this pair follows from \PRP{blowup} and
we denote the corresponding birational linear projection by $\rho\c\Ys\dto X$.
By \COR{L}, we may assume \Wlog that the real structure of $\Ys$
is defined by $\sigma_2$ in \LEM{invo-P8}.
We know from \PRP{moeb} that we may assume up to M\"obius equivalence
that $Q_c=V(q)$ for some invariant quadratic form $q\in I_2^G(\Ys)$, where $G$ is isomorphic to $\PSO(2)\times \PSO(2)$.
Thus it follows from \LEM{M} and \LEM{iqf}a that
\[
Q_c=\left\{ y\in\P^8 ~\Big\arrowvert~ \sum_{i\in\{1,3,5,7\}}c_i \left(y_0^2-y_iy_{i+1}\right)=0 \right\},
\]
for some coefficient vector $c=(c_1:c_3:c_5:c_7)\in\P^3$. 
The singular locus of $Q_c$ is defined by 
\[
\bS(Q_c)=\bigcap_{i\in I}\set{y\in\P^8}{ y_i=y_{i+1}=0 }
~~\text{with}~~
I:=\set{i\in\{1,3,5,7\}}{c_i\neq 0}.
\]
It follows from \LEM{invo-P8} that $\rho$ factors as
$\mu_t^{-1}\circ\rho_\ell$ and $\rho_r\circ\mu_2^{-1}$ 
as in the following commutative diagram
\[
\begin{tikzcd}
\Ys\arrow[d,"\rho_\ell"']\arrow[r,"\mu_2^{-1}"] & X_2\arrow[d,"\rho_r"] \\
Z\arrow[r,"\mu_t^{-1}"] & X
\end{tikzcd}
\]
Thus $\rho_\ell$ and $\rho_r$ are birational linear projections 
so that the real structures of $Z$ and $X$ are
induced by $\sigma_2\c\Ys\to\Ys$ and $\sigma_0\c X_2\to X_2$, \resp.    
The center of $\rho_\ell$ coincides with the singular locus of $Q_c$
by the definition of M\"obius pair and thus $\rho_\ell$ is a toric projection
(see \RMK{proj}). 
The vector space~$I_2^G(Z)$ is generated by the generators of $I_2^G(\Ys)$ 
that do not contain an element in $\{y_i\}_{i\in I}\cup\{y_{i+1}\}_{i\in I}$ as a variable.
We obtain the lattice type $\bL(Z)$ by 
taking the convex hull of the lattice polygon that is obtained by
removing the lattice points of the polygon in \TAB{L2}b
that are indexed by $\{y_i\}_{i\in I}\cup\{y_{i+1}\}_{i\in I}$ in \TAB{coord}.

We first want to determine the possible values for $\bT(X)$, $\bS(X)$, $\dim \P(I_2^G(X))$
and whether $\bM(X)$ is equal to $\autc(X)$.
We make a case distinction on $I\subset \{1,3,5,7\}$.
Notice that $|I|\leq 2$, otherwise the resulting lattice polygon is 1-dimensional.
\begin{Mlist}
 \item 
If $I=\emptyset$, then $\bT(X)=(2,8,7)$, $\bS(X)=\emptyset$, $\dim \P(I_2^G(X))=3$
and $\bM(X)\subsetneq \autc(X)$ as a direct consequence of the definitions.

\item
If $I\in\{\{1\},\{3\}\}$, then $\bL(Z)$ is as in \TAB{L2}b, $\bT(X)=(2,8,5)$
and $\bM(X)\subsetneq \autc(X)$.
Notice that if $I=\{3\}$, then the surface $Z$ 
is projectively isomorphic to the surface obtained with $I=\{1\}$.
If $I=\{1\}$, then, as discussed before, we omit the generators of $I_2^G(\Ys)$
that contain $y_1$ or $y_2$ as variable and find that
$I_2^G(Z)=\bas{ y_0^2-y_3y_4,~ y_0^2-y_5y_6,~ y_0^2-y_7y_8 }_\C$
so that $\dim \P(I_2^G(X))=2$.
We conclude from the monomial parametrization $\rho_\ell\circ \xi\c\T^2\to Z$
that $\bS(Z)=\emptyset$ and thus $\bS(X)=\emptyset$.

\item
If $I\in\{\{5\},\{7\}\}$, then $\bT(X)=(3,6,5)$ and $\bL(Z)$ is equivalent to \TAB{L}b.
It follows from \PRP{blowup} that $\bM(X)=\autc(X)$ and $\bS(X)=\emptyset$.
As before we verify that $\dim \P(I_2^G(X))=2$.

\item
If $I\in\{\{1,5\},\{1,7\},\{3,5\},\{3,7\},\{5,7\}\}$,
then $\bT(X)=(4,4,3)$ and $\bL(Z)$ is equivalent to \TAB{L}e.
It follows from \PRP{blowup} that $\bM(X)=\autc(X)$ and $\bS(X)=4A_1$.
We verify that $\dim \P(I_2^G(X))=1$ as before.

\item
If $I=\{1,3\}$, then the lattice points corresponding to $y_0$,
$y_5$, $y_6$, $y_7$ and $y_8$ in \TAB{coord},
correspond after the unimodular transformation
$(x,y)\mapsto(x-y,y+x)$
to a 2:1 monomial map $\xi_e(s^2,t^2)$
\st
the lattice type of the monomial parametrization $\xi_e(s,t)$
is as in \TAB{L}e.
Thus $\bL(Z)$ is equivalent to 
\TAB{L}e and we may assume \Wlog that $I=\{5,7\}$ which we already considered.
\end{Mlist}
We verify that $X$ is in all five cases biregular isomorphic to its linear normalization~$X_N$.
We know from \PRP{blowup} that $X_N$ is toric and thus $X$ is toric as well.

It remains to show that $\bD(X)=\dim \P(I_2^G(X))$.
It follows from \PRP{moeb}a that 
$(\Ys,Q_c)$ and $(\Ys,Q_{c'})$
correspond to M\"obius equivalent celestial surfaces 
\Iff 
there exists $\alpha\in \aut(\Ys)$ %\subset\aut(\P^8)$
\st $\alpha(Q_c)=Q_{c'}$.
Let $\varphi=(\varphi_1,\varphi_2)$ be an indeterminate element of $\autc(\P^1)\times\autc(\P^1)$.
Thus $\varphi_1$ and $\varphi_2$ are nonsingular $2{\times}2$-matrices in 
eight indeterminates $\vec{a}=(a_1,\ldots,a_8)$.
Recall from \EQN{S} that there exist a value for $\vec{a}$
\st $\alpha$ is defined by the $9{\times}9$-matrix $\cS(\varphi)$.
We compose, for all $i\in I$, the polynomials $y_0^2-y_iy_{i+1}$ with the 
map defined by $\cS(\varphi)$ so that we obtain quadratic polynomials 
in $y_i$ and coefficients in $\Q[a_1,\ldots, a_8]$.
Since $\alpha(Q_c)=Q_{c'}$, we require that coefficients of 
monomials, that are not of the form $y_0^2$ or $y_jy_{j+1}$ for some $j>0$, vanish.
We verify with a computer algebra system that the only possible value for $\vec{a}$ 
\st $\varphi_1$ and $\varphi_2$ have nonzero determinant, is when
$\vec{a}$ defines the identity automorphism.
Therefore
$(\Ys,Q_c)$ and $(\Ys,Q_{c'})$ are equivalent \Iff $c=c'$. 
We conclude that $\bD(X)=\dim \P(I_2^G(X))$ as was left to be shown.
\end{proof}

\begin{proof}[Proof of \THM{M}.]
Suppose that the celestial surface $X\subset\S^n$ is $\lambda$-circled.
If $\lambda<\infty$, then \THM{M} follows from \LEM{M} and \LEM{M1}.
If $\lambda=\infty$, then \THM{M} follows from 
\citep[Section~1]{kol2} and \citep[Section 2.4.3]{olm1};
we will give an alternative proof at \THM{vero}.
\end{proof}

\begin{proof}[Proof of \COR{iqf}.]
Our goal is as in \RMK{goal}, but with signature $(4,5)$ or $(3,6)$ instead of $(1,n+1)$.
Notice that everything in \SEC{pair} works if we replace $\S^n$ with 
a hyperquadric $Q$ of different signature.
% The double Segre surface $\Ys$ is by assumption 
% covered by conics that contain real points.
% It follows from \LEM{invo-P1P1}b that the real structure of~$\Ys$
% is defined by $\sigma_0$ in \LEM{invo-P8}.
It follows from \LEM{invo-P1P1}b that the real structure of $\P^1\times\P^1$
with real points is either $\sigma_+\times\sigma_+$
or $\sigma_s$.
These real structures are compatible with $\sigma_0\c\Ys\to\Ys$ and $\sigma_3\c\Ys\to\Ys$
in \LEM{invo-P8}, \resp.
This corollary is now a direct consequence of \PRP{moeb} and \LEM{iqf}d. 
We remark that if $Q$ has signature $(3,6)$, then
the unique double Segre surface in $Q$ is not covered by real conics.
\end{proof}

\section{The classification of \texorpdfstring{$\P^2$}{P2} revisited}
\label{sec:vero}

If $X\subseteq\S^n$ is $\infty$-circled, then 
$\bT(X)$ is either $(\infty,4,4)$ or $(\infty,2,2)$.
We know from \citep[Section~1]{kol2} and \citep[Section 2.4.3]{olm1}
that $\bM(X)$ is either $\PSO(3)$ or $\PSO(3,1)$.
Moreover, $X\subseteq\S^n$ is in both cases unique up to M\"obius equivalence. 
%that either $X=\S^2$ or $X$ is the Veronese surface and $\bT(X)=(\infty,4,4)$.
%For these cases \THM{M} is proven in \citep[Section~1]{kol2} and \citep[Section 2.4.3]{olm1}.
%Unlike embeddings of $\P^1\times\P^1$ into $\S^7$, 
%all Veronese embeddings of $\P^2$ into $\S^4$ are M\"obius equivalent.
% This indicates that the known proofs for the Veronese embedding do not extend for the other cases in \THM{M}.
We believe it might be instructive to give an alternative proof by using the methods of \SEC{pair}. 
We hope that this convinces the reader that our methods 
have the potential to be used outside the scope of this paper.

Suppose that $\Yc\subset\P^5$ is the Veronese surface with lattice type $\bL(\Yc)$
as in \TAB{L}d. 
Indeed, by \COR{L}b, we may assume \Wlog that 
the antiholomorphic involution acting on $\Yc$ is complex conjugation.
We proceed analogously as in \SEC{segre} with the coordinates 
in \TAB{coord} (right side) so that we obtain the following parametric map 
\[
\xi_d\c\T^2\to\Yc\subset\P^5,\quad (s,t)\mapsto(1:st:s:t:s^2:t^2)=(y_0:\ldots:y_5), 
\]
which extends to
$\tilde{\xi_d}\c\P^2\to\Yc\subset\P^5,\quad(s:t:u)\mapsto(u^2:st:su:tu:s^2:t^2)$.
Since $\Yc$ is isomorphic to $\P^2$ via $\tilde{\xi_d}$, we have $\autc(\Yc)\cong \PSL(3)$.
Using~$\xi_d$ we find the following 6 generators for the vector space of quadratic forms
on~$\Yc$ and it follows from \LEM{I2} that these form a basis:
\begin{gather*}
I_2(\Yc) = 
\langle 
y_1y_1 - y_4y_5,~
y_0y_1 - y_2y_3,~
y_2y_2 - y_0y_4,~
y_3y_3 - y_0y_5,~\\
y_1y_2 - y_3y_4,~
y_1y_3 - y_2y_5 
\rangle_\C.
\end{gather*}
Our goal is as in \RMK{goal} with $\Yc$ as $Y$.
Notice that the real structure of $\Yc$ acts as complex conjugation on the Lie algebra $\Msl_3$.
We consider the following elements in $\Msl_3$:
\begin{gather*}
a_1:=
\left[\begin{smallmatrix}
0 &  1 & 0\\
0 &  0 & 0\\  
0 &  0 & 0  
\end{smallmatrix}\right]
,~
a_2:=
\left[\begin{smallmatrix}
0 &  0 & 1\\
0 &  0 & 0\\  
0 &  0 & 0  
\end{smallmatrix}\right]
,~
a_3:=
\left[\begin{smallmatrix}
0 &  0 & 0\\
0 &  0 & 1\\  
0 &  0 & 0  
\end{smallmatrix}\right]
,~
c_1:=
\left[\begin{smallmatrix}
1 &  0 & 0\\
0 & -1 & 0\\  
0 &  0 & 0  
\end{smallmatrix}\right]
,
\\[2mm]
b_1:=
\left[\begin{smallmatrix}
0 &  0 & 0\\
1 &  0 & 0\\  
0 &  0 & 0  
\end{smallmatrix}\right]
,~
b_2:=
\left[\begin{smallmatrix}
0 &  0 & 0\\
0 &  0 & 0\\  
1 &  0 & 0  
\end{smallmatrix}\right]
,~
b_3:=
\left[\begin{smallmatrix}
0 &  0 & 0\\
0 &  0 & 0\\  
0 &  1 & 0  
\end{smallmatrix}\right]
,~
c_2:=
\left[\begin{smallmatrix}
0 &  0 &0\\
0 &  1 &0\\  
0 &  0 &-1  
\end{smallmatrix}\right].
\end{gather*}

\begin{lemma}
\textrm{\bf(invariant quadratic forms for $\P^2$)}
\label{lem:iqf-vero}
\\
Suppose that 
$G\subseteq \autc(\Yc)$ is a Lie subgroup,
where $\Yc\subset\P^5$ is the Veronese surface.
If $\lie(G)=\bas{b_1-a_1,b_2-a_2,b_3-a_3}$, then $G\cong \PSO(3)$ and 
\[
I^G_2(\Yc)
=
\bas{x_1^2 + x_2^2 + x_3^2 - x_0x_4 - x_0x_5 - x_4x_5}_\C.
\]
If $\lie(G)=\bas{a_1,a_2,a_3,b_1,b_2,b_3,c_1,c_2}$, then $G\cong \PSL(3)$ 
and $I^G_2(\Yc)=\bas{1}_\C$.
\end{lemma}

\begin{proof}
The subgroup $\PSO(3)\subset \PSL(3)$ is 
generated by the three $3{\times}3$ rotation matrices and 
thus $\Mso_3=\bas{b_1-a_1,b_2-a_2,b_3-a_3}$. 
The generators for the Lie algebra $\Msl_3$ can be found for example in \citep[Section 6.2]{hab1}.
For the remaining assertions we apply \THM{D} as in the proof of \LEM{iqf}. 
\end{proof}

\begin{lemma}
\label{lem:vero-pair}
If $(\Yc,Q)$ and $(\Yc,Q')$ are M\"obius pairs 
of celestial surfaces in~$\S^n$ for $n\geq 3$, 
then these pairs are equivalent. 
\end{lemma}

\begin{proof}
We consider the following group actions
\[
\cA\c\PSL(3)\times \P^2\to\P^2
\quad\text{and}\quad
\cB\c\PSL(3)\times \Yc\to\Yc.
\]
As in \EQN{S}, the group action $\cB$ is defined via 
$\sym_2(\cdot)$ and can be computed via the isomorphism 
$\tilde{\xi_d}\c\P^2\to\Yc\subset\P^5$.
These group actions induce group actions on the spaces of quadratic forms
$V:=\P(I_2(\P^2))$ and $W:=\P(I_2(\Yc))$:
\[
\cA_\star\c\PSL(3)\times V\to V
\quad\text{and}\quad
\cB_\star\c\PSL(3)\times W\to W.
\]
Recall that a quadratic form in $V$ is equivalent via $\cA_\star$ to a diagonal form
of signature $(1,0)$, $(2,0)$, $(1,1)$, $(3,0)$ or $(2,1)$.
It is left to the reader to verify that $W$ contains quadratic forms
of signatures $(1, 2)$, $(1, 3)$, $(1, 5)$, $(2, 2)$ and $(3, 3)$.
We can also define a group action
$\cC_\star\c\PSL(3)\times W\to W$
via the action $\cA_\star$ and an isomorphism $V\to W$.
Irreducible representations $\PSL(3)\to\aut(\P^5)$ are isomorphic
and thus $\cB_\star$ and $\cC_\star$ must be isomorphic group actions.
Hence we can match the orbits of $\cA_\star$ with the orbits of $\cB_\star$
and thus we identified all possible signatures of quadratic forms in $W$.
Thus $Q=V(q)$ and $Q'=V(q')$, where the quadratic forms $q$ and $q'$ in $W$ have both signature $(1,5)$. 
The group action $\cA_\star$, and thus also the group action $\cB_\star$, acts transitively on quadratic forms of the same signature.
It follows that $q'=q\circ \varphi^{-1}$ for some $\varphi\in \autc(\Yc)$.
Therefore $\varphi(Q)=Q'$ so
that $(\Yc,Q)$ and $(\Yc,Q')$ are equivalent as M\"obius pairs.
\end{proof}

The following theorem is a consequence of 
\citep[Theorem~23]{kol2}. We give an alternative proof
by applying the methods in \SEC{pair}.

\begin{theoremext}
{\bf [Koll\'ar, 2018]}
\label{thm:vero}
\\
If $X\subset \S^n$ is an $\infty$-circled celestial surface,
then \THM{M} holds for $X$.
\end{theoremext}

\begin{proof}
If $n\leq 2$, then $\bT(X)=(\infty,2,2)$, $\bS(X)=\emptyset$, 
$\bM(X)\cong\PSO(3,1)$ and $\bD(X)=0$ so that \THM{M} holds.
If $n>2$, then we know from \THM{circle}a that $\bT(X)=(\infty,4,4)$, $\bS(X)=\emptyset$
and $X$ is projectively equivalent to the Veronese surface $\Yc$.
By assumption there exists a subgroup $G\subseteq \bM(X)$ \st $\dim G\geq 2$.
We know from \PRP{moeb}b that $X$ has M\"obius pair $(\Yc,V(q))$ for 
some invariant quadratic form $q\in I_2^G(\Yc)$ of signature $(1,5)$.
It follows from \LEM{iqf-vero} and \PRP{moeb}c that $G\ncong\PSL(3)$
and that if $G\cong\PSO(3)$, then $q\circ\varphi=x_1^2 + x_2^2 + x_3^2 - x_0x_4 - x_0x_5 - x_4x_5$
for some $\varphi \in \aut(\Yc)$.
It follows from \LEM{vero-pair} and \PRP{moeb}a that
$X$ is unique up to M\"obius equivalence. 
Therefore $\PSO(3)\subseteq \bM(X)$ and $\bD(X)=0$.
There exists no subgroup $G$ \st $\PSO(3)\subsetneq G \subsetneq \PSL(3)$
and thus we conclude that $\bM(X)\cong \PSO(3)$.
\end{proof}

\section{Acknowledgements}

I thank 
M.\ Gallet, W.\ de Graaf, J.\ Koll\'ar, H.\ Pottmann, J.\ Schicho and M.~Skopenkov 
for interesting comments and discussions.
The constructive methods are implemented using \citep[Sage]{sage}.
The images in \FIG{dP6} are rendered with \citep[Povray]{povray}.
Our implementations are open source and can be found at \citep[\texttt{moebius\_aut}]{maut}.
Financial support was provided by the Austrian Science Fund (FWF): P33003.

\section{References}
\bibliography{moebius-aut}

\paragraph{address of author:}~
Johann Radon Institute for Computational and Applied Mathematics (RICAM), Austrian Academy of Sciences
\\
\textbf{email:} niels.lubbes@gmail.com

%\printindex
\end{document}